\documentclass[review,onefignum,onetabnum]{siamart171218}

\usepackage{a4wide}
\usepackage{amsfonts}
\usepackage{amscd}
\usepackage{amssymb}
\usepackage{amsmath}
\usepackage{latexsym}
\usepackage{dsfont}
\usepackage{placeins}
\usepackage{mathrsfs}

\usepackage{nicefrac}
\usepackage{MnSymbol}
\usepackage[normalem]{ulem}
\usepackage{pgfplots}
\pgfplotsset{compat=newest}
\usepackage{booktabs}

\usepackage{mathtools}
\usepackage{esvect} 
\usepackage{blindtext}

\usepackage{graphicx}
\usepackage{subfigure}
\usepackage{epstopdf}
\newtheorem{innercustomthm}{Theorem}
\newenvironment{customthm}[1]
  {\renewcommand\theinnercustomthm{#1}\innercustomthm}
  {\endinnercustomthm}

\theoremstyle{plain}
\newtheorem{rem}[theorem]{Remark}
\newtheorem{assumption}[theorem]{Assumption}

{%
\setcounter{enumi}{0}
\renewcommand{\theenumi}{(\textbf{A}.\arabic{enumi})}

\begin{enumerate}}%
{\end{enumerate} }
{%
\setcounter{enumi}{0}
\renewcommand{\theenumi}{(\textbf{D}.\arabic{enumi})}

\begin{enumerate}}%
{\end{enumerate} }

\newcommand{\hione}{{(1)}}
\newcommand{\hitwo}{{(2)}}
\newcommand{\apprlevy}{{(\varepsilon_l)}}
\newcommand{\apprgrf}{{(\varepsilon_W)}}

\newcommand{\cD}{\mathcal{D}}

\makeatletter
\usepackage{color}

\makeatother

\newcommand{\be}{\begin {equation}}
\newcommand{\ee}{\end  {equation}}
\numberwithin{equation}{section} \allowdisplaybreaks[1]

\definecolor{darkgreen}{rgb}{0,.6,0}

\newcommand{\bee}{\begin {equation*}}
\newcommand{\eee}{\end {equation*}}

\nolinenumbers

\title{Subordinated Gaussian Random Fields in elliptic Partial differential equations}

\author{Andrea Barth \thanks{IANS\textbackslash SimTech, University of Stuttgart 
(\email{andrea.barth@mathematik.uni-stuttgart.de})},
\and Robin Merkle \thanks{IANS\textbackslash SimTech, University of Stuttgart 
	(\email{robin.merkle@mathematik.uni-stuttgart.de}).}}

\headers{SUBORDINATED GAUSSIAN RANDOM FIELDS IN ELLIPTIC PDEs}{A. Barth and R. Merkle}

\begin{document}
		
	\maketitle
		
	\begin{abstract}
	To model subsurface flow in uncertain heterogeneous\textbackslash ~fractured media an elliptic equation with a discontinuous stochastic diffusion coefficient --- also called random field --- may be used. 
    In case of a one-dimensional parameter space, L\'evy processes allow for jumps and display great flexibility in the distributions used. However, in various situations (e.g. microstructure modeling), a one-dimensional parameter space is not sufficient. Classical extensions of L\'evy processes on two parameter dimensions suffer from the fact that they do not allow for spatial discontinuities (see for example \cite{ApproximationAndSimulation}). In this paper a new subordination approach is employed (see also \cite{SubordGRFTheory}) to generate L\'evy-type discontinuous random fields on a two-dimensional spatial parameter domain. Existence and uniqueness of a (pathwise) solution to a general elliptic partial differential equation is proved and an approximation theory for the diffusion coefficient and the corresponding solution provided. Further, numerical examples using a Monte Carlo approach on a Finite Element discretization validate our theoretical results.
	\end{abstract}
	


\section{Introduction}\label{sec:Intro}
Over the last decade partial differential equations with stochastic operators\textbackslash~data\textbackslash~domain became a widely studied object. This branch of research is oftentimes called uncertainty quantification. Especially for problems where data is sparse or measurement errors are unavoidable, like subsurface flow problems, the theory provides an approach to quantify this uncertainty. There are two main approaches to discretize the uncertain problem: intrusive and non-intrusive methods. The former require the solution of a high dimensional partial differential equation, where the dimensionality depends on the smoothness of the random field or process (see for example ~\cite{GalerkiNFEApprOfStochEllPDEs}, \cite{FEsForEllProblemsWithStochCoeff}, \cite{OnTheConvOfTheStochGalerkinMethodForRandomEllPDEs} and the references therein). The latter consist of (essentially) sampling methods and require multiple solutions of a low dimensional problem (see, among others,~
\cite{MLMCMethodsForStochEllMultiscalePDEs},
\cite{MLMCFEMethodForEllPDEsWithStochCoeff},
\cite{AStudyOfElliptic}, \cite{AMultilevelMonteCarloAlgorithmForParabolicAdvectionDiffusionProblemsWithDiscontinuousCoefficients},  \cite{MultiLevelMonteCarloWeakGalerkinMethodForEllipticEquationsWithStochasticJumpCoefficients}, \cite{FurtherAnalysisOfMultilevelMonteCarloMethodsForEllipticPDEsWithRandomCoefficients}).
Up to date mainly Gaussian random fields were used to model the diffusivity in an elliptic equation (as a model for a subsurface flow problem). Gaussian random fields have the advantage that they may be used in both approaches and that they are stochastically very well understood objects. A great disadvantage is however, that the distributions underlying the field are Gaussian and therefore the field lack flexibility, in the sense that the field is continuous and cannot have pointwise marginal distributions with heavy-tails.

In this paper we propose a two-dimensional subordinated Gaussian random field as stochastic diffusion coefficient in an elliptic equation. The subordinated Gaussian random field is a type of a (discontinuous) L\'evy field. Different subordinators display unique patterns in the discontinuities and have varied marginal distributions (see \cite{SubordGRFTheory}). Naturally the spatial regularity of a subordinated Gaussian random field depends on the subordinator. We prove existence and uniqueness of a solution to the elliptic equation in a pathwise sense and provide different discretization schemes. 

We structured the rest of the paper as follows: In Section~\ref{sec:elliptic_problem} we introduce a general pathwise existence and uniqueness result for a stochastic elliptic equation under mild assumptions on the coefficient. These assumptions accommodate the subordinated Gaussian random fields we introduce in Section~\ref{sec:subordinated_GRF}. In Section~\ref{sec:approx_GRF} we approximate the specific diffusion coefficient which is used in this paper and show convergence of the elliptic equation with the approximated coefficient to the unapproximated solution in Section~\ref{sec:convergence}. Section~\ref{sec:approx_solution} provides spatial approximation methods and in Section~\ref{sec:numerics} numerical examples are presented.

\section{The stochastic elliptic problem}\label{sec:elliptic_problem}
	In this section we introduce the framework of the general stochastic elliptic boundary value problem which allows for discontinuous diffusion coefficients. For the general setting and pathwise existence theory we follow~\cite{AStudyOfElliptic}. In the following, let $(\Omega,\mathcal{F},\mathbb{P})$ be a complete probability space. To accommodate Banach-space-valued random variables we introduce the so-called Bochner spaces.

\begin{definition}
	Let $(B,\|\cdot\|_B)$ be a Banach space and $Z$ be a $B$ valued random variable, i.e. a strongly measurable function $Z:\Omega\rightarrow B$. The space $L^p(\Omega;B)$ contains all $B$-valued random variables with $\|Z\|_{L^p(\Omega;B)}<+\infty$, for $p\in [1,+\infty)$, where the norm is defined by 
	\begin{align*}
	\|Z\|_{L^p(\Omega;B)} = \begin{cases}\mathbb{E}(\|Z\|_B^p)^\frac{1}{p}&,\text{ if } 1\leq p<+\infty \\ \underset{\omega\in\Omega}{\operatorname{ess}\,\sup} \|Z\|_B &,\text{ if } p=+\infty \end{cases}.
	\end{align*}
\end{definition}
	\subsection{Problem formulation}
	  Let $\mathcal{D}\subset \mathbb{R}^d$ for $d\in\mathbb{N}$ be a bounded, connected Lipschitz domain. We consider the equation
	\begin{align}\label{EQ:EllProblem}
	-\nabla(a(\omega,\underline{x})\nabla u(\omega,\underline{x}))=f(\omega,\underline{x}) \text{ in }\Omega\times\mathcal{D},
	\end{align}
	where $a:\Omega\times\mathcal{D}\rightarrow\mathbb{R}_+$ is a stochastic (jump diffusion) coefficient and $f:\Omega\times\mathcal{D}\rightarrow\mathbb{R}$ is a (measurable) random source function. Further, we impose the following boundary conditions
	\begin{align}
	u(\omega,\underline{x})&=0 \text{ on } \Omega\times \Gamma_1,\label{EQ:EllProblemBCD}\\
	a(\omega,\underline{x}) \overrightarrow{n}\cdot\nabla u(\omega,\underline{x})&=g(\omega,\underline{x}) \text{ on } \Omega\times \Gamma_2,\label{EQ:EllProblemBCN}
	\end{align}
	where we assume to have a decomposition $\partial\mathcal{D}=\Gamma_1\overset{.}{\cup}\Gamma_2$ with two $(d-1)$-dimensional manifolds $\Gamma_1,~\Gamma_2$ such that the exterior normal derivative $\overrightarrow{n}\cdot\nabla u$ on $\Gamma_2$ is well-defined for every $u\in C^1(\overline{\mathcal{D}})$. Here, $\overrightarrow{n}$ is the outward unit normal vector to $\Gamma_2$ and $g:\Omega\times\Gamma_2\rightarrow\mathbb{R}$ a measurable function.
	Note that we just reduce the theoretical analysis to the case of homogeneous Dirichlet boundary conditions to simplify notation. It would be also possible to work under non-homogeneous Dirichlet boundary conditions, since such a problem can always be considered as a version of \eqref{EQ:EllProblem} - \eqref{EQ:EllProblemBCN} where the source term and the Neumann data have been changed (see also \cite[Remark 2.1]{AStudyOfElliptic}).
	
	We now state assumptions under which the elliptic boundary value problem has a unique solution.
	\begin{assumption}\label{ASS:ProblemAssumptionsGeneral}
	Let $H:=L^2(\mathcal{D})$. We assume that for all $\omega\in \Omega$ it holds that 
	\begin{enumerate}
	\item for any fixed $\underline{x}\in\mathcal{D}$ the mapping $\omega\mapsto a(\omega,\underline{x})$ is measurable, i.e. $a(\cdot,\underline{x})$ is a (real-valued) random variable.
	\item for any fixed $\omega\in\Omega$ the mapping $a(\omega,\cdot)$ is $\mathcal{B}(\mathcal{D})-\mathcal{B}(\mathbb{R}_+)$-measurable and it holds $a_{-}(\omega):=\underset{\underline{x}\in\mathcal{D}}{\operatorname{ess}\,\inf}\,a(\omega,\underline{x})>0$ and $a_+(\omega):=\underset{\underline{x}\in\mathcal{D}}{\operatorname{ess}\,\sup}\,a(\omega,\underline{x})<+\infty$,
	\item $\frac{1}{a_{-}}\in L^p(\Omega;\mathbb{R})$, $f\in L^q(\Omega;H)$ and $g\in L^q(\Omega;L^2(\Gamma_2))$ for some $p,q\in [1,+\infty]$ such that $r:=(\frac{1}{p} + \frac{1}{q})^{-1}\geq 1$. 
	\end{enumerate}
	\end{assumption}
	We identify $H$ by its dual space $H'$ and work on the Gelfand triplet $V\subset H\simeq H'\subset V'$. Hence, Assumption 
\ref{ASS:ProblemAssumptionsGeneral} guarantees that $f(\omega,\cdot)\in V'$ and $g(\omega,\cdot)\in H	^ {-\frac{1}{2}}(\Gamma_2)$ for $\mathbb{P}$-almost every $\omega\in \Omega$.
\begin{rem}
Note that Assumption \ref{ASS:ProblemAssumptionsGeneral} implies that the real-valued mappings $a_-,a_+:\Omega\rightarrow\mathbb{R}$ are measurable. This can be seen as follows:
For fixed $p\geq 1$ consider the mapping
\begin{align*}
I_p:\Omega &\rightarrow \mathbb{R},\\
\omega &\mapsto \|a(\omega,\cdot)\|_{L^p(\mathcal{D})} = (\int_{\mathcal{D}} a(\omega,\underline{x})^p d\underline{x}) )^{1/p},
\end{align*}
which is well-defined by Assumption~\ref{ASS:ProblemAssumptionsGeneral}. It follows from the definition of the Lebesgue integral and Assumption \ref{ASS:ProblemAssumptionsGeneral} \textit{i} that the mapping $\omega\mapsto I_p(\omega)$ is $\mathcal{F}-\mathcal{B}(\mathbb{R})$ measurable. For a fixed $\omega \in \Omega$, by the embedding theorem for $L^p$ spaces (see \cite[Theorem 2.14]{SobolevSpaces}), we get
\begin{align*}
a_+(\omega)=  \lim_{m\rightarrow\infty}\|a(\omega,\cdot)\|_{L^m(\mathcal{D})}.
\end{align*} 
Since this holds for all $\omega \in \Omega$ we obtain by \cite[Lemma 4.29]{InfiniteDimensionalAnalysis} that the mapping 
\begin{align*}
\omega \mapsto a_+(\omega),
\end{align*}
is $\mathcal{F}-\mathcal{B}(\mathbb{R})$-measurable. The measurability of $\omega\mapsto a_-(\omega)$ follows analogously.
Note that we do not treat the random coefficient $a:\Omega\times \mathcal{D}\rightarrow \mathbb{R}$ as a $L^\infty(\mathcal{D})$-valued random variable, since $L^\infty(\mathcal{D})$ is not separable and therefore the strong measurability of the mapping $a:\Omega\rightarrow L^\infty(\mathcal{D})$ is only guaranteed in a very restrictive setting. Nevertheless, the measurability of the functions $a_+$ and $a_-$ allows taking expectations of these real-valued random variables. In order to avoid confusion about that, we use the notation $\mathbb{E}(\operatorname{ess\, sup}\limits_{\underline{x}\in \mathcal{D}} |\cdot|^s)^{1/s}$.
\end{rem}

\subsection{Weak solution}
We denote by $H^1(\mathcal{D})$ the Sobolev space on $\mathcal{D}$ with the norm 
\begin{align*}
\|v\|_{H^1(\mathcal{D})}=\left(\int_D|v(\underline{x})|^2 + \|\nabla v(\underline{x})\|_2^2d\underline{x}\right)^2 \text{, for } v\in H^1(\mathcal{D}),
\end{align*}
(see for example \cite[Section 5.2]{PartialDifferentialEquations}). Here, $\|\underline{x}\|_2:=(\sum_{i=1}^d \underline{x}_i^2)^\frac{1}{2}$ denotes the Euclidean norm of the vector $\underline{x}\in \mathbb{R}^d$. Further, we denote by $T$ the trace operator with
\begin{align*}
T:H^1(\mathcal{D})\rightarrow H^{\frac{1}{2}}(\partial \mathcal{D})
 \end{align*}
where $Tv=v|_{\partial \mathcal{D}}$ for $v\in C^\infty (\overline{\mathcal{D}})$ (see \cite{TraceTheoremLipDomain}).
We define the subspace $V\subset H^1(\mathcal{D})$ as
\begin{align*}
V:=\{v\in H^1(\mathcal{D})~|~Tv|_{\Gamma_1}=0\},
\end{align*}
with the standard Sobolev norm, i.e. $\|\cdot\|_V:=\|\cdot\|_{H^1(\mathcal{D})}$.

We multiply Equation \eqref{EQ:EllProblem} by a test function $v\in V$, integrate by parts and use the boundary conditions \eqref{EQ:EllProblemBCD} and \eqref{EQ:EllProblemBCN} to obtain
\begin{align*}
\int_\mathcal{D}-\nabla\cdot(a(\omega,\underline{x})\nabla u(\omega,\underline{x}))v(\underline{x}) d\underline{x} =\int_\mathcal{D}a(\omega,\underline{x})\nabla u (\omega,\underline{x})\cdot\nabla v(\underline{x}) - \int_{\Gamma_2}g(\omega,\underline{x})[Tv](\underline{x})d\underline{x}.
\end{align*}
This leads to the following pathwise weak formulation of the problem: For any $\omega\in\Omega$, given $f(\omega,\cdot)\in V'$ and $g(\omega,\cdot)\in H^{-\frac{1}{2}}(\Gamma_2)$, find $u(\omega,\cdot)\in V$ such that 
\begin{align}\label{EQ:WeakFormProblem}
B_{a(\omega)}(u(\omega,\cdot),v) = F_\omega(v)
\end{align}
for all $v\in V$. The function $u(\omega,\cdot)$ is then called \emph{pathwise weak solution} to problem \eqref{EQ:EllProblem} - \eqref{EQ:EllProblemBCN}. Here, the bilinear form $B_{a(\omega)}$ and the operator $F_\omega$ are given by
\begin{align*}
B_{a(\omega)}:V\times V\rightarrow \mathbb{R}, ~(u,v)\mapsto \int_{\mathcal{D}}a(\omega,\underline{x})\nabla u(\underline{x})\cdot \nabla v(\underline{x})d\underline{x},
\end{align*}
and
\begin{align*}
F_\omega:V\rightarrow\mathbb{R}, ~v\mapsto \int_\mathcal{D}f(\omega,x)v(\underline{x})d\underline{x} + \int_{\Gamma_2}g(\omega,\underline{x})[Tv](\underline{x})d\underline{x},
\end{align*}	
for fixed $\omega\in \Omega$, where the integrals in $F_\omega$ are understood as the duality pairings:
\begin{align*}
\int_\mathcal{D}f(\omega,\underline{x})v(\underline{x})d\underline{x} = \prescript{}{V'}{\langle}f(\omega,\cdot),v\rangle_V
\end{align*}
and
\begin{align*}
\int_{\Gamma_2}g(\omega,\underline{x})[Tv](\underline{x})d\underline{x} = \prescript{}{H^{-\frac{1}{2}}(\Gamma_2)}{\langle	}g(\omega,\cdot),Tv\rangle_{H^\frac{1}{2}(\Gamma_2)},
\end{align*}
for $v\in V$.

\begin{theorem}\label{TH:ExistenceTheoremElliptic}(see \cite[Theorem 2.5]{AStudyOfElliptic})
Under Assumption \ref{ASS:ProblemAssumptionsGeneral}, there exists a unique pathwise weak solution $u(\omega,\cdot)\in V$ to problem \eqref{EQ:WeakFormProblem} for very $\omega\in \Omega$. Furthermore, $u\in L^r(\Omega;V)$ and 
\begin{align*}
\|u\|_{L^r(\Omega;V)}\leq C(a_-,\mathcal{D}, p)(\|f\|_{L^q(\Omega;H)} + \|g\|_{L^q(\Omega;L^2(\Gamma_2))}),
\end{align*}
where $C(a_-,\mathcal{D},p)>0$ is a constant depending on $a_-$, $p$ and the volume of $\mathcal{D}$.
\end{theorem}	

In addition to the existence of the solution, the following remark gives a rigorous justification for the measurability of the solution mapping 
\begin{align*}
u:\Omega&\rightarrow V\\
\omega &\mapsto u(\omega,\cdot),
\end{align*}
which maps any $\omega \in \Omega$ on the corresponding pathwise weak PDE solution. 

\begin{rem}\label{Rem:Measurability of the solution}
Let $(v_n,~n\in\mathbb{N})\subset V$ be an orthonormal basis of the separable Hilbert space $V$. For every $n\in \mathbb{N}$ we define the mapping
\begin{align*}
J_n:\Omega \times V &\rightarrow \mathbb{R}\\
(\omega,v)&\mapsto \int_\mathcal{D} a(\omega,\underline{x}) \nabla v(\underline{x})\cdot \nabla v_n(\underline{x})d\underline{x} - \int_\mathcal{D}f(\omega,\underline{x})v_n(\underline{x})d\underline{x} - \int_\mathcal{\Gamma_2}g(\omega,\underline{x})[Tv_n](\underline{x})d\underline{x}.
\end{align*}
It is easy to see that this mapping is Carathéodory for any $n\in\mathbb{N}$, i.e. $J_n(\cdot,v)$ is $\mathcal{F}$-$\mathcal{B}(\mathbb{R})$-measurable for any fixed $v\in V$ and $J_n(\omega,\cdot)$ is continuous on $V$ for any fixed $\omega\in\Omega$. We define the correspondences
\begin{align*}
\phi_n(\omega):=\{v\in V~|~J_n(\omega,v)=0\},
\end{align*}
for every $n\in \mathbb{N}$. It follows from \cite[Corollary 18.8]{InfiniteDimensionalAnalysis} that this correspondence has a measurable graph, i.e.
\begin{align*}
\{(\omega,v)\in \Omega\times V~|~v\in\phi_n(\omega)\}\in\mathcal{F}\otimes \mathcal{B}(V).
\end{align*}
Further, by Assumption \ref{ASS:ProblemAssumptionsGeneral} and the Lax-Milgram Lemma (see for example \cite[Lemma 6.97]{EllipticDifferentialEquations} and \cite[Theorem 2.5]{AStudyOfElliptic}) we know that for every fixed $\omega\in\Omega$, there exists a unique solution $u(\omega,\cdot)\in V$ satisfying \eqref{EQ:WeakFormProblem} for every $v\in V$. Therefore, we obtain for the graph of the solution mapping:
\begin{align*}
 \{(\omega,u(\omega,\cdot))~|~\omega\in\Omega\} &= \{(\omega,v)\in \Omega\times V~|~J_n(\omega,v)=0,\text{ for all }n\in\mathbb{N}\} \\
 &=\bigcap_{n\in\mathbb{N}} \{(\omega,v)\in \Omega\times V~|~ v\in\phi_n	(\omega) \} \in \mathcal{F}\otimes \mathcal{B}(V). 
\end{align*}
This implies for an arbitrary measurable set $\tilde{V}\in \mathcal{B}(V)$ 
\begin{align*}
\{(\omega,u(\omega,\cdot))~|~\omega\in\Omega\} \cap \Omega\times \tilde{V} \in \mathcal{F}\otimes \mathcal{B}(V)
\end{align*}
and therefore 
\begin{align*}
\{\omega\in\Omega~|~u(\omega,\cdot) \in \tilde{V}\}\in\mathcal{F},
\end{align*}
by the projection theorem (see \cite[Theorem 18.25]{InfiniteDimensionalAnalysis}), which gives the measurability of the solution mapping.

\end{rem}

\section{Subordinated Gaussian random fields}\label{sec:subordinated_GRF}

A random field $W:\Omega\times \mathcal{D}	\rightarrow \mathbb{R}$ is called $\mathbb{R}-valued$ \textit{Gaussian random field (GRF)} if for any tuple $(\underline{x}_1,\dots,\underline{x}_n)\subset \mathcal{D}$ and any number $n\in \mathbb{N}$ the $\mathbb{R}^n$-valued random variable
\begin{align*}
[W(\underline{x}_1),\dots,W(\underline{x}_n)]^T:\Omega\rightarrow  \mathbb{R}^n
\end{align*}
is multivariate normally distributed (see \cite[Section 1.2]{RandomFieldsAndGeometry}). Here $\underline{x}^T$ denotes the transpose of the vector $\underline{x}$.
We denote by 
\begin{align*}
m(\underline{x})&:=\mathbb{E}(W(\underline{x})),~\underline{x}\in\mathcal{D}, \text{ and }\\
q(\underline{x},\underline{y})&:=Cov(W(\underline{x}),W(\underline{y})),~\underline{x},\underline{y}\in\mathcal{D},
\end{align*}
the associated mean and covariance function. The \textit{covariance operator} $Q:L^2(\mathcal{D})\rightarrow L^2(\mathcal{D})$ of $W$ is defined by 
\begin{align*}
Q(\psi)(\underline{x})=\int_\mathcal{D}q(\underline{x},\underline{y})\psi(\underline{y})d\underline{y} \text{ for } \underline{x}\in \mathcal{D}.
\end{align*}
Further, if $\mathcal{D}\subset \mathbb{R}^d$ is compact and $W$ is centered, i.e. $m\equiv 0$, there exists a decreasing sequence $(\lambda_i,~i\in \mathbb{N})$ of real eigenvalues of $Q$ with corresponding eigenfunctions $(e_i,~i\in\mathbb{N})\subset L^2(\mathcal{D})$ which form an orthonormal basis of $L^2(\mathcal{D})$ (see \cite[Section 3.2]{RandomFieldsAndGeometry} and \cite[Theorem VI.3.2 and Chapter II.3]{Funktionalanalysis}).\\

\subsection{Construction of subordinated GRFs}\label{Subsec:AShortIntroductionTo}

A real-valued stochastic process $l=(l(t),~t\geq 0)$ is said to be a \textit{L\'evy process} if $l(0)=0$ $\mathbb{P}$-a.s., $l$ has independent and stationary increments and $l$ is stochastically continuous (see \cite[Section 1.3]{LevyProcessesAndStochasticCalculus}).
One of the most important properties of L\'evy processes is the so called \textit{L\'evy-Khinchin formula}. 
\begin{theorem}\label{TH:LevyKhinchine}(L\'evy-Khinchin formula, see \cite[Th. 1.3.3]{LevyProcessesAndStochasticCalculus})\\
Let $l$ be a real-valued L\'evy process on $\subset \mathbb{R}_+:=[0,+\infty)$. There exist constants, $\gamma_l \in \mathbb{R}$, $\sigma_l^2\in \mathbb{R}_+$ and a measure $\nu$ on $(\mathcal{T},\mathcal{B}(\mathcal{T}))$ such that the characteristic function $\phi_{l(t)}$, for $t\in \mathbb{R}_+$, admits the representation 
\begin{align*}
\phi_{l(t)}(u) := \mathbb{E}(\exp(iul(t))) = \exp\left(t\left(i\gamma_l u- \frac{\sigma_l^2}{2}u^2 + \int_{\mathbb{R}\setminus \{0\}} e^{iuy} - 1 - iuy\mathds{1}_{\{|y|\leq 1\}}\nu(dy)\right)\right).
\end{align*}

\end{theorem}
Motivated by Theorem \ref{TH:LevyKhinchine} we denote by $(\gamma_l,\sigma_l^2,\nu)$ the characteristic triplet of the L\'evy process $l$.
A (L\'evy-)subordinator is a L\'evy process which is \textit{non-decreasing} $\mathbb{P}$-a.s.. By \cite[Theorem 1.3.15]{LevyProcessesAndStochasticCalculus} it follows that the L\'evy triplet of a L\'evy-subordinator always admits the form $(\gamma_l,0,\nu)$ with a measure $\nu$ on $(\mathbb{R},\mathcal{B}(\mathbb{R}))$ satisfying
\begin{align*}
\nu(-\infty,0)=0 \text{ and } \int_0^\infty	 \min(y,1) \,\nu(dy)<\infty.
\end{align*}

\begin{rem}
Let $B=(B(t),~t\geq 0)$ be a standard Brownian motion and $l=(l(t),~t\geq 0)$ be a L\'evy subordinator. The stochastic process defined by 
\begin{align*}
L(t):=B(l(t)),~t\geq 0,
\end{align*}
is called \textit{subordinated Brownian motion} and is again a L\'evy process (see \cite[Theorem 1.3.25]{LevyProcessesAndStochasticCalculus}). 
\end{rem}

In \cite{SubordGRFTheory} the authors propose a new approach to extend standard subordinated L\'evy processes on a higher dimensional parameter space. Motivated by the rich class of subordinated Brownian motions, the authors construct discontinuous random fields by subordinating a GRF on a $d$-dimensional parameter domain by $d$ one-dimensional L\'evy subordinators. In case of a two-dimensional parameter space the construction is as follows: For a GRF $W:\Omega\times \mathbb{R}_+^2\rightarrow \mathbb{R}$ and two (L\'evy-)subordinators $l_1,~l_2$ on $[0,D]$, with a finite $D>0$, we define the real-valued random field
\begin{align*}
L(x,y):=W(l_1(x),l_2(y)),\text{ for } x,y\in[0,D].
\end{align*}
Figure~\ref{Fig:SamplesSubordGRF} shows samples of a GRF with Mart\'ern-1.5 covariance function and the corresponding subordinated field where we used Poisson and Gamma processes as subordinators.
\begin{figure}[ht]
	\centering
	\subfigure{\includegraphics[scale=0.29]{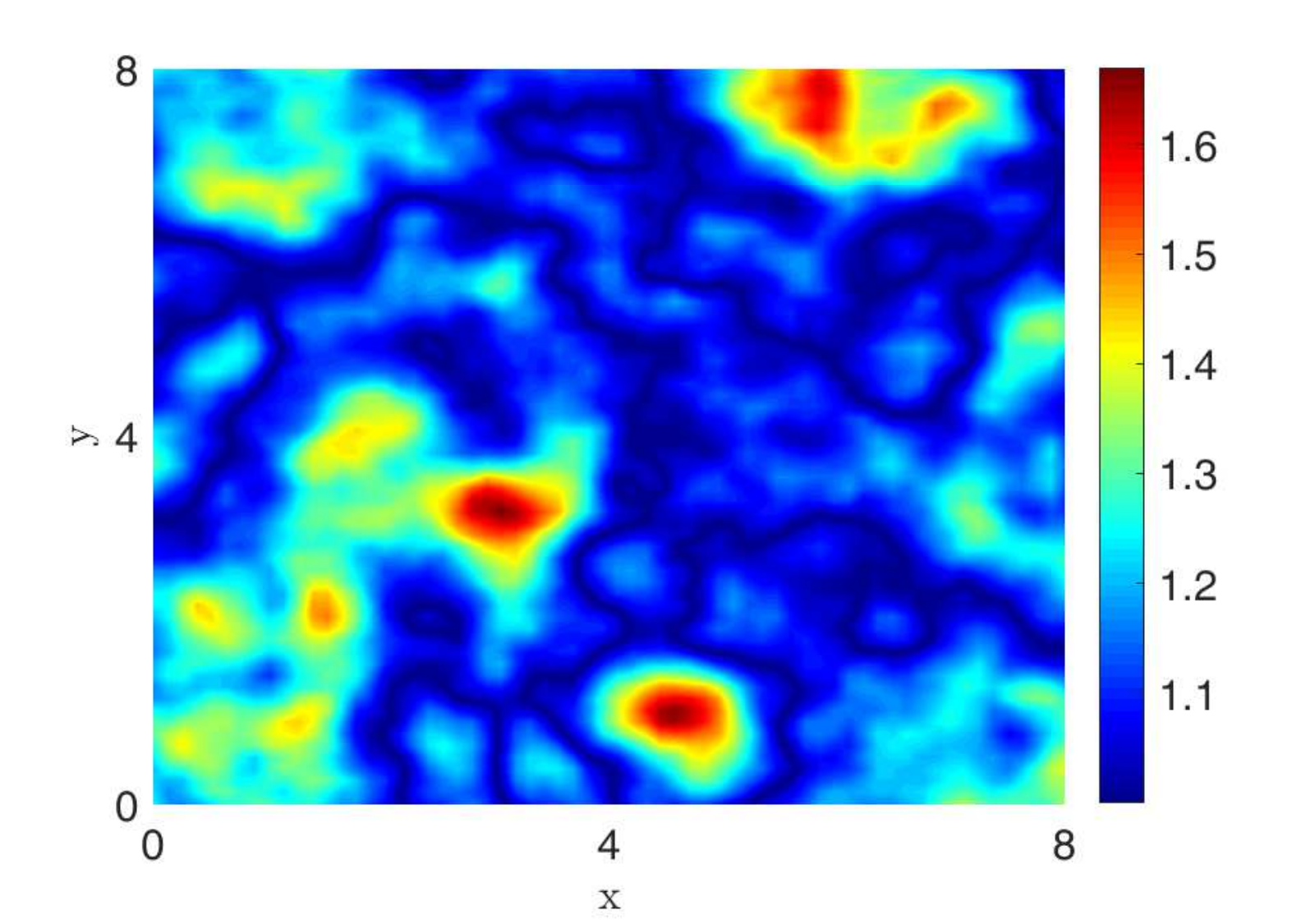}}
	\subfigure{\includegraphics[scale=0.29]{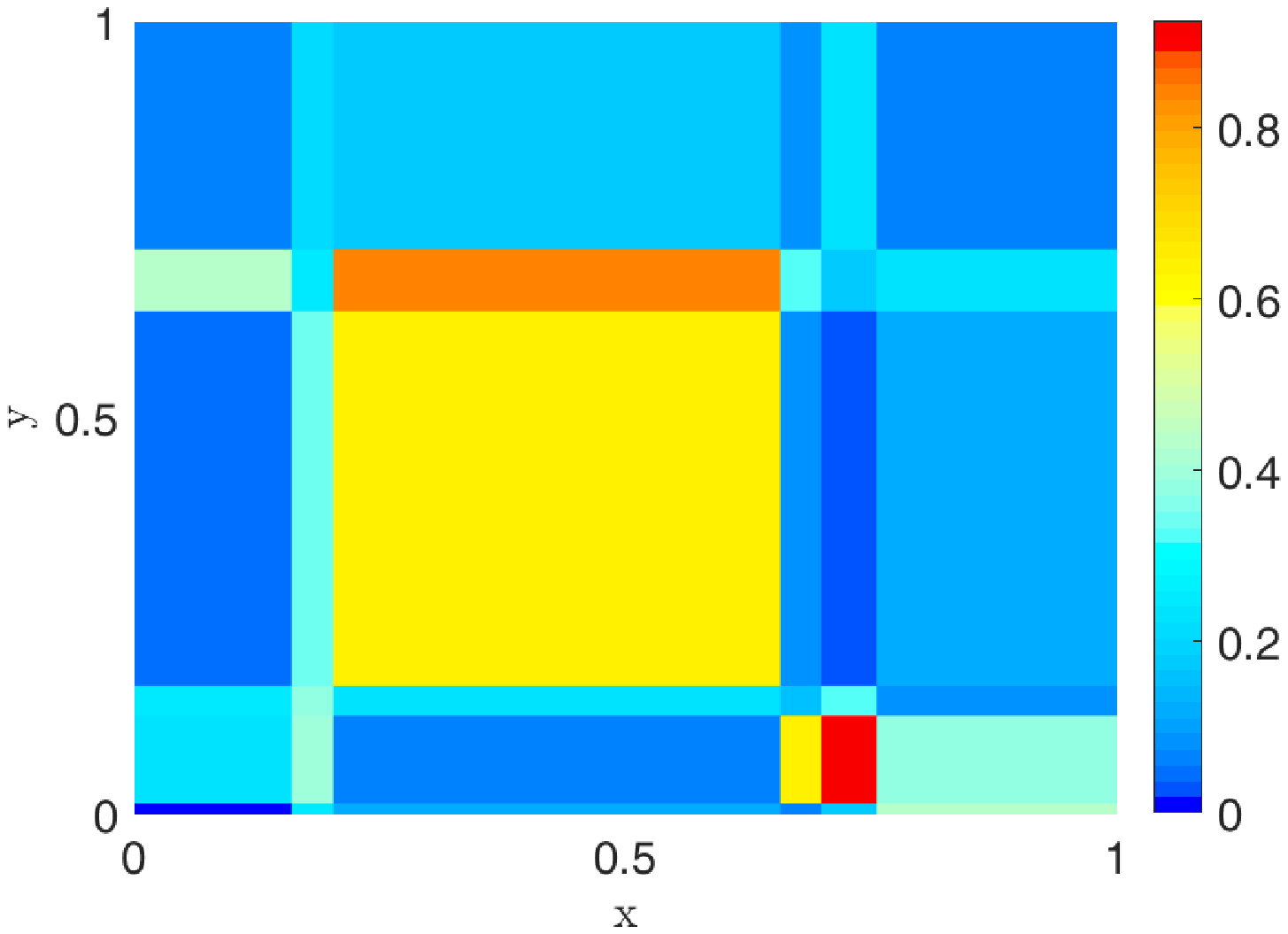}}
	\subfigure{\includegraphics[scale=0.29]{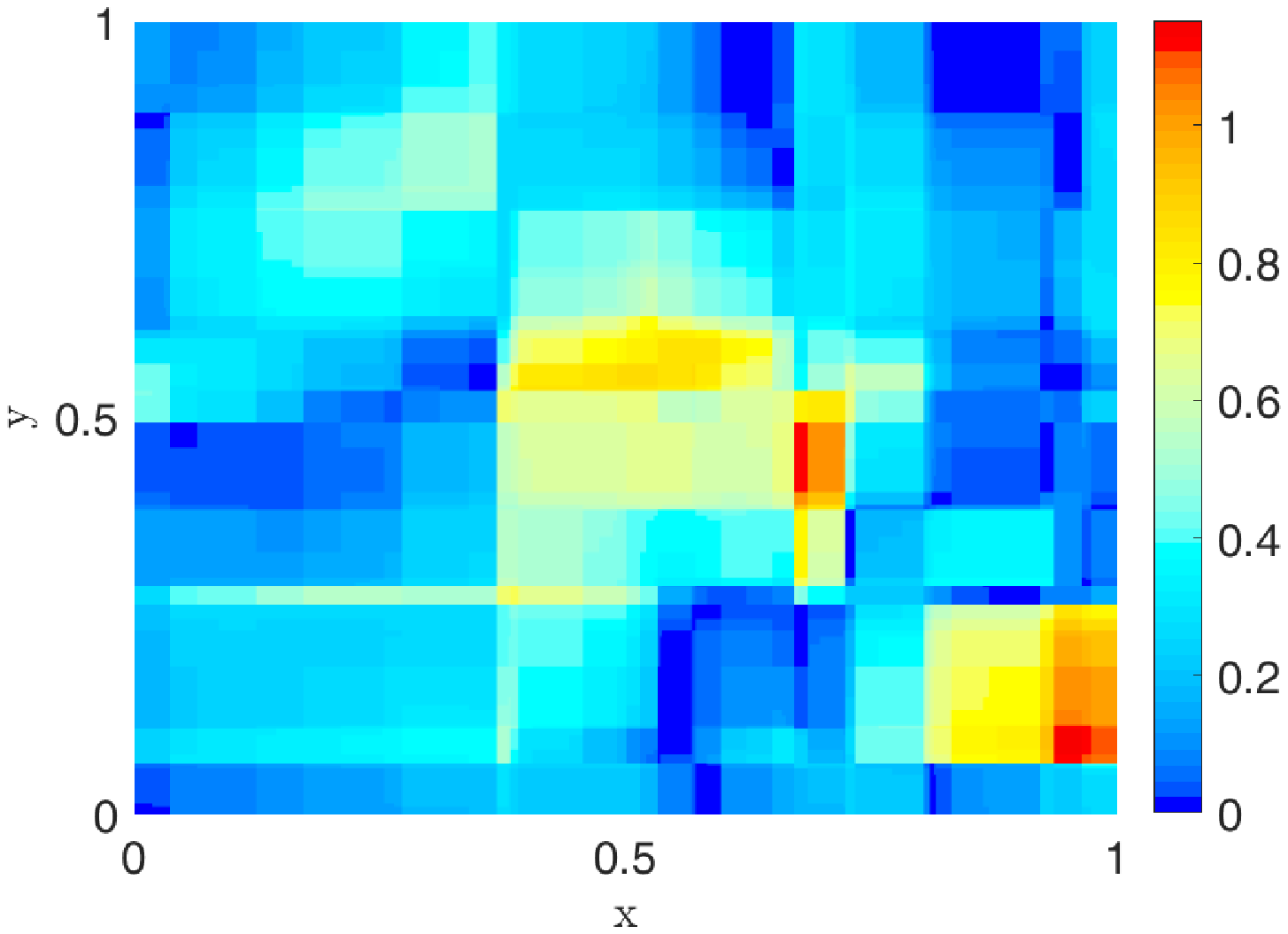}}
	\caption{Sample of Matérn-1.5-GRF (left), Poisson-subordinated GRF (middle) and Gamma-subordinated GRF (right).}\label{Fig:SamplesSubordGRF}
\end{figure}

This construction yields a rich class of discontinuous random fields which also admit a L\'evy-Khinchin-type formula. Further, the newly constructed random fields are also interesting for practical reasons, since they admit a semi-explicit formula for the covariance function which is very useful for applications, e.g. in statistical fitting.
For a theoretical investigation of the constructed random fields we refer to \cite{SubordGRFTheory}.

\subsection{Subordinated GRFs as diffusion coefficients in elliptic problems}
In the following, we define the specific diffusion coefficient that we consider in problem~\eqref{EQ:EllProblem} - \eqref{EQ:EllProblemBCN}. In order to allow discontinuities, we incorporate a subordinated GRF in the coefficient additionally to a Gaussian component. The construction of the coefficient is done so that Theorem \ref{TH:ExistenceTheoremElliptic} is applicable and, at the same time, the coefficient is as versatile as possible. 
\begin{definition}\label{DEF:DefCoeff}
We consider the domain $\mathcal{D}=(0,D)^2$ with $D<+\infty$\footnote{For simplicity we chose a square domain, recangular ones may be considered in the same way}. We define the jump-diffusion coefficient $a$ in problem \eqref{EQ:EllProblem} - \eqref{EQ:EllProblemBCN} with $d=2$ as
\begin{align}\label{EQ:DiffCoeffDefi}
a:\Omega\times\mathcal{D}\rightarrow (0,+\infty),~(\omega,x,y)\mapsto \overline{a}(x,y) + \Phi_1(W_1(x,y)) + \Phi_2(W_2(l_1(x),l_2(y))),
\end{align}
where 
\begin{itemize}
\item $\overline{a}:\mathcal{D}\rightarrow (0,+\infty)$ is deterministic, continuous and there exist constants $\overline{a}_+,\overline{a}_->0$ with $\overline{a}_-\leq \overline{a}(x,y)\leq \overline{a}_+$ for $(x,y)\in\mathcal{D}$.
\item $\Phi_1,~\Phi_2:\mathbb{R}\rightarrow [0,+\infty)$ are continuous .
\item $W_1$ and $W_2$ are zero-mean GRFs on $\mathcal{D}$ respectively on $[0,+\infty)^2$ with $\mathbb{P}-a.s.$ continuous paths.
\item $l_1$ and $l_2$ are L\'evy subordinators on $[0,D]$ with L\'evy triplets $(\gamma_1,0,\nu_1)$ and $(\gamma_2,0,\nu_2)$ which are independent of the GRFs $W_1$ and $W_2$.
\end{itemize}
\end{definition}	

\begin{rem} 
The first two assumptions ensure that the diffusion coefficient $a$ is positive over the domain $\mathcal{D}$. To show the convergence of the approximated diffusion coefficient in Subsection \ref{SUBSUBSEC:BoundOnE1} we have to impose independence of the GRFs $W_1$ and $W_2$ (see Assumption~\ref{ASS:GRFsIndependent} and the proof of Theorem~\ref{TH:QuantificationOfDiffApprLinftynorm}). This assumption is in the sense natural as also one-dimensional L\'evy processes admit an additive decomposition into a continuous part and a pure-jump part which are stochastically independent (L\'evy-It\^{o} decomposition, see e.g. \cite[Theorem 2.4.11]{LevyProcessesAndStochasticCalculus}). For the same reason the assumption that the L\'evy subordinators are independent of the GRFs is also natural (see for example \cite[Section 1.3.2]{LevyProcessesAndStochasticCalculus}). 
\end{rem}

In order to verify Assumption \ref{ASS:ProblemAssumptionsGeneral} \textit{i} and \textit{ii} we need the following Lemma.
\begin{lemma}\label{lem:MeasurabilityOfSubordGRF}
For fixed $(x,y)\in\mathcal{D}$ the mapping $\omega \mapsto a(\omega,x,y)$ is $\mathcal{F} -\mathcal{B}(\mathbb{R}_+)$-measurable. Further, for fixed $\omega \in \Omega$, the mapping $(x,y)\mapsto a(\omega,x,y)$ is $\mathcal{B}(\mathcal{D})-\mathcal{B}(\mathbb{R}_+)$-measurable.
\end{lemma}
\begin{proof}
Since $\overline{a}$ and $\Phi_1,\Phi_2$ are deterministic and continuous functions, it is enough to show the measurability of the mapping $\omega\mapsto W_2(\omega,l_1(\omega,x),l_2(\omega,y))$ to confirm the first claim. This can be seen as follows. Since $(\omega,x,y)\mapsto W_2(\omega,x,y)$ is a Carath\'eodory function, it is $\mathcal{F}\otimes \mathcal{B}(\mathbb{R}_+^2)-\mathcal{B}(\mathbb{R})$-measurable by \cite[Lemma 4.51]{InfiniteDimensionalAnalysis}. Further, since $\omega\mapsto l_1(x,\omega)$ and $\omega\mapsto l_2(y,\omega)$ are $\mathcal{F}-\mathcal{B}(\mathbb{R}_+)$-measurable and $\omega\mapsto \omega $ is $\mathcal{F}-\mathcal{F}$-measurable we obtain by \cite[Lemma 4.49]{InfiniteDimensionalAnalysis} that the mapping $\omega\mapsto (\omega,l_1(\omega,x),l_2(\omega,y))$ is $\mathcal{F}-\mathcal{F}\otimes \mathcal{B}(\mathbb{R}_+^2)$-measurable and, hence, the composition mapping $\omega\mapsto W_2(\omega,l_1(\omega,x),l_2(\omega,y))$ is $\mathcal{F}-\mathcal{B}(\mathbb{R})$-measurable. 

Next, we show that $(x,y)\mapsto W_2(\omega,l_1(\omega,x),l_2(\omega,y))$ is $\mathcal{B}(\mathcal{D})-\mathcal{B}(\mathbb{R})$-measurable for fixed $\omega\in \Omega$. Then the second claim follows by the continuity of the deterministic functions $\overline{a},\Phi_1$ and $\Phi_2$ together with the pathwise continuity of the GRF $W_1$.  For a fixed $\omega\in \Omega$, the mapping $(x,y)\mapsto W_2(\omega,x,y)$ is continuous and therefore $\mathcal{B}(\mathbb{R}_+^2)-\mathcal{B}(\mathbb{R})$-measurable. Further, the mappings $x\mapsto l_1(\omega,x)$ and $y\mapsto l_2(\omega,y)$ are $\mathcal{B}((0,D))-\mathcal{B}(\mathbb{R}_+)$-measurable, since the L\'evy subordinators are c\`adl\`ag mappings. If we define the extended functions $\tilde{l}_1(\omega,x,y):=l_1(\omega,x)$ and $\tilde{l}_2(\omega,x,y):=l_2(\omega,y)$ we obtain that the mappings $(x,y)\mapsto \tilde{l}_1(\omega,x,y)$ and $(x,y)\mapsto \tilde{l}_2(\omega,x,y)$ are Carath\'eodory functions and therefore $\mathcal{B}(\mathcal{D})-\mathcal{B}(\mathbb{R}_+)$-measurable by \cite[Lemma 4.51]{InfiniteDimensionalAnalysis}. Finally, by \cite[Lemma 4.49]{InfiniteDimensionalAnalysis}, the mapping $(x,y)\mapsto (\tilde{l}_1(\omega,x,y),\tilde{l}_2(\omega,x,y))$ is $\mathcal{B}(\mathcal{D})-\mathcal{B}(\mathbb{R}_+^2)$-measurable and therefore the composition function $(x,y)\mapsto W_2(\omega,\tilde{l}_1(\omega,x,y),\tilde{l}_2(\omega,x,y))=W_2(\omega,l_1(\omega,x),l_2(\omega,y))$ is $\mathcal{B}(\mathcal{D})-\mathcal{B}(\mathbb{R})$-measurable.
\end{proof}

\noindent Definition~\ref{DEF:DefCoeff} guarantees the existence of a pathwise weak solution to problem~\eqref{EQ:EllProblem}, as we prove in the following theorem.
\begin{theorem}\label{TH:ExistenceOfSolutionSubord}
Let $a$ be as in Definition \ref{DEF:DefCoeff} and let $f\in L^q(\Omega;H),~g\in L^q(\Omega;L^2(\Gamma_2))$ for some $q\in [1,+\infty)$. Then there exists a unique pathwise weak solution $u(\omega,\cdot)\in V$ to problem~\eqref{EQ:EllProblem} for $\mathbb{P}$-almost every $\omega\in\Omega$. Furthermore, $u\in L^r(\Omega;V)$ for all $r\in[1,q)$ and 
\begin{align*}
\|u\|_{L^r(\Omega;V)}\leq C(\overline{a}_-,\mathcal{D})(\|f\|_{L^q(\Omega;H)} + \|g\|_{L^q(\Omega;L^2(\Gamma_2))}),
\end{align*}
where $C(\overline{a}_-,\mathcal{D})>0$ is a constant depending only on the indicated parameter and the volume of $\mathcal{D}$.
\end{theorem}	
\begin{proof}
In order to apply Theorem \ref{TH:ExistenceTheoremElliptic} we have to verify Assumption \ref{ASS:ProblemAssumptionsGeneral}. We have $a_-(\omega)=\inf\limits_{x\in\mathcal{D}}a(\omega,x)\geq \overline{a}_- $ for every fixed $\omega\in\Omega$ by Definition~\ref{DEF:DefCoeff}. Further, $W_2(\omega)$ is continuous on $K(\omega):=[0,l_1(\omega,D)]\times[0,l_2(\omega,D)]$ and therefore
\begin{align*}
a_+(\omega)=\underset{(x,y)\in\mathcal{D}}{\sup}\,a(\omega,x,y)\leq \overline{a}_+ + \underset{(x,y)\in\mathcal{D}}{\sup}\,\Phi_1(W_1(\omega,x,y)) + \underset{(x,y)\in K(\omega)}{\sup}\,\Phi_2(W_2(\omega,x,y))<+\infty.
\end{align*}
For $1\leq r<q$ define $p:=(\frac{1}{r}-\frac{1}{q})^{-1}>0$. We observe that
\begin{align*}
0\leq \frac{1}{a_-(\omega)}\leq \frac{1}{\overline{a}_-}<\infty,
\end{align*}
$\mathbb{P}$-a.s. and hence $1/a_-\in L^p(\Omega;\mathbb{R})$. Therefore, Assumption \ref{ASS:ProblemAssumptionsGeneral} holds with $r=(1/p + 1/q)^{-1}$ and the assertion follows by Theorem \ref{TH:ExistenceTheoremElliptic}.
\end{proof}

\section{Approximation of the diffusion coefficient}\label{sec:approx_GRF}
To simulate the solution to the elliptic equation we need to define a tractable approximation of the diffusion coefficient. 

In order to approximate the solution to problem \eqref{EQ:EllProblem} - \eqref{EQ:EllProblemBCN} we face a new challenge regarding the GRF $W_2$ which is subordinated by the L\'evy processes $l_1$ and $l_2$: Due to the fact that the L\'evy subordinators in general can attain any value in $[0,+\infty)$ we have to consider (and approximate) the GRF $W_2$ on the \emph{unbounded} domain $[0,+\infty)$. In most cases where elliptic PDEs of the form \eqref{EQ:EllProblem} have been considered with a random coefficient, the problem is stated on a bounded domain, see e. g. \cite{AStudyOfElliptic, FiniteElementErrorAnalysisOfEllipticPDEsWIthRandomCoefficients, StrongAndWeakErrorEstimatesForTheSolutionsOfEllipticPDEsWithRandomCoefficients, QuasiMonteCarloFEMethodsForEllipticPDEsWithLognormalRandomCoefficients, CirculantEmbeddingWithWMCAnalysisForEllipicPDEWithLognormalCoefficients}. Many regularity results for GRFs formulated for a bounded parameter space cannot easily be transferred to an unbounded parameter space (see also \cite[Chapter 1]{RandomFieldsAndGeometry}, especially the discussion on p. 13). Even the Karhunen-Lo\`eve expansion of a GRF requires compactness of the domain (see e.g. \cite[Section 3.2]{RandomFieldsAndGeometry}). 

Furthermore, to show convergence of the solution in Section \ref{sec:convergence} we need to bound the coefficient from above by a deterministic upper bound $A$ (see Theorem~\ref{TH:ErrorBoundE2} Remark \ref{REM:OnDiffCutoff}). Subsequently we show that this induces an error in the solution approximation which can be controlled and which vanishes for growing $A$ (see Subsection \ref{SUBSUBSEC:BoundOnE1}).

Therefore, we derive an approximation in three steps: First, we bound the subordinators, and, second, we cut-off the diffusion coefficient itself. Finally, we consider approximations of the GRFs and the subordinators itself and prove the convergence of this approximation of the diffusion coefficient under suitable assumptions.
	
\subsection{First approximation: bounding the L\'evy subordinators}

For a fixed $K\in (0,+\infty)$, we define the cut-function $\chi_K:[0,+\infty)\rightarrow [0,K]$ as $\chi_K(z):=\min(z,K)$ for $z\in[0,+\infty)$. Instead of problem~\eqref{EQ:EllProblem} we consider the following modificated problem
\begin{align}\label{EQ:EllProblemCut}
-\nabla(a_K(\omega,\underline{x})\nabla u_K(\omega,\underline{x}))=f(\omega,\underline{x}) \text{ in }\Omega\times\mathcal{D},
	\end{align}
	and impose the boundary conditions
	\begin{align}
	u_K(\omega,\underline{x})&=0 \text{ on } \Omega\times \Gamma_1,\label{EQ:EllProblemBCDCut}\\
	a_K(\omega,\underline{x}) \overrightarrow{n}\cdot\nabla u_K(\omega,\underline{x})&=g(\omega,\underline{x}) \text{ on } \Omega\times \Gamma_2.\label{EQ:EllProblemBCNCut}
\end{align} 
Here, the diffusion coefficient is defined by 
\begin{align}\label{EQ:DiffCoeffDefiCut}
a_K:\Omega\times\mathcal{D}\rightarrow (0,+\infty),~(\omega,x,y)\mapsto \overline{a}(x,y) + \Phi_1(W_1(x,y)) + \Phi_2(W_2(\chi_K(l_1(x)),\chi_K(l_2(y)))).
\end{align}	
For functions $f\in L^q(\Omega;H)$ and $g\in L^q(\Omega;L^2(\Gamma_2))$ with $q\in [1,+\infty)$, there exists a weak solution $u_K\in L^r(\Omega;V)$ to problem~\eqref{EQ:EllProblemCut} - \eqref{EQ:EllProblemBCNCut} for $r\in[1,q)$ (see Theorem \ref{TH:ExistenceOfSolutionSubord})\footnote{For simplicity we assume one fixed $K$ for all spacial dimensions. The results in the subsequent sections hold for individual independent values in each spacial dimension as well.}.
\begin{rem}\label{rem:approxK}
We note that the influence of this problem modification can be controlled: one may choose $K>0$ such that
\begin{align*}
\mathbb{P}(\max(\underset{x\in[0,D]}{\sup}\,l_1(x),\underset{y\in[0,D]}{\sup}\,l_2(y))\geq K) = \mathbb{P}(\max(l_1(D),l_2(D))\geq K)<\varepsilon  
\end{align*} 
for any $\varepsilon>0$. In other words, pathwise the modified problem coincides with the original one up to a set of samples, whose probability can be made arbitrarily small.
\end{rem}

\subsection{Second modification: diffusion cut-off}
\label{Subs:diffusioncutoff}
We consider again the cut function $\chi_{A}(z) :=\min(z,A)$ for $z\in[0,+\infty)$ with a fixed positive number $A>0$ and consider the following problem

\begin{align}\label{EQ:EllProblemCutUpperCut}
-\nabla(a_{K,A}(\omega,\underline{x})\nabla u_{K,A}(\omega,\underline{x}))=f(\omega,\underline{x}) \text{ in }\Omega\times\mathcal{D},
	\end{align}
where we impose the boundary conditions
	\begin{align}
	u_{K,A}(\omega,\underline{x})&=0 \text{ on } \Omega\times \Gamma_1,\label{EQ:EllProblemBCDCutUpperCut}\\
	a_{K,A}(\omega,\underline{x}) \overrightarrow{n}\cdot\nabla u_{K,A}(\omega,\underline{x})&=g(\omega,\underline{x}) \text{ on } \Omega\times \Gamma_.\label{EQ:EllProblemBCNCutUpperCut}
\end{align} 
The diffusion coefficient $a_{K,A}$ is defined by 
\begin{align}\label{EQ:DiffCoeffDefiCutUpperCut}
a_{K,A}:\Omega\times\mathcal{D}&\rightarrow (0,+\infty),\notag\\
(\omega,x,y)&\mapsto \chi_{A}\Big(\overline{a}(x,y) + \Phi_1(W_1(x,y)) + \Phi_2(W_2(\chi_K(l_1(x)),\chi_K(l_2(y))))\Big).
\end{align}	
Again, Theorem~\ref{TH:ExistenceOfSolutionSubord} applies in this case and yields the existence of a pathwise weak solution $u_{K,A}\in L^r(\Omega;V)$ for $r\in[1,q)$ if $f\in L^q(\Omega;H)$ and $g\in L^q(\Omega;L^2(\Gamma_2))$. The error of the modification vanishes for growing $A$ as shown in the end of Section~\ref{SUBSUBSEC:BoundOnE1}.

\subsection{Approximation of GRF and subordinators}\label{sec:Workingassumptions}
Here, we show how to approximate the modificated diffusion coefficient $a_{K,A}$ using approximations $W_1^{\varepsilon_W}\approx W_1$, $W_2^{\varepsilon_W}\approx W_2$ of the GRFs and $l_1^{\varepsilon_l}\approx l_1$, $l_2^{\varepsilon_l}\approx l_2$ of the L\'evy subordinators.
To this end we need additional assumptions on the data of the elliptic problem, the covariance operators of the Gaussian fields and the subordinators.

\begin{assumption}\label{ASS:CutProblemEigenvalues}
Let $W_1$ be a zero-mean GRF on $[0,D]^2$ and $W_2$ be a zero-mean GRF on $[0,K]^2$. We denote by $q_1:[0,D]^2\times [0,D]^2\rightarrow\mathbb{R}$ and $q_2:[0,K]^2\times [0,K]^2\rightarrow\mathbb{R}$ the covariance functions of these random fields and by $Q_1,Q_2$ the associated covariance operators defined by
\begin{align*}
Q_j\phi=\int_{[0,z_j]^2}q_j((x,y),(x',y'))\phi(x',y')d(x',y'),
\end{align*}
for $\phi\in L^2([0,z_j]^2)$ with $z=(D,K)$ and $j=1,2$. We denote by $(\lambda_i^{(1)},e_i^{(1)},~i\in \mathbb{N})$ resp. $(\lambda_i^{(2)},e_i^{(2)},~i\in \mathbb{N})$ the eigenpairs associated to the covariance operators $Q_1$ and $Q_2$. In particular, $(e_i^{(1)},~i\in \mathbb{N})$ resp. $(e_i^{(2)},~i\in \mathbb{N})$ are ONBs of $L^2([0,D]^2)$ resp. $L^2([0,K]^2)$.
\begin{enumerate}
\item We assume that the eigenfunctions are continuously differentiable and there exist positive constants $\alpha, ~\beta, ~C_e, ~C_\lambda>0$ such that for any $i\in\mathbb{N}$ it holds

\begin{align*}
\|e_i^\hione\|_{L^\infty([0,D]^2)},~\|e_i^\hitwo\|_{L^\infty([0,K]^2)}&\leq C_e,\\
  \|\nabla e_i^\hione\|_{L^\infty([0,D]^2)},~\|\nabla	e_i^\hitwo\|_{L^\infty([0,K]^2)}&\leq C_e i^\alpha,~ \\
\sum_{i=1}^ \infty (\lambda_i^\hione + \lambda_i^\hitwo)i^\beta&\leq C_\lambda	< + \infty.
\end{align*}

\item There exist constants $\phi,~\psi, C_{lip}>0$ such that the continuous functions $\Phi_1,~\Phi_2:\mathbb{R}\rightarrow[0,+\infty)$ from Definition \ref{DEF:DefCoeff} satisfy
\begin{align*}
|\Phi_1'(x)|\leq \phi\, \exp(\psi |x|),~ |\Phi_2(x)-\Phi_2(y)|\leq C_{lip}\,|x-y| \text{ for } x,y\in \mathbb{R}.
\end{align*}
In particular, $\Phi_1\in C^1(\mathbb{R})$.

\item $f\in L^q(\Omega;H)$ and $g\in L^q(\Omega;L^2(\Gamma_2))$ for some $q\in (1,+\infty).$

\item $\overline{a}:\mathcal{D}\rightarrow (0,+\infty)$ is deterministic, continuous and there exist constants $\overline{a}_+,\overline{a}_->0$ with $\overline{a}_-\leq \overline{a}(x,y)\leq \overline{a}_+$ for $(x,y)\in\mathcal{D}$.

\item $l_1$ and $l_2$ are Lévy subordinatos on $[0,D]$ with Lévy triplets $(\gamma_1,0,\nu_1)$ and $(\gamma_2,0,\nu_2)$ which are intependent of the GRFs $W_1$ and $W_2$. Further, we assume that we have approximations $l_1^\apprlevy,~l_2^\apprlevy$ of these processes and there exist constants $\eta,C_l>0$ such that for every $s\in[1,\eta-1)$ it holds 
\begin{align*}
\mathbb{E}(|l_j(x)-l_j^\apprlevy(x)|^s)\leq C_l\varepsilon_l,
\end{align*}
for $\varepsilon_l >0$, $x\in[0,D]$ and $j=1,2$.
\end{enumerate}
\end{assumption}

\begin{rem}
Note that the first assumption on the eigenpairs of the GRFs is natural (see \cite{AStudyOfElliptic} and \cite{QuasiMonteCarloFEMethodsForEllipticPDEsWithLognormalRandomCoefficients}). For example, the case that $Q_1,~Q_2$ are Mat\'ern covariance operators are included. Assumption~\ref{ASS:CutProblemEigenvalues} \textit{ii} is necessary to be able to quantify the error of the approximation of the diffusion coefficient. Assumption~\ref{ASS:CutProblemEigenvalues} \textit{iii} is necessary to ensure the existence of a solution and has already been formulated in Assumption \ref{ASS:ProblemAssumptionsGeneral}. The last assumption ensures that we can approximate the L\'evy subordinators in an $L^s$-sense.
This can always be achieved under appropriate assumptions on the tails of the distribution of the subordinators, see~\cite[Assumption 3.6, Assumption 3.7 and Theorem 3.21]{ApproximationAndSimulation}.
\end{rem}
For any numerical simulation we have to approximate the GRF as well as the subordinating L\'evy processes, which results in an additional  approximation of the coefficient $a_{K,A}$ given in Equation~\eqref{EQ:DiffCoeffDefiCutUpperCut}. In the following we want to quantify the error induced by this approximation. 

It follows by an application of the Kolmogorov-Chentsov theorem (\cite[Theorem 3.5]{StochasticEquationsInInfiniteDimensions}) that $W_1$ and $W_2$ can be assumed to have H\"older-continuous paths with H\"older exponent $b \in (0,(2\gamma k-2)/(2k))$ for $0<\gamma\leq \min(1,\beta/(2\alpha))$ and every $k\in\mathbb{N}$ (see the proof of \cite[Lemma 3.5]{AStudyOfElliptic}). 
Further, it follows by an application of the Sobolev embedding theorem that 
$W_1\in L^n(\Omega;C^{0,\gamma}([0,D]^2))$ and $W_2\in L^n(\Omega;C^{0,\gamma}([0,K]^2))$, i.e.
\begin{align}\label{EQ:MeanHoelderContGRFs}
\mathbb{E}\left(\Big(\underset{z\neq z'\in [0,D]^2}{\sup}\,\frac{|W_1(z)-W_1(z')|}{|z-z'|_2^\gamma}\Big)^n\right),~\mathbb{E}\left(\Big(\underset{z\neq z'\in [0,K]^2}{\sup}\,\frac{|W_2(z)-W_2(z')|}{|z-z'|_2^\gamma}\Big)^n\right)<+\infty,
\end{align}
for every $n\in[1,+\infty)$ and $\gamma< \min(1,\beta/(2\,\alpha))$ (see \cite[Proposition 3.1]{StrongAndWeakErrorEstimatesForTheSolutionsOfEllipticPDEsWithRandomCoefficients}).  

Next, we prove a bound on the error of the approximated diffusion coefficient, where the GRFs are approximated by a discrete evaluation and (bi-)linear interpolation between these points (see \cite{AnalysisOfCirculantEmbeddingMethodsForSamplingStationaryRandomFields} and  \cite{CirculantEmbeddingWithWMCAnalysisForEllipicPDEWithLognormalCoefficients}).

\begin{lemma}\label{LE:StrongErrorBoundGRFApprox}
We consider the discrete grids $G_1^\apprgrf=\{(x_i,x_j)|~i,j=0,\dots,M_{\varepsilon_W}^{(1)}\}$ on $[0,D]^2$ and $G_2^\apprgrf=\{(y_i,y_j)|~i,j=0,\dots,M_{\varepsilon_W}^{(2)}\}$ on $[0,K]^2$ where $(x_i,~i=0,\dots,M_{\varepsilon_W}^{(1)})$ is an equidistant grid on $[0,D]$ with maximum step size $\varepsilon_W$ and $(y_i,~i=0,...,M_{\varepsilon_W}^{(2)})$ is an equidistant grid on $[0,K]$ with maximum step size $\varepsilon_W$. Further, let $W_1^\apprgrf$ and $W_2^\apprgrf$ be approximations of the GRFs $W_1,~W_2$ on the discrete grids $G_1^\apprgrf$ resp. $G_2^\apprgrf$ which are constructed by point evaluation of the random fields $W_1$ and $W_2$ on the grids and linear interpolation between the grid points. Under Assumption \ref{ASS:CutProblemEigenvalues} \textit{i} it holds for $n\in[1,+\infty)$:
\begin{align*}
\|W_1-W_1^\apprgrf\|_{L^n(\Omega;L^\infty([0,D]^2))}&\leq C(D,n)\varepsilon_W^\gamma\\
\|W_2-W_2^\apprgrf\|_{L^n(\Omega;L^\infty([0,K]^2))}&\leq C(K,n)\varepsilon_W^\gamma
\end{align*}
for $\gamma< \min (1,\beta/(2\alpha))$ where $\beta$ and $\alpha$ are the parameters from Assumption \ref{ASS:CutProblemEigenvalues}.
\end{lemma}

\begin{proof}
Note that for any fixed $\omega\in\Omega$ and $\cD_{ij}:=[x_i,x_{i+1}]\times[y_j,y_{j+1}]$ with $i,j\in\{1,...,M_{\varepsilon_W}^{(1)}\}$, it holds
\begin{align*}
&\underset{(x,y)\in \cD_{ij}}{\max} W_1^\apprgrf(x,y),\underset{(x,y)\in \cD_{ij}}{\min} W_1^\apprgrf(x,y)\\
&\hspace{5cm}\in \{W_1(x_i,y_j),W_1(x_{i+1},y_j),W_1(x_i,y_{j+1}),W_1(x_{i+1},y_{j+1})\}.
\end{align*} 
This holds since $W_1^\apprgrf$ is constructed by (bi-)linear interpolation of the GRF $W_1$ and the piecewise linear interpolants attain their maximum and minimum at the corners (the Hessian evaluated at the (unique) stationary point of the bilinear basis functions is always indefinite). Therefore, for a fixed $(x,y)\in[x_i,x_{i+1}]\times[y_j,y_{j+1}]$ it follows from the intermediate value theorem that $W_1^\apprgrf(x,y)=W_1(x',y')$ for appropriate $(x',y')\in[x_i,x_{i+1}]\times[y_j,y_{j+1}]$. 
Using this observation we estimate
\begin{align*} 
\|W_1-W_1^\apprgrf\|_{L^n(\Omega;L^\infty([0,D]^2))}^n&=\mathbb{E}\Big(\underset{(x,y)\in[0,D]^2}{\sup}\,|W_1(x,y)-W_1^\apprgrf(x,y)|^n\Big)\\
&\leq \mathbb{E}\Big(\underset{\substack{(x,y), (x',y')\in[0,D]^2, \\ |(x,y)^T-(x',y')^T|_2\leq \sqrt{2}\varepsilon_W}}{\sup}\,|W_1(x,y)-W_1(x',y')|^n\Big)\\
&=\varepsilon_W^{n\,\gamma}\,\mathbb{E}\Big(\underset{\substack{(x,y), (x',y')\in[0,D]^2, \\ |(x,y)^T-(x',y')^T|_2\leq \sqrt{2}\varepsilon_W}}{\sup}\,\frac{|W_1(x,y)-W_1(x',y')|^n}{\varepsilon_W^{n\,\gamma}}\Big)\\
&\leq 2^\frac{n\,\gamma}{2}\varepsilon_W^{n\,\gamma}\,\mathbb{E}\Big(\Big(\underset{\substack{(x,y)\neq (x',y')\in[0,D]^2, \\ |(x,y)^T-(x',y')^T|_2\leq \sqrt{2}\varepsilon_W}}{\sup}\,\frac{|W_1(x,y)-W_1(x',y')|}{|(x,y)^T-(x',y')^T|_2^\gamma}\Big)^n\Big)\\
&\leq 2^\frac{n\,\gamma}{2} \varepsilon_W^{n\,\gamma} \,\mathbb{E}\Big(\Big(\underset{(x,y)\neq (x',y')\in[0,D]^2}{\sup}\,\frac{|W_1(x,y)-W_1(x',y')|}{|(x,y)^T-(x',y')^T|_2^\gamma}\Big)^n\Big)\\
&\leq C(D) \varepsilon_W^{n\,\gamma},
\end{align*}
where we used Equation \eqref{EQ:MeanHoelderContGRFs} in the last step. Equivalently the error bound for $W_2$ follows.
\end{proof}

\begin{rem}\label{REM:ErrorGRFApprWeakNorm}
Note that Lemma~\ref{LE:StrongErrorBoundGRFApprox} immediately implies for $m\in[1,+\infty)$
\begin{align*}
\|W_1-W_1^\apprgrf\|_{L^n(\Omega;L^m([0,D]^2))}&=\mathbb{E}\Big(\big(\int_{[0,D]^2}|W_1(x,y)-W_1^\apprgrf(x,y)|^md(x,y)\big)^\frac{n}{m}\Big)^\frac{1}{n}\\
&\leq D^\frac{2}{m}\mathbb{E}\Big(\underset{(x,y)\in[0,D]^2}{\sup}\,|W_1(x,y)-W_1^\apprgrf(x,y)|^n\Big)^\frac{1}{n}\\
&=D^\frac{2}{m}\|W_1-W_1^\apprgrf\|_{L^n(\Omega;L^\infty([0,D]^2))}\leq C(D,m,n)\varepsilon_W^\gamma.
\end{align*}
\end{rem}
  
\subsection{Convergence to the modificated diffusion coefficient}\label{Subs:ConvModDiffCoeff}

Given some approximations $W_1^\apprgrf\approx W_1,~W_2^\apprgrf\approx W_2$ as in Lemma \ref{LE:StrongErrorBoundGRFApprox} and approximations $l_1^\apprlevy\approx l_1,~l_2^\apprlevy \approx	l_2$ as in Assumption \ref{ASS:CutProblemEigenvalues} \textit{v} as well as some fixed constants $K,A>0$, we approximate the diffusion coefficient $a_{K,A}$ in \eqref{EQ:DiffCoeffDefiCutUpperCut} by $a_{K,A}^{(\varepsilon_W,\varepsilon_l)}:\Omega\times\mathcal{D}\rightarrow(0,+\infty)$ with
\begin{align}\label{EQ:DiffCoeffApprDef}
a_{K,A}^{(\varepsilon_W,\varepsilon_l)}(x,y) = \chi_{A}\Big(\overline{a}(x,y) + \Phi_1(W_1^\apprgrf(x,y)) + \Phi_2(W_2^\apprgrf(\chi_K(l_1^\apprlevy(x)),\chi_K(l_2^\apprlevy(y))))\Big)
\end{align}
for $(x,y)\in\mathcal{D}$.
To prove a convergence result for this approximated coefficient (Theorem~\ref{TH:ErrorBoundCoeff}) we need the following two technical lemmas. The second can be proved by the use of \cite[Proposition 1.16]{LevyProcessesAndInfinitelyDivisibleDistributions}. For a detailed proof we refer to \cite{SubordGRFTheory}.
\begin{lemma}\label{LE:IneqOfNestedLpNorms}
For $n,m\in[1,+\infty)$ with $n\geq m$ and $\varphi\in L^n([0,D]^2\times\Omega;\mathbb{R},\lambda\otimes \mathbb{P}))$, where $\lambda$ denotes the Lebesgue measure on $(\mathbb{R}^2,\mathcal{B}(\mathbb{R}^2))$, it holds 
\begin{align*}
\|\varphi\|_{L^n(\Omega;L^m([0,D]^2))}\leq C(D,n,m) \|\varphi\|_{L^n([0,D]^2\times\Omega;\mathbb{R}))}.
\end{align*}
\end{lemma}

\begin{proof}
The case $n=m$ is trivial. For $n>m$ we use H\"older's inequality 
and obtain
\begin{align*}
\|\varphi\|_{L^n(\Omega;L^m([0,D]^2))}&= \Big(\int_\Omega \Big(\int_{[0,D]^2} |\varphi(x,y)|^m d(x,y) \Big)^\frac{n}{m} d\mathbb{P}\Big)^\frac{1}{n}\\
&\leq \Big(\int_\Omega \Big(\int_{[0,D]^2} |\varphi(x)|^n d(x,y) \Big)\,D^\frac{2(n-m)}{m} d\mathbb{P}\Big)^\frac{1}{n}\\
&=D^{\frac{2}{m}-\frac{2}{n}}\|\varphi\|_{L^n([0,D]^2\times\Omega;\mathbb{R}))}.
\end{align*}
\end{proof}

\begin{lemma}\label{LE:ExpValueIndepRVs}
	Let $W:\Omega\times \mathbb{R}_+^d\rightarrow \mathbb{R}$ be a $\mathbb{P}-a.s.$ continuous random field and let $Z:\Omega\rightarrow\mathbb{R}_+^d$ be a $\mathbb{R}_+^d$-valued random variable which is independent of the random field $W$. Further, let $\varphi:\mathbb{R}\rightarrow\mathbb{R}$ be a deterministic, continuous function. It holds
	\begin{align*}
	\mathbb{E}(\varphi(W(Z))=\mathbb{E}(\zeta(Z)),
	\end{align*}
	where $\zeta(z):=\mathbb{E}(\varphi(W(z))$ for deterministic $z\in \mathbb{R}_+^d$.
	\end{lemma}

\begin{theorem}\label{TH:ErrorBoundCoeff}
Let $W_1^\apprgrf\approx W_1,~W_2^\apprgrf\approx W_2$ be approximations of the GRFs on discrete grids as in Lemma \ref{LE:StrongErrorBoundGRFApprox}. Further, let $1\leq t\leq s<\eta -1$ and $0<\gamma<min(1,\beta/(2\alpha))$ such that $s\gamma\geq 2$. Under Assumption~\ref{ASS:CutProblemEigenvalues} we get the following error bound for the approximation of the diffusion coefficient:
\begin{align*}
\|a_{K,A}-a_{K,A}^{(\varepsilon_W,~\varepsilon_l)}\|_{L^s(\Omega;L^t([0,D]^2))}\leq  C(\varepsilon_W^\gamma + \varepsilon_l^\frac{1}{s}),
\end{align*}
with a constant $C$ which does not depend on the discretization parameters $\varepsilon_W$ and $\varepsilon_l$.
\end{theorem}

\begin{proof} Since the cut function $\chi_{A}$ is Lipschitz continuous with Lipschitz constant $1$ we calculate
\begin{align*}
\|a_{K,A}&-a_{K,A}^{(\varepsilon_W,~\varepsilon_l)}\|_{L^s(\Omega;L^t([0,D]^2))}\\
&~\leq  \|\Phi_1(W_1) - \Phi_1(W_1^\apprgrf)\|_{L^s(\Omega;L^t([0,D]^2))} \\
&~+ \|\Phi_2(W_2(\chi_K(l_1),\chi_K(l_2)))-\Phi_2(W_2^\apprgrf(\chi_K(l_1^\apprlevy),\chi_K(l_2^\apprlevy)))\|_{L^s(\Omega;L^t([0,D]^2))}\\&=:I_1 + I_2
\end{align*}
First, we consider $I_1$ and use Assumption \ref{ASS:CutProblemEigenvalues} \textit{ii} and the same calculation as in Remark \ref{REM:ErrorGRFApprWeakNorm} to get
\begin{align*}
I_1&\leq D^\frac{2}{t}\|\Phi_1(W_1) - \Phi_1(W_1^\apprgrf)\|_{L^s(\Omega;L^\infty([0,D]^2))}.
\end{align*}
The mean value theorem yields, for fixed $(x,y)\in[0,D]^2$ and an appropriately chosen value $\xi\in(\min(W_1(x,y),W_1^\apprgrf(x,y)),\max(W_1(x,y),W_1^\apprgrf(x,y)))$,
\begin{align*}
|\Phi_1(W_1(x,y))-&\Phi_1(W_1^\apprgrf(x,y))|\\
&= |\Phi_1'(\xi)| \,|W_1(x,y)-W_1^\apprgrf(x,y)|\\
&\leq \phi\,\exp(\psi|\xi|)|W_1(x,y)-W_1^\apprgrf(x,y)|\\
&\leq \phi\,\max\{\exp(\psi|W_1(x,y)|),\exp(\psi|W_1^\apprgrf(x,y)|)\}\,|W_1(x,y)-W_1^\apprgrf(x,y)|,
\end{align*}
for $\mathbb{P}$-almost every $\omega\in \Omega$. As already mentioned in the proof of Lemma~\ref{LE:StrongErrorBoundGRFApprox}, for any $\omega\in\Omega$ and fixed $\cD_{ij}:=[x_i,x_{i+1}]\times[y_j,y_{j+1}]$ with $i,j\in\{1,...,M_{\varepsilon_W}^{(1)}\}$, it holds
\begin{align*}
&\underset{(x,y)\in \cD_{ij}}{\max} W_1^\apprgrf(x,y),~\underset{(x,y)\in \cD_{ij}}{\min} W_1^\apprgrf(x,y)\\
&\hspace{5cm}\in \{W_1(x_i,y_j),W_1(x_{i+1},y_j),W_1(x_i,y_{j+1}),W_1(x_{i+1},y_{j+1})\}.
\end{align*} 
Therefore, we obtain the pathwise estimate
\begin{align*}
&\|\Phi_1(W_1)-\Phi_1(W_1^\apprgrf)\|_{L^\infty([0,D]^2)}\\
&\leq \underset{(x,y)\in[0,D]^2}{\max}\,\phi\,\max\{\exp(\psi|W_1(x,y)|),\exp(\psi|W_1^\apprgrf(x,y)|)\}\,\underset{(x,y)\in[0,D]^2}{\max}\,|W_1(x,y)-W_1^\apprgrf(x,y)|\\
&\leq \phi\exp\big(\psi\, \max\{\underset{(x,y)\in[0,D]^2}{\max}|W_1(x,y)|,\underset{(x,y)\in[0,D]^2}{\max}|W_1^\apprgrf(x,y)|\} \big)\\
&\phantom{\big(\psi\, \max\{\underset{(x,y)\in[0,D]^2}{\max}|W_1(x,y)|,\underset{(x,y)\in[0,D]^2}{\max}|W_1^\apprgrf(x,y)|\} \big)}\times  \underset{(x,y)\in[0,D]^2}{\max}|W_1(x,y)-W_1^\apprgrf(x,y)|\\
&=\phi\,\exp\Big(\psi\, \underset{(x\mathcal,y)\in[0,D]^2}{\max}\,|W_1(x,y)|\Big) \underset{(x,y)\in[0,D]^2}{\max}\,|W_1(x,y)-W_1^\apprgrf(x,y)|.
\end{align*}
Finally, we obtain for any $n_1,n_2\in[1,+\infty)$ with $1/n_1 + 1/n_2 = 1$ by H\"older's inequality

\begin{align*}
I_1&\leq D^\frac{2}{t}\|\Phi_1(W_1) - \Phi_1(W_1^\apprgrf)\|_{L^s(\Omega;L^\infty([0,D]^2))}\\
&\leq D^\frac{2}{t}\phi\,\|\exp(\psi|W_1|)\|_{L^{sn_1}(\Omega;L^\infty([0,D]^2))}\,\|W_1-W_1^\apprgrf\|_{L^{sn_2}(\Omega;L^\infty([0,D]^2))}\\
&\leq C(D)\varepsilon_W^\gamma,
\end{align*}
where we used Lemma~\ref{LE:StrongErrorBoundGRFApprox} and the fact that $\|\exp(\psi|W_1|)\|_{L^{sn_1}(\Omega;L^\infty([0,D]^2))}<\infty$ (see~\cite[Theorem 2.1.1]{RandomFieldsAndGeometry} and the proof of Theorem~\ref{TH:QuantificationOfDiffApprLinftynorm} for more details).

For the second summand we calculate:
\begin{align*}
I_2&\leq \|\Phi_2(W_2(\chi_K(l_1),\chi_K(l_2)))-\Phi_2(W_2(\chi_K(l_1^\apprlevy),\chi_K(l_2^\apprlevy)))\|_{L^s(\Omega;L^t([0,D]^2))}\\
&+\|\Phi_2(W_2(\chi_K(l_1^\apprlevy),\chi_K(l_2^\apprlevy)))-\Phi_2(W_2^\apprgrf(\chi_K(l_1^\apprlevy),\chi_K(l_2^\apprlevy)))\|_{L^s(\Omega;L^t([0,D]^2))}\\
&=I_3 + I_4
\end{align*}
We use the same calculation as in Remark~\ref{REM:ErrorGRFApprWeakNorm} and the Lipschitz continuity of $\Phi_2$ to calculate for the summand $I_4$
\begin{align*}
I_4&\leq D^\frac{2}{t}\|\Phi_2(W_2(\chi_K(l_1^\apprlevy),\chi_K(l_2^\apprlevy)))-\Phi_2(W_2^\apprgrf(\chi_K(l_1^\apprlevy),\chi_K(l_2^\apprlevy)))\|_{L^s(\Omega;L^\infty([0,D]^2))}\\
&\leq C_{lip}D^\frac{2}{t}\|W_2(\chi_K(l_1^\apprlevy),\chi_K(l_2^\apprlevy))-W_2^\apprgrf(\chi_K(l_1^\apprlevy),\chi_K(l_2^\apprlevy))\|_{L^s(\Omega;L^\infty([0,D]^2))}\\
&\leq C_{lip}D^\frac{2}{t}\|W_2-W_2^\apprgrf\|_{L^s(\Omega;L^\infty([0,K]^2))}\\
&\leq C_{lip}D^\frac{2}{t}C(K,s)\varepsilon_W^\gamma,
\end{align*}
where we used Lemma~\ref{LE:StrongErrorBoundGRFApprox}.
It remains to bound the summand $I_3$: We estimate using Lemma~\ref{LE:IneqOfNestedLpNorms}
\begin{align*}
&I_3 = \|\Phi_2(W_2(\chi_K(l_1),\chi_K(l_2)))-\Phi_2(W_2(\chi_K(l_1^\apprlevy),\chi_K(l_2^\apprlevy)))\|_{L^s(\Omega;L^t([0,D]^2))}\\
&~\leq D^{\frac{2}{t}-\frac{2}{s}}\|\Phi_2(W_2(\chi_K(l_1),\chi_K(l_2)))-\Phi_2(W_2(\chi_K(l_1^\apprlevy),\chi_K(l_2^\apprlevy)))\|_{L^s([0,D]^2\times\Omega)}\\
&~= D^{\frac{2}{t}-\frac{2}{s}}\Big( \int_{[0,D]^2} \mathbb{E}\Big( | \Phi_2\big(W_2(\chi_K(l_1(x)),\chi_K(l_2(y)))\big)\\
&\hspace{5cm}-\Phi_2(W_2(\chi_K(l_1^\apprlevy(x)),\chi_K(l_2^\apprlevy(y)))) |^s \Big) d(x,y)\Big)^\frac{1}{s}\\
&~\leq C_{lip}D^{\frac{2}{t}-\frac{2}{s}}\Big( \int_{[0,D]^2} \mathbb{E}( | W_2(\chi_K(l_1(x)),\chi_K(l_2(y)))-W_2(\chi_K(l_1^\apprlevy(x)),\chi_K(l_2^\apprlevy(y))) |^s ) d(x,y)\Big)^\frac{1}{s}\\
\end{align*}
We know by Lemma~\ref{LE:ExpValueIndepRVs} that it holds for $(x,y)\in[0,D]^2$
\begin{align*}
\mathbb{E}( | W_2(\chi_K(l_1(x)),\chi_K(l_2(y)))&-W_2(\chi_K(l_1^\apprlevy(x)),\chi_K(l_2^\apprlevy(y))) |^s ) \\
&~= \mathbb{E}(\varphi(\chi_K(l_1(x)),\chi_K(l_2(y)),\chi_K(l_1^\apprlevy(x)),\chi_K(l_2^\apprlevy(y))))
\end{align*}
where 
\begin{align*}
\varphi(x,y,v,w):=\mathbb{E}(|W_2(x,y) - W_2(v,w)|^s).
\end{align*}
For $(x,y)=(v,w)$ it holds $\varphi(x,y,v,w)=0$ and for $(x,y)\neq (v,w)\in [0,K]^2$ we get
\begin{align*}
\varphi(x,y,v,w)&= |(x,y)^T - (v,w)^T|_2^{\gamma s}\,\mathbb{E}\Big(\frac{|W_2(x,y) - W_2(v,w)|^s}{|(x,y)^T-(v,w)^T|_2^{\gamma s}}\Big)\\
&\leq |(x,y)^T - (v,w)^T|_2^{\gamma s} \mathbb{E}\left(\Big(\underset{z\neq z'\in [0,K]^2}{\sup}\,\frac{|W_2(z)-W_2(z')|}{|z-z'|_2^\gamma}\Big)^s\right)\\
&\leq C(K)|(x,y)^T - (v,w)^T|_2^{\gamma s}
\end{align*}
by Equation~\eqref{EQ:MeanHoelderContGRFs}.
Further, we know from H\"older's inequality for $\gamma s\geq 2$ that it holds
\begin{align*}
|(x,y)^T - (v,w)^T|_2^{\gamma s} &= ((x-v)^2 + (y-w)^2)^\frac{\gamma s}{2}\\
&\leq 2^{\frac{\gamma s}{2}-1}(|x-v|^{\gamma s} + |y-w|^{\gamma s} )
\end{align*}
and therefore we calculate
\begin{align*}
&\mathbb{E}( | W_2(\chi_K(l_1(x)),\chi_K(l_2(y)))-W_2(\chi_K(l_1^\apprlevy(x)),\chi_K(l_2^\apprlevy(y))) |^s ) \\
&~\leq 2^{\frac{\gamma s}{2}-1}C(K) \mathbb{E} (|\chi_K(l_1(x)) - \chi_K(l_1^\apprlevy(x))|^{\gamma s} + |\chi_K(l_2(y))-\chi_K(l_2^\apprlevy(y))|^{\gamma s} )\\
&~\leq 2^{\frac{\gamma s}{2}-1}C(K) \mathbb{E} (|l_1(x) - l_1^\apprlevy(x)|^{\gamma s} + |l_2(y)-l_2^\apprlevy(y)|^{\gamma s} )\\
&~\leq  2^{\frac{\gamma s}{2}}C(K)C_l\varepsilon_l
\end{align*}
where we used the Lipschitz continuity of $\chi_K$
and Assumption~\ref{ASS:CutProblemEigenvalues} \textit{v} in the last step. Therefore, we finally obtain
\begin{align*}
I_3\leq C_{lip} D^\frac{2}{t} 2^{\frac{\gamma}{2}}C(K)^\frac{1}{s} C_l^\frac{1}{s} \varepsilon_l^\frac{1}{s}=:C(C_{lip},D,t,\gamma,s,K,C_l)\varepsilon_l^\frac{1}{s}
\end{align*}
which proves that
\begin{align*}
\|a_{K,A}-a_{K,A}^{(\varepsilon_W,~\varepsilon_l)}\|_{L^s(\Omega;L^t([0,D]^2))} \leq C(D,K,C_{lip},t,\gamma,s,C_l)(\varepsilon_W^\gamma + \varepsilon_l^\frac{1}{s}).
\end{align*}
\end{proof}

\section{Convergence analysis}\label{sec:convergence}
In this section we derive an error bound for the approximation of the solution. We split the error in two components: the first component is associated with the cut-off of the diffusion coefficient we described in Subsection \ref{Subs:diffusioncutoff}. The second error contributor corresponds to the approximation of the GRFs and the L\'evy subordinators we considered in Subsection \ref{Subs:ConvModDiffCoeff}.
 
 Let $r\in[1,q)$ with $q$ as in Assumption~\ref{ASS:CutProblemEigenvalues} \textit{iii} and denote by $u_K\in L^r(\Omega;V)$ the weak solution to problem~\eqref{EQ:EllProblemCut} - \eqref{EQ:DiffCoeffDefiCut}. Further, let $u_{K,A}^{(\varepsilon_W,\varepsilon_l)}\in L^r(\Omega;V)$ be the weak solution to the problem
   \begin{align}\label{EQ:EllProblemCutAppr}
-\nabla(a_{K,A}^{(\varepsilon_W,\varepsilon_l)}(\omega,\underline{x})\nabla u_{K,A}^{(\varepsilon_W,\varepsilon_l)}(\omega,\underline{x}))=f(\omega,\underline{x}) \text{ in }\Omega\times\mathcal{D},
	\end{align}
with boundary conditions
	\begin{align}
	u_{K,A}^{(\varepsilon_W,\varepsilon_l)}(\omega,\underline{x})&=0 \text{ on } \Omega\times \Gamma_1,\label{EQ:EllProblemBCDCutAppr}\\
	a_{K,A}^{(\varepsilon_W,\varepsilon_l)}(\omega,\underline{x}) \overrightarrow{n}\cdot\nabla u_{K,A}^{(\varepsilon_W,\varepsilon_l)}(\omega,x)&=g(\omega,x) \text{ on } \Omega\times \Gamma_2.\label{EQ:EllProblemBCNCutAppr}
	\end{align}
	
\noindent Note that Theorem~\ref{TH:ExistenceOfSolutionSubord} also applies to the elliptic problem with coefficient $a_{K,A}^{(\varepsilon_W,\varepsilon_l)}$.
	   The aim of this section is to quantify the error of the approximation $u_{K,A}^{(\varepsilon_W,\varepsilon_l)} \approx u_K$\footnote{The error of the approximation $u_{K} \approx u$ may be controlled as in Remark~\ref{rem:approxK} but cannot be quantified for the solution.}. By the triangle inequality we obtain
	   \begin{align}\label{EQ:ErrorSplit}
	   \|u_K-u_{K,A}^{(\varepsilon_W,\varepsilon_l)}\| \leq \|u_K-u_{K,A}\| + \|u_{K,A} -u_{K,A}^{(\varepsilon_W,\varepsilon_l)}\| =:E_1 + E_2
	   \end{align}
	   for an arbitrary norm $\|\cdot\|$ (to be specified later). Here, $u_{K,A}$ is the solution to the truncated problem~\eqref{EQ:EllProblemCutUpperCut} - \eqref{EQ:DiffCoeffDefiCutUpperCut}. We consider the two error contributions $E_1$ and $E_2$ separately.
	   
\subsection{Bound on \texorpdfstring{$\mathbf{E_1}$}{E1}}
\label{SUBSUBSEC:BoundOnE1}
   
	   \begin{assumption}\label{ASS:GRFsIndependent}
	   We assume that the GRFs $W_1$ and $W_2$ occurring in the diffusion coefficient \eqref{EQ:DiffCoeffDefi} are stochastically independent.
	   \end{assumption}
	   The aim of this subsection is to show that the first error contributor $E_1$ in Equation~\eqref{EQ:ErrorSplit} vanishes for increasing cut-off threshold $A$. The strategy consists of two separated steps: in the first step we show the stability of the solution, which means that the value $E_1$ can be controlled by the quality of the approximation of the diffusion coefficient $a_{K,A}\approx a_K$. In the second step, we show that the quality of the approximation of the diffusion coefficient can be controlled by the cut-off threshold $A$. The first step is given by Theorem~\ref{TH:ErrorCutProblemBoundedByCoeffError}. In order to prove it we need the following lemma.
	   
	   \begin{lemma}\label{Lemma:ErrSolVnorm}
	   For fixed cut-off levels $A$ and $K$ we consider the solution $u_K\in L^r(\Omega;V)$ and its approximation $u_{K,A}\in L^r(\Omega;V)$ for $r\in[1,q)$. It holds the pathwise estimate
	   \begin{align*}
	   \|u_K-u_{K,A}\|_V\leq a_{K,+}C(\overline{a}_-,\mathcal{D})\|\nabla u_K-\nabla u_{K,A}\|_{L^2(\mathcal{D})},
	   \end{align*}
	   for $\mathbb{P}$-almost every $\omega\in \Omega$. Here, the constant $C(\overline{a}_-,\mathcal{D})$ only depends on the indicated parameters and we define $a_{K,+}(\omega):=\max \{1, \underset{(x,y)\in\mathcal{D}}{\operatorname{ess}\,\sup}\,a_K(\omega,x,y)\}<\infty$ for $\omega\in \Omega$.

	   \end{lemma}
	   
\begin{proof}
For a fixed $\omega\in \Omega$ we consider the variational problem: find a unique $\hat{w}\in V$ such that 
\begin{align*}
 B_{a_K}(\hat{w},v) = \langle u_K-u_{K,A}, v\rangle_{L^2(\mathcal{D})},
\end{align*}
for all $v\in V$.
By the Lax-Milgram theorem there exists a unique solution $\hat{w}\in V$ with 
\begin{align*}
\|\hat{w}\|_V \leq C'(\overline{a}_-,\mathcal{D})\|u_K-u_{K,A}\|_{L^2(\mathcal{D})},
\end{align*}
(see Theorem \ref{TH:ExistenceOfSolutionSubord} and \cite[Theorem 2.5]{AStudyOfElliptic}). Therefore, we obtain by H\"older's inequality
\begin{align*}
\|u_K-u_{K,A}\|_{L^2(\mathcal{D})}^2 &= B(\hat{w},u_K-u_{K,A}) \\
&= \langle a_K\nabla \hat{w}, \nabla u_K-\nabla u_{K,A}\rangle_{L^2(\mathcal{D})} \\
&\leq a_{K,+} \|\nabla \hat{w}\|_{L^2(\mathcal{D})} \|\nabla u_K-\nabla u_{K,A}\|_{L^2(D)}\\
&\leq \|u_K-u_{K,A}\|_{L^2(\mathcal{D})}a_{K,+}C'(\overline{a}_-,\mathcal{D})\|\nabla u_K-\nabla u_{K,A}\|_{L^2(\mathcal{D})}\\
&\leq \frac{1}{2}\|u_K-u_{K,A}\|_{L^2(\mathcal{D})}^2 + a_{K,+}^2C'(\overline{a}_-,\mathcal{D})^2/2 \|\nabla u_K-\nabla u_{K,A}\|_{L^2(\mathcal{D})}^2,
\end{align*}
where we used Young's inequality in the last step. Finally, we obtain
\begin{align*}\|u_K-u_{K,A}\|_V^2 &= \|u_K-u_{K,A}\|_{L^2(\mathcal{D})}^2 + \|\nabla u_K-\nabla u_{K,A}\|_{L^2(\mathcal{D})}^2 \\
&\leq (1+a_{K,+}^2C'(\overline{a}_-,\mathcal{D})^2) \|\nabla u_K-\nabla u_{K,A}\|_{L^2(\mathcal{D})}^2.
\end{align*}

\end{proof}	   
	   
	   \begin{theorem}\label{TH:ErrorCutProblemBoundedByCoeffError}
	   Let $f\in L^q(\Omega;H)$ and $g\in L^q(\Omega;L^2(\Gamma_2))$ for some $q\in[1,+\infty)$. Further, for a given number $t\in(1,+\infty)$ we define the dual number $n:=\frac{t}{t-1}$. Then, for any for $r\in[1,q/n)$, holds
	   \begin{align*}
\|u_K-u_{K,A}\|_{L^r(\Omega;V)}\leq C(\mathcal{D},\overline{a}_-,r)\,(\|f\|_{L^q(\Omega;H)} + \|g\|_{L^q(\Omega;\Gamma_2))})\,\mathbb{E}(\underset{\underline{x}\in\mathcal{D}}{\operatorname{ess}\,\sup}\,|a_K(\underline{x})-a_{K,A}(\underline{x})|^{rt})^{\frac{1}{rt}}
	   \end{align*}
	   \end{theorem}
	   
	   \begin{proof}
	By a direct calculation we obtain
	\begin{align*}
	\|\nabla u_K-\nabla u_{K,A}\|_{L^2(\mathcal{D})}^2&\leq \frac{1}{\overline{a}_-}\int_\mathcal{D}a_{K,A}|\nabla u_{K} - \nabla u_{K,A}|_2^2 d(x,y).
	\end{align*}
Since $u_K$ and $u_{K,A}$ are weak solutions of problem~\eqref{EQ:EllProblemCut} - \eqref{EQ:DiffCoeffDefiCut} resp.~\eqref{EQ:EllProblemCutUpperCut} - \eqref{EQ:DiffCoeffDefiCutUpperCut} it holds
\begin{align*}
\int_\mathcal{D}a_{K,A}\nabla u_{K,A}\cdot \nabla u_K d(x,y)&= \int_\mathcal{D}a_K|\nabla u_K|_2^2d(x,y), \\
~ \int_\mathcal{D}a_{K,A}|\nabla u_{K,A}|_2^2d(x,y) &= \int_\mathcal{D}a_K\nabla u_K \cdot \nabla u_{K,A} d(x,y),
\end{align*}
$\mathbb{P}$-almost surely and therefore 
\begin{align*}
\int_\mathcal{D}a_{K,A}|\nabla u_K-\nabla u_{K,A}|_2^2d(x,y) = \int_\mathcal{D}(a_{K,A} - a_K)\nabla u_K (\nabla u_K-\nabla u_{K,A})d(x,y).
\end{align*}
We estimate using H\"older's inequality
\begin{align*}
\|\nabla u_K-\nabla u_{K,A}\|_{L^2(\mathcal{D})}^2&\leq \frac{1}{\overline{a}_-} \|a_K-a_{K,A}\|_{L^\infty(\mathcal{D})}\|\nabla u_K\|_{L^2(\mathcal{D})} \|\nabla u_K - \nabla u_{K,A}\|_{L^2(\mathcal{D})}\\
&\leq  \frac{1}{\overline{a}_-} \|a_K-a_{K,A}\|_{L^\infty(\mathcal{D})} \|u_K\|_V\|\nabla u_K-\nabla u_{K,A}\|_{L^2(\mathcal{D})}
\end{align*}
and therefore we obtain
\begin{align*}
\|\nabla u_K-\nabla u_{K,A}\|_{L^2(\mathcal{D})}\leq \frac{1}{\overline{a}_-} \|a_K-a_{K,A}\|_{L^\infty(\mathcal{D})} \|u_K\|_V.
\end{align*}
Using Lemma \ref{Lemma:ErrSolVnorm} we obtain the pathwise estimate
\begin{align*}
\|u_K-u_{K,A}\|_V\leq C(\overline{a}_-,\mathcal{D})a_{K,+}\|a_K-a_{K,A}\|_{L^\infty(\mathcal{D})} \|u_K\|_V.
\end{align*}
Using again H\"older's inequality we have 
\begin{align*}
\|u_K-u_{K,A}\|_{L^r(\Omega;V)}\leq C(\overline{a}_-,\mathcal{D}) \mathbb{E}(\underset{\underline{x}\in\mathcal{D}}{\operatorname{ess}\,\sup}\,|a_K(\underline{x})-a_{K,A}(\underline{x})|^{rt})^{\frac{1}{rt}}\mathbb{E}(a_{K,+}^{nr}\|u_K\|_{V}^{nr})^\frac{1}{nr}.
\end{align*}
By assumption it holds $nr<q$. Therefore, we can choose a real number $\rho>1$ such that $nr\rho<q$. We define the dual number $\rho':=\frac{\rho}{\rho-1}\in (1,+\infty)$ and use H\"older's inequality to calculate

\begin{align*}
\|u_K-u_{K,A}\|_{L^r(\Omega;V)}\leq C(\overline{a}_-,\mathcal{D}) \mathbb{E}(\underset{\underline{x}\in\mathcal{D}}{\operatorname{ess}\,\sup}\,|a_K(\underline{x})-a_{K,A}(\underline{x})|^{rt})^{\frac{1}{rt}}\mathbb{E}(a_{K,+}^{nr\rho'})^{\frac{1}{nr\rho'}}\|u_K\|_{L^{nr\rho}(\Omega;V)}.
\end{align*}

Obviously Theorem~\ref{TH:ExistenceOfSolutionSubord} applies also to problem~\eqref{EQ:EllProblemCut} - \eqref{EQ:DiffCoeffDefiCut}. Therefore, since $nr\rho<q$ by assumption we conclude 
\begin{align*}
\|u_K-u_{K,A}\|_{L^r(\Omega;V)}\leq C(\mathcal{D},\overline{a}_-,r)\mathbb{E}(\underset{\underline{x}\in\mathcal{D}}{\operatorname{ess}\,\sup}\,|a_K(\underline{x})-a_{K,A}(\underline{x})|^{rt})^{\frac{1}{rt}}(\|f\|_{L^q(\Omega;H)} + \|g\|_{L^q(\Omega;\Gamma_2))}),
\end{align*}
where we additionally used the fact that $\mathbb{E}(a_{K,+}^{nr\rho'})^{\frac{1}{nr\rho'}}<+\infty$ (see Theorem \ref{TH:QuantificationOfDiffApprLinftynorm}). 
\end{proof}
	   
In other words, finding a bound for the error contribution $E_1$ in Equation~\eqref{EQ:ErrorSplit} reduces to quantifying the quality of the approximation of the diffusion coefficient $a_{K,A}\approx a_K$. For readability the proof of the following theorem can be found in Appendix~\ref{appendix:proof}

\begin{theorem}\label{TH:QuantificationOfDiffApprLinftynorm}
For any $n\in(1,+\infty)$ it holds

\begin{align*}
\mathbb{E}(\underset{\underline{x}\in\mathcal{D}}{\operatorname{ess}\,\sup}\, a_K(\underline{x})^n)^\frac{1}{n}<+\infty.
\end{align*}
Further, for any $\delta>0$ there exists a constant $A=A(\delta,n)>0$ such that 
	   \begin{align*}
	   \mathbb{E}(\underset{\underline{x}\in\mathcal{D}}{\operatorname{ess}\,\sup}\,|a_K(\underline{x})-a_{K,A}(\underline{x})|^{n})^{1/n}<\delta.
	   \end{align*}
\end{theorem}

From Theorem~\ref{TH:ErrorCutProblemBoundedByCoeffError} together with Theorem~\ref{TH:QuantificationOfDiffApprLinftynorm} we obtain
	   \begin{align*}
	   \|u_K-u_{K,A}\|_{L^r(\Omega;V)} \rightarrow 0 \text{, for } A\rightarrow\infty,
	   \end{align*}
for every $r<q$.
	   
\subsection{Bound on \texorpdfstring{$\mathbf{E_2}$}{E2}} The aim is to bound the second term of Equation~\eqref{EQ:ErrorSplit} given by 
	  \begin{align*}
	  E_2=\|u_{K,A} - u_{K,A}^{(\varepsilon_W,\varepsilon_l)}\|
	  \end{align*}
in an appropriate norm. For technical reasons we have to impose an additional assumption on the solution of the truncated problem. The subsequent remarks discuss situations under which this assumption is fulfilled.
\begin{assumption}\label{ASS:IntegrabilityOfSolGradient}
We assume that there exist constants $j_{reg}>0$ and $k_{reg}\geq 2$ such that
\begin{align}\label{EQ:SolRegAssumption}
C_{reg}:=\mathbb{E}(\|\nabla u_{K,A}\|_{L^{2 + j_{reg}}(\mathcal{D})}^{k_{reg}})< +\infty. 
\end{align}
\end{assumption}

Since $u_{K,A}\in H^1(\mathcal{D})$ we already know that $\nabla u_{K,A}\in L^2(\mathcal{D})$. Assumption~\ref{ASS:IntegrabilityOfSolGradient} requires a slightly higher integrability over the spatial domain. Since this is an assumption on the  regularity of the solution $u_{K,A}$ we denote the above constant by $C_{reg}$.

\begin{rem}\label{REM:OnHigherIntAssumpSol}
Note that Assumption~\ref{ASS:IntegrabilityOfSolGradient} is fulfilled if there exists $\theta\in(0,1)$ such that
\begin{align}\label{EQ:AlternativeAssumptionOnSolRef}
\|u_{K,A}\|_{H^{1+\theta}(\mathcal{D})}\leq C(f,a)
\end{align}
with some constant $C(f,a)$ with $\mathbb{E}(C(f,a)^{k_{reg}})<\infty$. This is true since for any $\rho\geq 2$  and an arbitrary function $\varphi\in H^{1+\theta}(\mathcal{D})$ the inequality
\begin{align*}
\|\nabla \varphi\|_{L^\rho(\mathcal{D})} \leq C\|\nabla \varphi\|_{H^{1-\frac{2}{\rho}(\mathcal{D})}(\mathcal{D})}\leq C\|\varphi\|_{H^{2-\frac{2}{\rho}}(\mathcal{D})}
\end{align*}
holds for $\theta:=1-\frac{2}{\rho}$ (see \cite[Theorem 6.7]{HitchhikersGuideToTheFractionalSobolevSpaces}). Here, the constant $C=C(\mathcal{D}, \theta)$ depends only on the indicated parameters. Hence, the condition \eqref{EQ:AlternativeAssumptionOnSolRef} implies Equation \eqref{EQ:SolRegAssumption} with $j_{reg}=\frac{2\theta}{1-\theta}$.
\end{rem}

\begin{rem}\label{REM:OnDiffCutoff}
	Note that there are several results about higher integrability of the gradient of the solution to an elliptic PDE of the form~\eqref{EQ:EllProblemCutUpperCut} - \eqref{EQ:DiffCoeffDefiCutUpperCut}. For instance \cite{RegularityResultsForLaplaceInterfaceProblemsInTroDimensions} yields that the solution $u_{K,A}$ has $H^{1+\delta/(2\pi)}$ regularity with $\delta=min(1,\overline{a}_-,A^{-1})$ under mixed boundary conditions and under the assumption that $a_{K,A}$ is piecewise constant (see~\cite[Theorem 7.3]{RegularityResultsForLaplaceInterfaceProblemsInTroDimensions}). This corresponds to the case where no Gaussian noise is considered (i.e. $\Phi_1\equiv 0$) and $\overline{a}$ is constant. Another important result is given in \cite{HigherIntegrabilityOfTheGradientInDegenerateEllipticEquations}. It follows by \cite[Theorem 1]{HigherIntegrabilityOfTheGradientInDegenerateEllipticEquations} that under the assumption that there exists $q>2$ with $f\in L^q(\mathcal{D})$ $\mathbb{P}-a.s.$ there exists a constant $C=C(\mathcal{D},\|f\|_{L^q(\mathcal{D})},\overline{a}_-,A)$ and a positive number $\vartheta=\vartheta(\mathcal{D},\|f\|_{L^q(\mathcal{D})},\overline{a}_-,A)>0$ only depending on the indicated parameters, such that:
	\begin{align}\label{EQ:SolRegularitySpecialCase}
	\| \nabla u_{K,A}\|_{L^{2+\vartheta}(\mathcal{D})}\leq C.
	\end{align}
	In particular, if the right hand side $f$ of the problem is deterministic, then $\vartheta$ and the constant $C$ in \eqref{EQ:SolRegularitySpecialCase} are deterministic and one immediately obtains 
	\begin{align*}
	\mathbb{E}(\| \nabla u_{K,A}\|_{L^{2+\vartheta}(\mathcal{D})}^{k_{reg}})<+\infty,
	\end{align*}
	for any $k_{reg}\geq 1$ and a deterministic, positive constant $\vartheta>0$.  We note that although \cite[Theorem 1]{HigherIntegrabilityOfTheGradientInDegenerateEllipticEquations} suggests a dependence of the parameter $\vartheta=\vartheta(\mathcal{D},\|f\|_{L^q(\mathcal{D})},\overline{a}_-,A)$ on the constant $A$, this dependence is not numerically detectable for our diffusion coefficient, as numerical experiments show. Of course, it depends on the other parameters $\mathcal{D},\|f\|_{L^q(\mathcal{D})}$ and $\overline{a}_-$.
	\end{rem}

Next, we show that for a given approximation of the diffusion coefficient, the resulting error contributor $E_2$ is bounded by the approximation error of the diffusion coefficient.
Similar to the corresponding assertion we gave in Subsection \ref{SUBSUBSEC:BoundOnE1} we need the following Lemma for the proof of this error bound. For a proof we refer to Lemma \ref{Lemma:ErrSolVnorm}.
 \begin{lemma}\label{Lemma:ErrSolVnormE2}
	   For fixed cut-off levels $A$, $K$ and fixed approximation parameters $\varepsilon_W,\varepsilon_l$ we consider the PDE solutions $u_{A,K}\in L^r(\Omega;V)$ and  $u_{K,A}^{(\varepsilon_W,\varepsilon_l)}\in L^r(\Omega;V)$ for $r\in[1,q)$. It holds the pathwise estimate
	   \begin{align*}
	   \|u_{K,A}-u_{K,A}^{(\varepsilon_W,\varepsilon_l)}\|_V\leq a_{K,+}C(\overline{a}_-,\mathcal{D})\|\nabla u_{K,A}-\nabla u_{K,A}^{(\varepsilon_W,\varepsilon_l)}\|_{L^2(\mathcal{D})},
	   \end{align*}
	   for $\mathbb{P}$-almost every $\omega\in \Omega$. Here, constant $C(\overline{a}_-,\mathcal{D})$ depends only on the indicated parameters and $a_{K,+}(\omega):=\max \{1, \underset{(x,y)\in\mathcal{D}}{\operatorname{ess}\,\sup}\,a_K(\omega,x,y)\}<\infty$ for $\omega\in \Omega$.

	   \end{lemma}

	\begin{theorem}\label{TH:ErrorBoundE2}
	Let $r\geq 2$ and $b,c\in[1,+\infty]$ be given such that it holds
	\begin{align*}
	rc\gamma\geq 2 \text{ and }2b\leq rc< \eta-1 
	\end{align*}
	with a fixed real number $\gamma\in(0,min(1,\beta/(2\alpha))$. Here, the parameters $\eta, \alpha$ and $\beta$ are determined by the GRFs $W_1$, $W_2$ and the L\'evy subordinators $l_1$, $l_2$ (see Assumption \ref{ASS:CutProblemEigenvalues}). \\
	Let $m,n\in[1,+\infty]$ be real numbers such that 
	\begin{align*}
	\frac{1}{m} + \frac{1}{c} = \frac{1}{n} + \frac{1}{b}=1,
	\end{align*}
	and let $k_{reg}\geq 2$ and $j_{reg}>0$ be the regularity specifiers given by Assumption~\ref{ASS:IntegrabilityOfSolGradient}.
	If it holds that
	\begin{align*}
	n<1+\frac{j_{reg}}{2} \text{ and } rm< k_{reg},
	\end{align*}
	then the approximated solution $u_{K,A}^{(\varepsilon_W,\varepsilon_l)}$ converges to the solution $u_{K,A}$ of the truncated problem for $\varepsilon_W,\varepsilon_l\rightarrow 0$ and it holds
	\begin{align*}
	\|u_{K,A}-u_{K,A}^{(\varepsilon_W,\varepsilon_l)}\|_{L^r(\Omega;V)}&\leq  C(\overline{a}_-,\mathcal{D},r)\|a_{K,A}^{(\varepsilon_W,\varepsilon_l)}-a_{K,A}\|_{L^{rc}(\Omega;L^{2b}(\mathcal{D}))}\|\nabla u_{K,A}\|_{L^{rm}(\Omega;L^{2n}(\mathcal{D}))}\\
	&\leq  C_{reg}C(\overline{a}_-,\mathcal{D},r)(\varepsilon_W^\gamma + \varepsilon_l^\frac{1}{rc}).
	\end{align*}
	\end{theorem}
	
	\begin{proof}
	By a direct calculation we obtain the pathwise estimate
	\begin{align*}
	\|\nabla u_{K,A}-\nabla u_{K,A}^{(\varepsilon_W,\varepsilon_l)}\|_{L^2(\mathcal{D})}^2&\leq \frac{1}{\overline{a}_-}\int_\mathcal{D}a_{K,A}^{(\varepsilon_W,\varepsilon_l)}(|\nabla u_{K,A} - \nabla u_{K,A}^{(\varepsilon_W,\varepsilon_l)}|_2^2) d\underline{x}.
	\end{align*}
	Since $u_{K,A}$ (resp. $u_{K,A}^{(\varepsilon_W,\varepsilon_l)}$) is the weak solution to problem~\eqref{EQ:EllProblemCutUpperCut} - \eqref{EQ:DiffCoeffDefiCutUpperCut} (resp.~\eqref{EQ:EllProblemCutAppr} - \eqref{EQ:EllProblemBCNCutAppr}) we have
	\begin{align*}
	\int_D a_{K,A}^{(\varepsilon_W,\varepsilon_l)}\nabla u_{K,A}^{(\varepsilon_W,\varepsilon_l)}\cdot \nabla u_{K,A} d\underline{x} &= \int_\mathcal{D}a_{K,A}|\nabla u_{K,A}|_2^2d\underline{x},\\
	\int_\mathcal{D}a_{K,A}^{(\varepsilon_W,\varepsilon_l)}|\nabla u_{K,A}^{(\varepsilon_W,\varepsilon_l)}|_2^2d\underline{x} &= \int_\mathcal{D}a_{K,A}\nabla u_{K,A}\cdot \nabla u_{K,A}^{(\varepsilon_W,\varepsilon_l)}d\underline{x}
	\end{align*}
	$\mathbb{P}$-a.s. and therefore
	\begin{align*}
	\int_Da_{K,A}^{(\varepsilon_W,\varepsilon_l)}|\nabla u_{K,A} - \nabla u_{K,A}^{(\varepsilon_W,\varepsilon_l)}|_2^2d\underline{x}=\int_\mathcal{D}(a_{K,A}^{(\varepsilon_W,\varepsilon_l)}-a_{K,A})\nabla u_{K,A}\cdot(\nabla u_{K,A} - \nabla u_{K,A}^{(\varepsilon_W,\varepsilon_l)}) d\underline{x}.
	\end{align*}
	Using H\"older's inequality we calculate
	\begin{align*}
	\|\nabla u_{K,A}-\nabla u_{K,A}^{(\varepsilon_W,\varepsilon_l)}\|_{L^2(\mathcal{D})}^2&\leq \frac{1}{\overline{a}_-} \|(a_{K,A}^{(\varepsilon_W,\varepsilon_l)}-a_{K,A})\nabla u_{K,A}\|_{L^2(\mathcal{D})}\|\nabla u_{K,A} - \nabla u_{K,A}^{(\varepsilon_W,\varepsilon_l)}\|_{L^2(\mathcal{D})}
	\end{align*}
	and therefore
	\begin{align*}
		\|\nabla u_{K,A}-\nabla u_{K,A}^{(\varepsilon_W,\varepsilon_l)}\|_{L^2(\mathcal{D})}\leq \frac{1}{\overline{a}_-} \|(a_{K,A}^{(\varepsilon_W,\varepsilon_l)}-a_{K,A})\nabla u_{K,A}\|_{L^2(\mathcal{D})},
	\end{align*}
\noindent Next, we apply Lemma \ref{Lemma:ErrSolVnormE2} to obtain the following estimate.
\begin{align*}
\|u_{K,A}-u_{K,A}^{(\varepsilon_W,\varepsilon_l)}\|_{V}\leq C(\overline{a}_-,\mathcal{D})a_{K,+}\|(a_{K,A}^{(\varepsilon_W,\varepsilon_l)}-a_{K,A})\nabla u_{K,A}\|_{L^2(\mathcal{D})}
\end{align*}
\noindent Hence, it remains to bound the norm $\|(a_{K,A}^{(\varepsilon_W,\varepsilon_l)}-a_{K,A})\nabla u_{K,A}\|_{L^2(\mathcal{D})}$. By H\"older's inequality we obtain 
	\begin{align*}
	\|(a_{K,A}^{(\varepsilon_W,\varepsilon_l)}-a_{K,A})\nabla u_{K,A}\|_{L^2(\mathcal{D})} \leq \|a_{K,A}^{(\varepsilon_W,\varepsilon_l)}-a_{K,A}\|_{L^{2b}(\mathcal{D})}\|\nabla u_{K,A}\|_{L^{2n}(\mathcal{D})}.
	\end{align*}
Applying H\"older's inequality once more we estimate
\begin{align*}
\|u_{K,A}-u_{K,A}^{(\varepsilon_W,\varepsilon_l)}\|_{L^r(\Omega;V)}&\leq C(\overline{a}_-,\mathcal{D})\|a_{K,A}^{(\varepsilon_W,\varepsilon_l)}-a_{K,A}\|_{L^{rc}(\Omega;L^{2b}(\mathcal{D}))}\mathbb{E}(a_{K,+}^{rm}\|\nabla u_{K,A}\|_{L^{2n}(\mathcal{D})}^{rm})^\frac{1}{rm}
\end{align*}
\noindent By assumption it holds $rm<k_{reg}$. Hence, we can choose a real number $\rho>1$ such that $rm\rho<k_{reg}$. We define the dual number $\rho':=\frac{\rho}{1-\rho}\in(1,+\infty)$ and use H\"older's inequality to obtain
\begin{align*}
\mathbb{E}(a_{K,+}^{rm}\|\nabla u_{K,A}\|_{L^{2n}(\mathcal{D})}^{rm})^\frac{1}{rm}\leq \mathbb{E}(a_{K,+}^{rm\rho'})^\frac{1}{rm\rho'}\|\nabla u_{K,A}\|_{L^{rm\rho}(\Omega;L^{2n}(\mathcal{D}))}\leq C(\mathcal{D},r)C_{reg},
\end{align*}
where we again used the fact that $\mathbb{E}(a_{K,+}^{nr\rho'})^{\frac{1}{nr\rho'}}<+\infty$ (see Theorem \ref{TH:QuantificationOfDiffApprLinftynorm}) together with Assumption~\ref{ASS:IntegrabilityOfSolGradient}.
Finally we obtain the estimate
\begin{align*}
\|u_{K,A}-u_{K,A}^{(\varepsilon_W,\varepsilon_l)}\|_{L^r(\Omega;V)}&\leq  C(\overline{a}_-,\mathcal{D},r)C_{reg}\|a_{K,A}^{(\varepsilon_W,\varepsilon_l)}-a_{K,A}\|_{L^{rc}(\Omega;L^{2b}(\mathcal{D}))}\\
	&\leq C_{reg}C(\overline{a}_-,\mathcal{D},r)(\varepsilon_W^\gamma + \varepsilon_l^\frac{1}{rc}),
\end{align*}
where we applied Theorem~\ref{TH:ErrorBoundCoeff} in the last estimate.
\end{proof}

We close this section with a remark on how to choose the parameters $A$, $\varepsilon_W$ and $\varepsilon_l$ to obtain an approximation error smaller than any given threshold $\delta>0$.

\begin{rem}
For any given parameter $K$ large enough (see Remark \ref{rem:approxK}), we choose a positive numer $A>0$ such that the first error contributor satisfies $E_1=\|u_K-u_{K,A}\|_{L^r(\Omega;V)}< \delta/2$ (see Theorem \ref{TH:ErrorCutProblemBoundedByCoeffError} and Theorem \ref{TH:QuantificationOfDiffApprLinftynorm}). Afterwards, under the assumptions of Theorem \ref{TH:ErrorBoundE2}, we may choose the approximation parameters $\varepsilon_W$ and $\varepsilon_l$ small enough, such that the secontd error contributor satisfies $E_2 = \|u_{K,A}-u_{K,A}^{(\varepsilon_W,\varepsilon_l)}\|_{L^r(\Omega;V)}<\delta/2$. Hence, we get an overall error smaller than $\delta$ (see Equation \eqref{EQ:ErrorSplit}).
\end{rem}
	
\section{Pathwise sample-adapted Finite Element approximation}\label{sec:approx_solution}
	
	We want to approximate the solution $u$ to the problem~\eqref{EQ:EllProblem} - \eqref{EQ:EllProblemBCN} with diffusion coefficient $a$ given by Equation~\eqref{EQ:DiffCoeffDefi} using a pathwise Finite Element (FE) approximation of the solution $u_{K,A}^{(\varepsilon_W,\varepsilon_l)}$ of problem~\eqref{EQ:EllProblemCutAppr} - \eqref{EQ:EllProblemBCNCutAppr} where the approximated diffusion coefficient $a_{K,A}^{(\varepsilon_W,\varepsilon_l)}$ is given by \eqref{EQ:DiffCoeffApprDef}. Therefore, for almost all $\omega\in\Omega$, we have to find a function $u_{K,A}^{(\varepsilon_W,\varepsilon_l)}(\omega,\cdot)\in V$ such that it holds
	\begin{align}\label{EQ:VariationalProblemApprSol}
	B_{a_{K,A}^{(\varepsilon_W,\varepsilon_l)}(\omega)} (u_{K,A}^{(\varepsilon_W,\varepsilon_l)}(\omega,\cdot), v) &:=\int_{\mathcal{D}}a_{K,A}^{(\varepsilon_W,\varepsilon_l)}(\omega,\underline{x})\nabla	u_{K,A}^{(\varepsilon_W,\varepsilon_l)}(\omega,\underline{x})\cdot\nabla v(\underline{x})d\underline{x}\\
	&=\int_\mathcal{D}f(\omega,\underline{x})v(\underline{x})d\underline{x} + \int_{\Gamma_2} g(\omega,\underline{x})[Tv](\underline{x})d\underline{x}=:F_\omega(v),\notag
	\end{align}
	for every $v\in V$. Here, $K,A,\varepsilon_W,\varepsilon_l$ are fixed approximation parameters. In order to solve this variational problem numerically we consider a standard Galerkin approach and assume $\mathcal{V}=(V_\ell,~\ell\in\mathbb{N}_0)$ to be a sequence of finite-dimensional subspaces $V_\ell\subset	V$ with $\dim(V_\ell)=d_\ell$ and $V_\ell\subset V_{\ell+1}$ for all $\ell\geq 0$. We denote by $(h_\ell,~\ell\in\mathbb{N}_0)$ the corresponding sequence of refinement sizes which is assumed to decrease monotoncally to zero for $\ell\rightarrow \infty$. Let $\ell\in\mathbb{N}_0$ be fixed and denote by $\{v_1^{(\ell)},\dots,v_{d_\ell}^{(\ell)}\}$ a basis of $V_\ell$. The (pathwise) discrete version of~\eqref{EQ:VariationalProblemApprSol} reads:\\
	
	 Find $u_{K,A,\ell}^{(\varepsilon_W,\varepsilon_l)}(\omega,\cdot)\in V$ such that \begin{align*}
	B_{a_{K,A}^{(\varepsilon_W,\varepsilon_l)}(\omega)}(u_{K,A,\ell}^{(\varepsilon_W,\varepsilon_l)}(\omega,\cdot),v_\ell^{(i)})=F_\omega(v_\ell^{(i)}) \text{ for all } i=1,\dots,d_\ell.
\end{align*}		
We expand the function $u_{K,A,\ell}^{(\varepsilon_W,\varepsilon_l)}(\omega,\cdot)$ with respect to the basis $\{v_1^{(\ell)},\dots,v_{d_\ell}^{(\ell)}\}$:
\begin{align*}
u_{K,A,\ell}^{(\varepsilon_W,\varepsilon_l)}(\omega,\cdot)=\sum_{i=1}^{d_\ell} c_iv_i^{(\ell)},
\end{align*}
where the coefficient vector $\textbf{c}=(c_1,\dots,c_{d_\ell})^T\in\mathbb{R}^{d_\ell}$ is determined by the linear equation system 
\begin{align*}
\textbf{B}(\omega)\textbf{c}=\textbf{F}(\omega),
\end{align*}
with a stochastic stiffness matrix $\textbf{B}(\omega)_{i,j}=B_{a_{K,A}^{(\varepsilon_W,\varepsilon_l)}(\omega)}(v_i^{(\ell)},v_j^{(\ell)})$ and load vector $\mathbf{F}(\omega)_i=F_\omega (v_i^{(\ell)})$ for $i,j=1,\dots,d_\ell$.

	\begin{rem}\label{REM:ConvRateNA}
	Let $(\mathcal{K}_\ell,~\ell\in \mathbb{N}_0)$ be a sequence of triangulations on $\mathcal{D}$ and denote by $\theta_\ell>0$ the minimum interior angle of all triangles in $\mathcal{K}_\ell$. We assume $\theta_\ell\geq \theta>0$ for a positive constant $\theta$ and define the maximum diameter of the triangulation $\mathcal{K}_\ell$ by 
	$h_\ell:=\max\limits_{K\in\mathcal{K}_\ell}\, \text{diam}(K),$
	for $\ell\in\mathbb{N}_0$. Further, we define the finite dimensional subspaces by 
	$V_\ell:=\{v\in V~|~v|_K\in\mathcal{P}_1,K\in \mathcal{K}_\ell\},$
	where $\mathcal{P}_1$ denotes the space of all polynomials up to degree one. If we assume that for $\mathbb{P}-$almost all $\omega\in\Omega$ it holds $u_{K,A}^{(\varepsilon_W,\varepsilon_l)}(\omega,\cdot)\in H^{1+\kappa_a}(\mathcal{D})$ for some positive number $\kappa_a>0$, the pathwise discretization error is bounded by C\'ea's lemma $\mathbb{P}$-a.s. by 
	\begin{align*}
	\|u_{K,A}^{(\varepsilon_W,\varepsilon_l)}(\omega,\cdot)-u_{K,A,\ell}^{(\varepsilon_W,\varepsilon_l)}(\omega,\cdot)\|_V\leq C_{\theta,\mathcal{D}}\frac{A}{\overline{a}_-}\|u_{K,A}^{(\varepsilon_W,\varepsilon_l)}(\omega,\cdot)\|_{H^{1+\kappa_a}(\mathcal{D})}h_\ell^{\min(\kappa_a,1)},
	\end{align*}
	(see \cite[Section 4]{AStudyOfElliptic} and \cite[Chapter 8]{EllipticDifferentialEquations}).
	If the bound $\|u_{K,A}^{(\varepsilon_W,\varepsilon_l)}\|_{L^2(\Omega;H^{1+\kappa_a}(\mathcal{D}))}\leq C_u=C_u(K,A)$ is finite for the fixed approximation parameters $K,A$, we immediately obtain
	\begin{align*}
	\|u_{K,A}^{(\varepsilon_W,\varepsilon_l)}-u_{K,A,\ell}^{(\varepsilon_W,\varepsilon_l)}\|_{L^2(\Omega;V)}\leq C_{\theta,\mathcal{D}} \frac{A}{\overline{a}_-}C_u h_\ell^{\min(\kappa_a,1)}.
	\end{align*}
	\end{rem}
	Note that, for general jump-diffusion problems, one obtains a discretization error of order $\kappa_a \in(1/2,1)$. In general, we cannot expect the full order of convergence $\kappa_a=1$ since the diffusion coefficient is discontinuous. Without special treatment of the interfaces with respect to the triangulation, one cannot expect a convergence rate which is higher than $\kappa_a=1/2$ for the deterministic problem (see \cite{TheFiniteElementMethodForELlipticEquationsWithDiscontinuousCoefficients} and~\cite{AStudyOfElliptic}).\\
	
	\subsection{Sample-adapted triangulations}
	In~\cite{AStudyOfElliptic}, the authors suggest sample-adapted triangulations to improve the convergence rate of the FE approximation: Consider a fixed $\omega\in \Omega$ and assume that the discontinuities of the diffusion coefficient are described by the partition $\mathcal{T}(\omega)=(\mathcal{T}_i,~i=1,\dots, \tau(\omega))$ of the domain $\mathcal{D}$ where $\tau(\omega)$ describes the number of elements in the partition. We consider finite-dimensional subspaces $\hat{V}_\ell(\omega)\subset V$ with (stochastic) dimension $\hat{d}_\ell(\omega)\in\mathbb{N}$. We denote by $\hat{\theta}_\ell(\omega)$ the minimal interior angle within $\mathcal{K}_\ell(\omega)$ and assume the existence of a positive number $\theta>0$ such that $\inf\{\hat{\theta}_\ell(\omega) ~|~ \ell\in\mathbb{N}_0\}\geq \theta$ for $\mathbb{P}$-almost all $\omega\in\Omega$. Assume that $\mathcal{K}_\ell(\omega)$ is a triangulation of $\mathcal{D}$ which is adjusted to the partition $\mathcal{T}(\omega)$ in the sense that for every $i=1,\dots,\tau(\omega)$ it holds
	\begin{align*}
	\mathcal{T}_i\subset \bigcup_{\kappa\in\mathcal{K}_\ell(\omega)}\kappa \text{ and }\hat{h}_\ell(\omega):=\underset{K\in\mathcal{K}_\ell(\omega)}{\max}\, diam(K)\leq \overline{h}_\ell,
	\end{align*}
	for all $\ell\in\mathbb{N}_0$, where $(\overline{h}_\ell,~\ell\in\mathbb{N}_0)$ is a deterministic, decreasing sequence of refinement thresholds which converges to zero (see Figure \ref{FIG:NumExAdaptTri}).

	\begin{figure}[ht]
	\centering
	\subfigure{\includegraphics[scale=0.15]{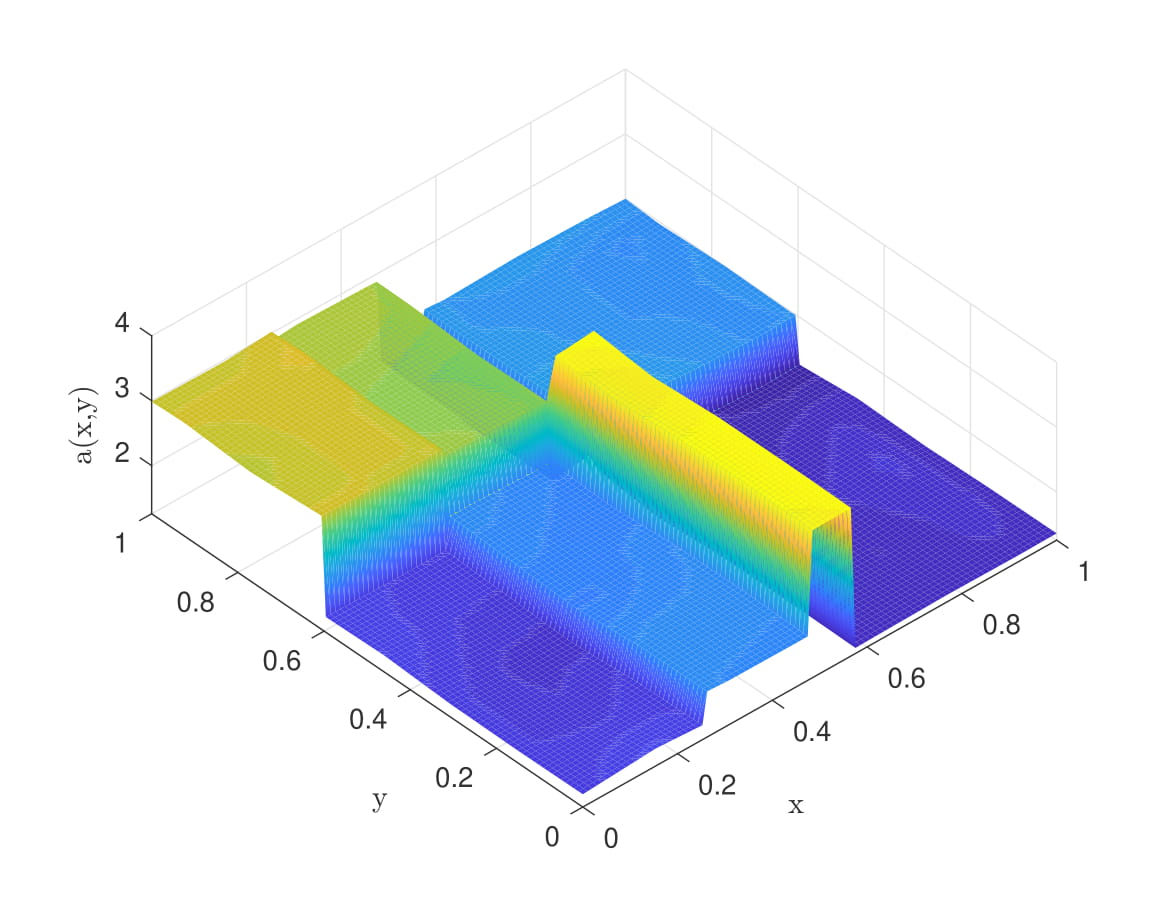}}
	\subfigure{\includegraphics[scale=0.39]{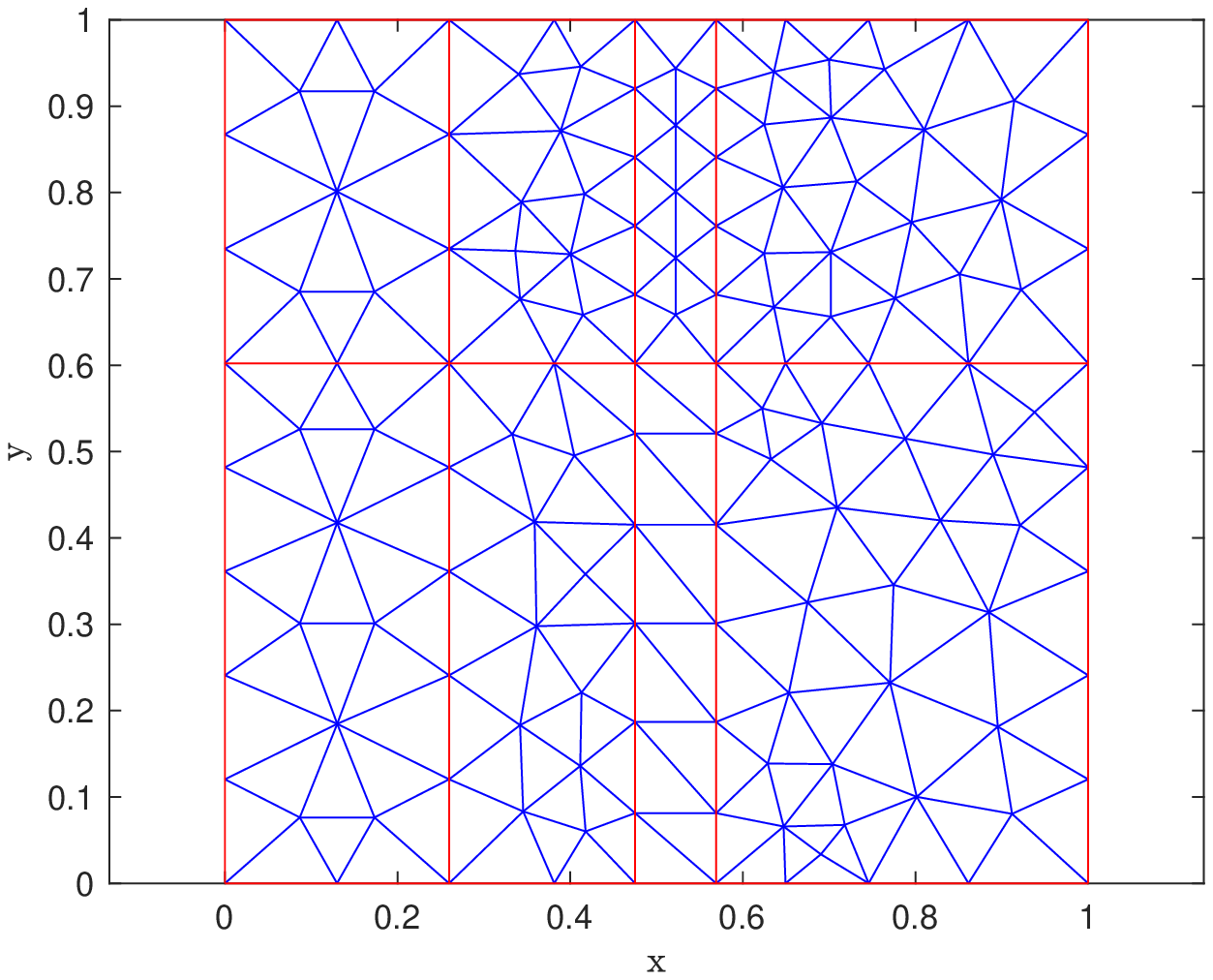}}
    \caption{A sample of a Poisson-subordinated Mat\'ern-1.5-GRF (left) with corresponding sample-adapted triangulations (right).}
    \label{FIG:NumExAdaptTri}
    \end{figure}

	Sample-adapted triangulations lead to an improved  convergence rate for our problem (see~\cite[Section 4.1]{AStudyOfElliptic} and Section~\ref{sec:numerics}).  This observation together with Remark~\ref{REM:ConvRateNA} motivate the following assumption for the rest of this paper (see~\cite[Assumption 4.4]{AStudyOfElliptic}).
	
	\begin{assumption}\label{ASS:ConvRateFEMMeanSquareError}
	There exist deterministic constants $\hat{C}_{u}, C_{u},\hat{\kappa}_a,\kappa_a>0$ such that for any $\varepsilon_W,\varepsilon_l>0$ and any $\ell\in\mathbb{N}_0$, the Finite Element approximation errors of $\hat{u}_{K,A,\ell}^{(\varepsilon_W,\varepsilon_l)}\approx u_{K,A}^{(\varepsilon_W,\varepsilon_l)}$ in the (sample-adapted) subspaces $\hat{V}_\ell$, respectively $u_{K,A,\ell}^{(\varepsilon_W,\varepsilon_l)}\approx u_{K,A}^{(\varepsilon_W,\varepsilon_l)}$ in $V_\ell$, are bounded by
	\begin{align*}
	\|u_{K,A}^{(\varepsilon_W,\varepsilon_l)} - \hat{u}_{K,A,\ell}^{(\varepsilon_W,\varepsilon_l)}\|_{L^2(\Omega;V)}&\leq \hat{C}_{u} \mathbb{E}(\hat{h}_\ell^{2\hat{\kappa}_a})^{1/2} , \text{ respectively,}\\
	\|u_{K,A}^{(\varepsilon_W,\varepsilon_l)} - u_{K,A,\ell}^{(\varepsilon_W,\varepsilon_l)}\|_{L^2(\Omega;V)}&\leq C_{u}h_\ell^{\kappa_a},
	\end{align*}
	where the constants $\hat{C}_{u},C_{u}$ may depend on $a,f,g,K,A$ but are independent of $\hat{h}_\ell,h_\ell,\hat{\kappa}_a$ and $\kappa_a$.
	\end{assumption}

	In practice, as we can see in the numerical examples in the subsequent section, one can (at least) recover the deterministic rates in the strong error. In fact, in the non-adapted case it is possible to get better convergence rates than expected for some examples. By construction of our random field, we always obtain an interface geometry with fixed angles and bounded jump height, which have great influence on the solution regularity, see e.g.~\cite{RegularityResultsForLaplaceInterfaceProblemsInTroDimensions}.

\section{Numerical examples}\label{sec:numerics}
In this section we verify our theoretical results in numerical examples. In all experiments we work on the domain $\mathcal{D}=(0,1)^2$ and use a FE method with hat-function basis. Here, we distinguish between the \textit{standard} FEM approach and the \textit{sample-adapted} FEM approach introduced in Section \ref{sec:approx_solution}.  We compare both and investigate how different L\'evy subordinators influence the strong convergence rate. In our first example we use Poisson processes with low intensity to investigate the superiority of the presented sample-adapted triangulation. In the second example we use Poisson subordinators with a significantly higher intensity. Besides Poisson subordinators we also use Gamma processes which have infinite activity.

\subsection{Strong error approximation}

In each numerical experiment we choose a problem dependent cut-off level $K$ for the subordinators in \eqref{EQ:DiffCoeffDefiCut} large enough so that its influence is negligibly (see Remark \ref{rem:approxK}).  Further, we choose the cut-off level for the diffusion coefficient $A$ in \eqref{EQ:DiffCoeffDefiCutUpperCut} large enough such that it has no influence in numerical experiments, $e.g.~ A=50$ and therefore the error induced by the error contributor $E_1$ in~\eqref{EQ:ErrorSplit} can be neglected in our experiments. We estimate the strong error using a standard Monte Carlo estimator. Assume that a sequence of (sample-adapted) finite-dimensional subspaces $(\hat{V}_\ell,~\ell \in \mathbb{N}_0)\subset V$ is given where we use the notation of Section \ref{sec:approx_solution}. For readability we only treat the case of pathwise sample-adapted Finite Element approximations in the rest of the theoretical consideration in this subsection. We would like to point out, however, that similar arguments lead to the corresponding results for standard FE approximations.

Under the assumptions of Theorem \ref{TH:ErrorBoundE2} and Assumption \ref{ASS:ConvRateFEMMeanSquareError} we obtain
\begin{align}\label{EQ:StrongErrSplit}
\|u_{K,A} - \hat{u}_{K,A,\ell}^{(\varepsilon_W,\varepsilon_l)}\|_{L^2(\Omega;V)} &\leq \|u_{K,A} - u_{K,A}^{(\varepsilon_W , \varepsilon_l)}\|_{L^2(\Omega;V)} + \|u_{K,A}^{(\varepsilon_W , \varepsilon_l)} - \hat{u}_{K,A,\ell}^{(\varepsilon_W , \varepsilon_l)}\|_{L^2(\Omega;V)}\\
&\leq C ( \varepsilon_W^\gamma + \varepsilon_l^\frac{1}{rc} + \mathbb{E}(\hat{h}_\ell^{2\hat{\kappa}_a})^{1/2}),\notag
\end{align}
with a constant $C=C(C_{reg},D,\overline{a}_-,\hat{C}_u)$. Therefore, in order to equilibrate all error contributions, we choose the approximation parameters $\varepsilon_W$ and $\varepsilon_l$ in the following way:
\begin{align}\label{EQ:RuleApprParams}
\varepsilon_W \simeq \mathbb{E}(\hat{h}_\ell^{2\hat{\kappa}_a})^{1/(2\gamma)} \text{ and } \varepsilon_l \simeq \mathbb{E}(\hat{h}_\ell^{2\hat{\kappa}_a})^{rc/2}.
\end{align}
For readability, we omit the cut-off parameters $K$ and $A$ in the following and use the notation $\hat{u}_{\ell,\varepsilon_W,\varepsilon_l}=\hat{u}_{K,A,\ell}^{(\varepsilon_W,\varepsilon_l)}$. Choosing the approximation parameters $\varepsilon_W, \varepsilon_l$ according to \eqref{EQ:RuleApprParams}, we can investigate the strong error convergence rate by a Monte Carlo estimation of the left hand side of \eqref{EQ:StrongErrSplit}: for a fixed natural number $M\in \mathbb{N}$ we approximate
\begin{align} \label{EQ:MCStrongErr}
\|u_{K,A} - \hat{u}_{K,A,\ell}^{(\varepsilon_W,\varepsilon_l)}\|_{L^2(\Omega;V)}^2 = \|u_{K,A} - \hat{u}_{\ell,\varepsilon_W,\varepsilon_l}\|_{L^2(\Omega;V)} ^2\approx \frac{1}{M}\sum_{i=1}^M \|u_{ref}^{(i)} - \hat{u}_{\ell,\varepsilon_W,\varepsilon_l}^{(i)}\|_V^2,
\end{align}
where $(u_{ref}^{(i)},~i=1,\dots , M)$ are i.i.d. realizations of the stochastic reference solution $u_{ref}\approx u_{K,A}$ and $(\hat{u}_{\ell,\varepsilon_W,\varepsilon_l}^{ (i)}, ~i=1\dots,M)$ are i.i.d. realizations of the FE approximation $\hat{u}_{\ell,\varepsilon_W,\varepsilon_l}$ of the PDE solution on the FE subspace $\hat{V}_\ell$. In all examples we choose the sample number $M$ so that the standard deviation of the MC samples is smaller than $10\%$ of the MC estimator itself.
\subsection{PDE Parameters}
In all of our numerical examples we choose $\overline{a}\equiv 1/10$, $f\equiv 10$, $\Phi_1=1/100\,\exp(\cdot)$ and $\Phi_2=5\,|\cdot|$. Further, we impose mixed Dirichlet-Neumann boundary conditions if nothing else is explicitly mentioned. To be precise, we split the domain boundary $\partial\mathcal{D}$ by $\Gamma_1=\{0,1\}\times[0,1]$ and $\Gamma_2=(0,1)\times\{0,1\}$ and impose the pathwise mixed Dirichlet-Neumann boundary conditions
\begin{align*}
u_{K,A}(\omega,\cdot)=\begin{cases} 0.1  & ~on~ \{0\}\times [0,1] \\ 0.3 &~on~\{1\}\times [0,1]  \end{cases} \text{ and } a_{K,A}\overrightarrow{n}\nabla \cdot u_{K,A}=0 \text{ on } \Gamma_2,
\end{align*}
for $\omega\in\Omega$.

We choose $W_1$ to be a Mat\'ern-1.5-GRF on $\mathcal{D}$ with correlation length $r_1=0.5$ and different variance parameters $\sigma_1^2$. Further, we set  $W_2$ to be a Mat\'ern-1.5-GRF on $[0,K]^2$ which is independent of $W_1$ with different variances $\sigma_2^2$ and correlation lengths $r_2$. We use a reference grid with $800\times 800$ equally spaced points on the domain $\mathcal{D}$ for interpolation and prolongation.
\subsection{Poisson subordinators}
In this section we use Poisson processes to subordinate the GRF $W_2$ in the diffusion coefficient in \eqref{EQ:DiffCoeffDefi}. We consider both, high and low intensity Poisson processes and vary the boundary conditions. Further, using Poisson subordinators allows for a detailed investigation of the approximation error caused by approximating the L\'evy subordinators $l_1$ and $l_2$ according to Assumption \ref{ASS:CutProblemEigenvalues} \textit{v}.
\subsubsection{The two approximation methods}
\label{Subsubsec:TheTwoApprMethods}

Using Poisson processes as subordinators allows for two different simulation approaches in the numerical examples: the first approach is an exact and grid-independent simulation of a Poisson process  using for example the \textit{Method of Exponential Spacings} or the \textit{Uniform Method} (see \cite[Section 8.1.2]{LevyProcessesInFinance}). On the other hand, one may also work with approximations of the Poisson processes satisfying Assumption \ref{ASS:CutProblemEigenvalues} \textit{v}.

We sample the values of the Poisson($\lambda$)-processes $l_1$ and $l_2$ on an equidistant grid $\{x_i,~i=0,...,N_l\}$ with $x_0=0$ and $x_{N_l}=1$ and step size $|x_{i+1}-x_i|\leq \varepsilon_l\leq 1$ for all $i=0,\dots,N_l-1$. Further, we approximate the stochastic processes by a piecewise constant extension $l_j^{(\varepsilon_l)}\approx l_j$ of the values on the grid:
\begin{align*}
l_j^{(\varepsilon_l)}(x)=\begin{cases} l_j(x_i) & x\in[x_i,x_{i+1}) \text{ for } i=0,...,N_l-1,  \\ l_j(x_{N_l-1}) & x=1.  \end{cases}
\end{align*}
for $j=1,2$. Since the Poisson process has independent, Poisson distributed increments, values of the Poisson process at the discrete points $\{x_i,~i=0,\dots,N_l\}$ can be generated by adding independent Poisson distributed random variables. In the following we refer to this approach as the \textit{approximation approach} to simulate a Poisson process. Note that in this case Assumption \ref{ASS:CutProblemEigenvalues} \textit{v} holds with $\eta =+\infty$. In fact, for any $s\in[1,+\infty)$ we obtain for $j=1,2$ and an arbitrary $x\in[0,1)$ with $x\in[x_i,x_{i+1})$:
\begin{align*}
\mathbb{E}(|l_j(x)-l_j^{(\varepsilon_l)}(x)|^s)=\mathbb{E}(|l_j(x)-l_j(x_i)|^s)\leq \mathbb{E}(|l_j(x_{i+1}-x_i)|^s)\leq\mathbb{E}(|l_j(\varepsilon_l)|^s),
\end{align*}
which is independent of the specific $x\in[0,1)$. Note that this also holds for $x=D=1$ and therefore
\begin{align*}
\underset{x\in[0,1]}{\sup}\,\mathbb{E}(|l_j(x)-l_j^{(\varepsilon_l)}(x)|^s)\leq \mathbb{E}(|l_j(\varepsilon_l)|^s).
\end{align*}
For a Poisson process with parameter $\lambda$ we obtain 
\begin{align*}
\mathbb{E}(|l_j(\varepsilon_l)|^s) = e^{-\lambda\varepsilon_l}\sum_{k=0}^\infty k^s \frac{(\lambda\varepsilon_l)^k}{k!}\leq\varepsilon_l\sum_{k=1}^\infty k^s \frac{\lambda^k}{k!}\leq C_l\varepsilon_l,
\end{align*}
where the series converges by the ratio test.

Since the Poisson process allows for both approaches - approximation and exact simulation of the process - the use of these processes are suitable to investigate the additional error in the approximation of the PDE solution resulting from an approximation of the subordinators.

\subsubsection{Poisson subordinators: low intensity and mixed boundary conditions}
\label{subsubsec:NumexPoiss1Mixed}

In this example we choose $l_1$ and $l_2$ to be Poisson($1$)-subordinators. Further, the variance parameter of the GRF $W_1$ is set to be $\sigma_1=1.5$ and the variance and correlation parameters of the GRF $W_2$ are given by $\sigma_2=0.3$ and $r_2=1$.

For independent Poisson($1$)-subordinators $l_1$ and $l_2$  we choose $K=8$ as the cut-off parameter (see~\eqref{EQ:DiffCoeffDefiCut}). With this choice we obtain
\begin{align*}
\mathbb{P}(\underset{t\in[0,1]}{\sup}\, l_j(t)\geq K)=\mathbb{P}(l_j(1)\geq K)\approx 1.1252e^{-06},
\end{align*}
for $j=1,2$, such that this cut-off has no influence in the numerical example.
Note that for Mat\'ern-1.5-GRFs we can expect $\gamma=1$ in Equation~\eqref{EQ:MeanHoelderContGRFs} (see~\cite[Chapter 5]{AMultilevelMonteCarloAlgorithmForParabolicAdvectionDiffusionProblemsWithDiscontinuousCoefficients},~\cite[Proposition 9]{QuasiMonteCarloFEMethodsForEllipticPDEsWithLognormalRandomCoefficients}).

We approximate the GRFs $W_1$ and $W_2$ by the circulant embedding method (see~ \cite{AnalysisOfCirculantEmbeddingMethodsForSamplingStationaryRandomFields} and \cite{CirculantEmbeddingWithWMCAnalysisForEllipicPDEWithLognormalCoefficients}) to obtain approximations $W_1^{(\varepsilon_W)}\approx W_1$ and $W_2^{(\varepsilon_W)}\approx W_2$ as in Lemma~\ref{LE:StrongErrorBoundGRFApprox}. 
Since $\eta=+\infty$ and $f\in L^q(\Omega;H)$ for every $q\geq 1$ we choose for any positive $\delta>0$
\begin{align*}
r=2, ~c=b=1+\delta 
\end{align*}
to obtain from Theorem \ref{TH:ErrorBoundE2}
\begin{align*}
\|u_{K,A}-u_{K,A}^{(\varepsilon_W,\varepsilon_l)}\|_{L^2(\Omega;V)}\leq C_{reg}\frac{C(D)}{\overline{a}_-} (\varepsilon_W + \varepsilon_l^ \frac{1}{2c}),
\end{align*}
where we have to assume that $j_{reg}\geq 2((1+\delta)/\delta-1)$ and $k_{reg}\geq 2(1+\delta)/\delta$ for the regularity constants $j_{reg},k_{reg}$ given in Assumption \ref{ASS:IntegrabilityOfSolGradient}. For $\delta=0.05$ we obtain 
\begin{align}\label{EQ:ErrorBoundNumEx1}
\|u_{K,A}-u_{K,A}^{(\varepsilon_W,\varepsilon_l)}\|_{L^2(\Omega;V)}\leq C_{reg}\frac{C(D)}{\overline{a}_-} (\varepsilon_W + \varepsilon_l^ \frac{1}{2.01}).
\end{align}
Therefore, we get $\gamma=1$ and $rc=2.01$ in the equilibration formula \eqref{EQ:RuleApprParams}.

Figure~\ref{fig:poisson1_samples} shows three different samples of the diffusion coefficient and the corresponding FE approximations of the PDE solution.

	\begin{figure}[ht]
	\centering
	\subfigure{\includegraphics[scale=0.14]{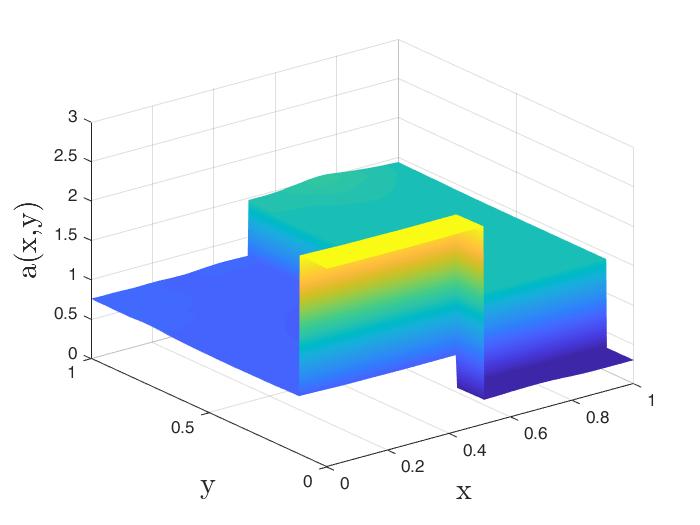}}
	\subfigure{\includegraphics[scale=0.14]{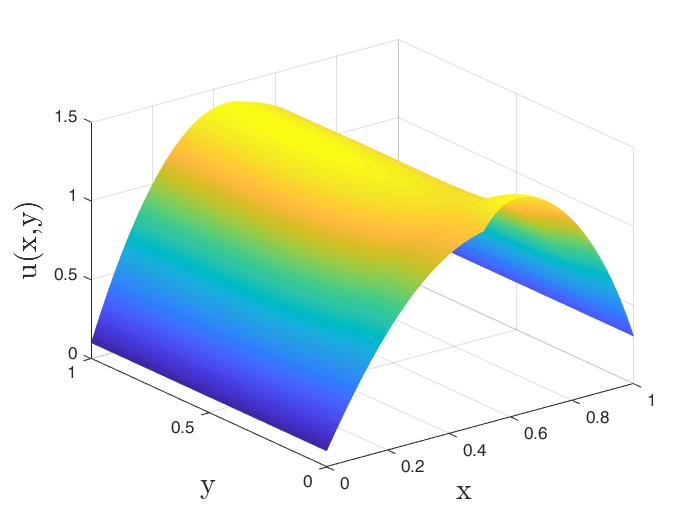}}
	\subfigure{\includegraphics[scale=0.14]{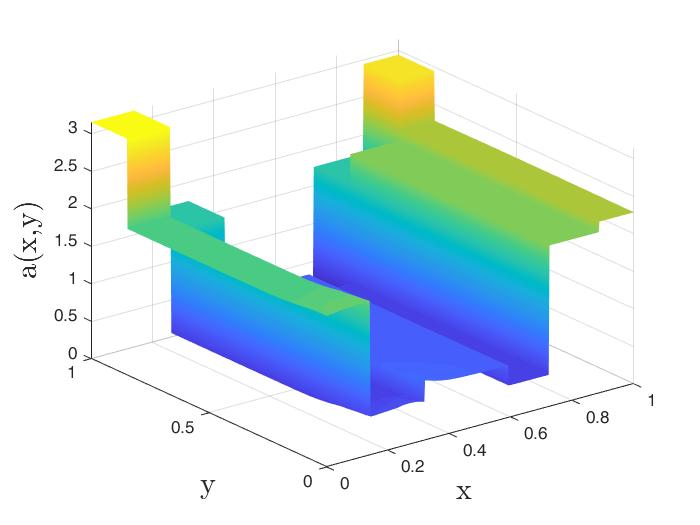}}
	\subfigure{\includegraphics[scale=0.14]{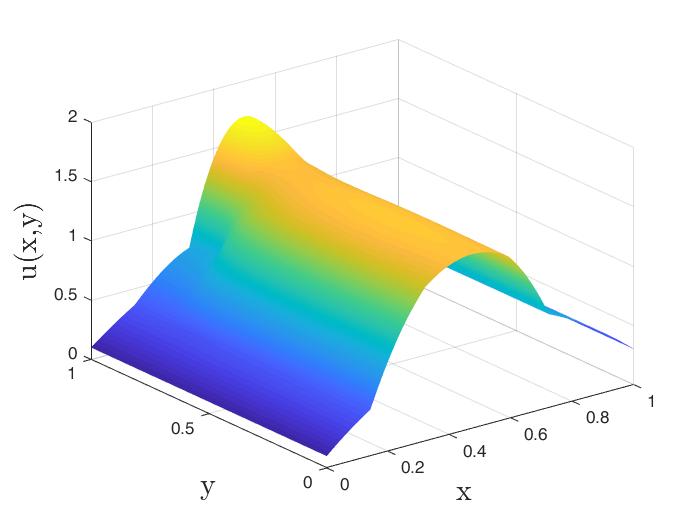}}
    \caption{Different samples of the diffusion coefficient with Poisson($1$)-subordinators and the corresponding PDE solutions with mixed Dirichlet-Neumann boundary conditions.}\label{fig:poisson1_samples}
    \end{figure}
The FE discretization parameters are given by $h_\ell = 0.4\cdot 2^{-(\ell-1)}$ for $l=1,...,7$. We set $u_{ref}=\hat{u}_{7,\varepsilon_W,\varepsilon_l}$, where the approximation parameters $\varepsilon_W$ and $\varepsilon_l$ are choosen according to \eqref{EQ:RuleApprParams} and we compute $M=100$ samples to estimate the strong error by the Monte Carlo estimator (see \eqref{EQ:MCStrongErr}). In this experiment we investigate the strong error convergence rate for the sample-adapted FE approach as well as convergence rate for the non-adapted FE approach (see Section \ref{sec:approx_solution}). In  Subsection \ref{Subsubsec:TheTwoApprMethods} we described two approaches to simulate Poisson subordinators. We run this experiment with both approaches: first, we approximate the Poisson process via sampling on an equidistant level-dependent grid and, in a second run of the experiment, we simulate the Poisson subordinators exactly using the Uniform Method described in \cite[Section 8.1.2]{LevyProcessesInFinance}.
 The convergence results for the both approaches for this experiment are given in the Figure~\ref{Fig:NumExPoiss1MixedConv}.

 	\begin{figure}[ht]
	\centering
	\subfigure{\includegraphics[scale=0.49]{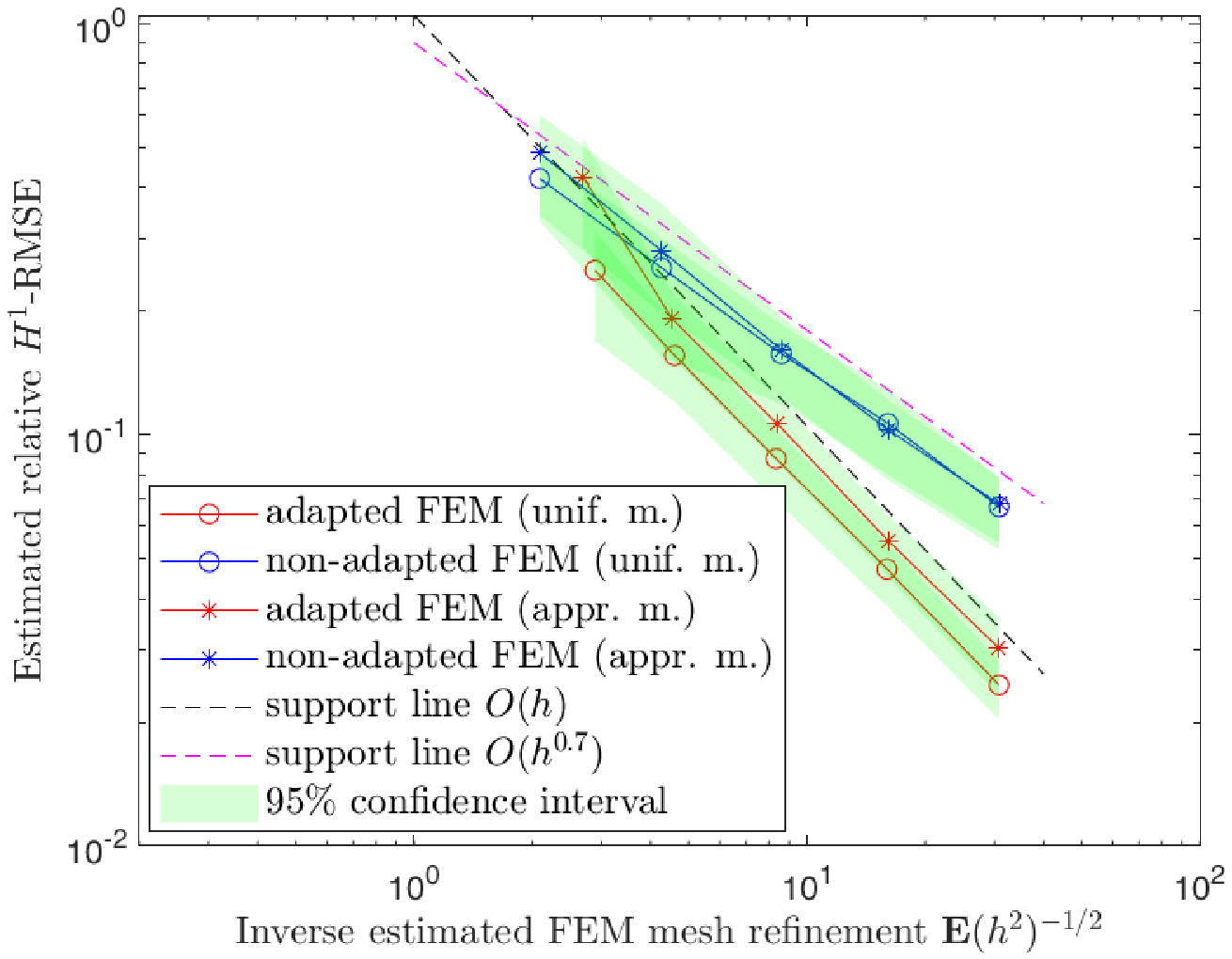}}
	\subfigure{\includegraphics[scale=0.49]{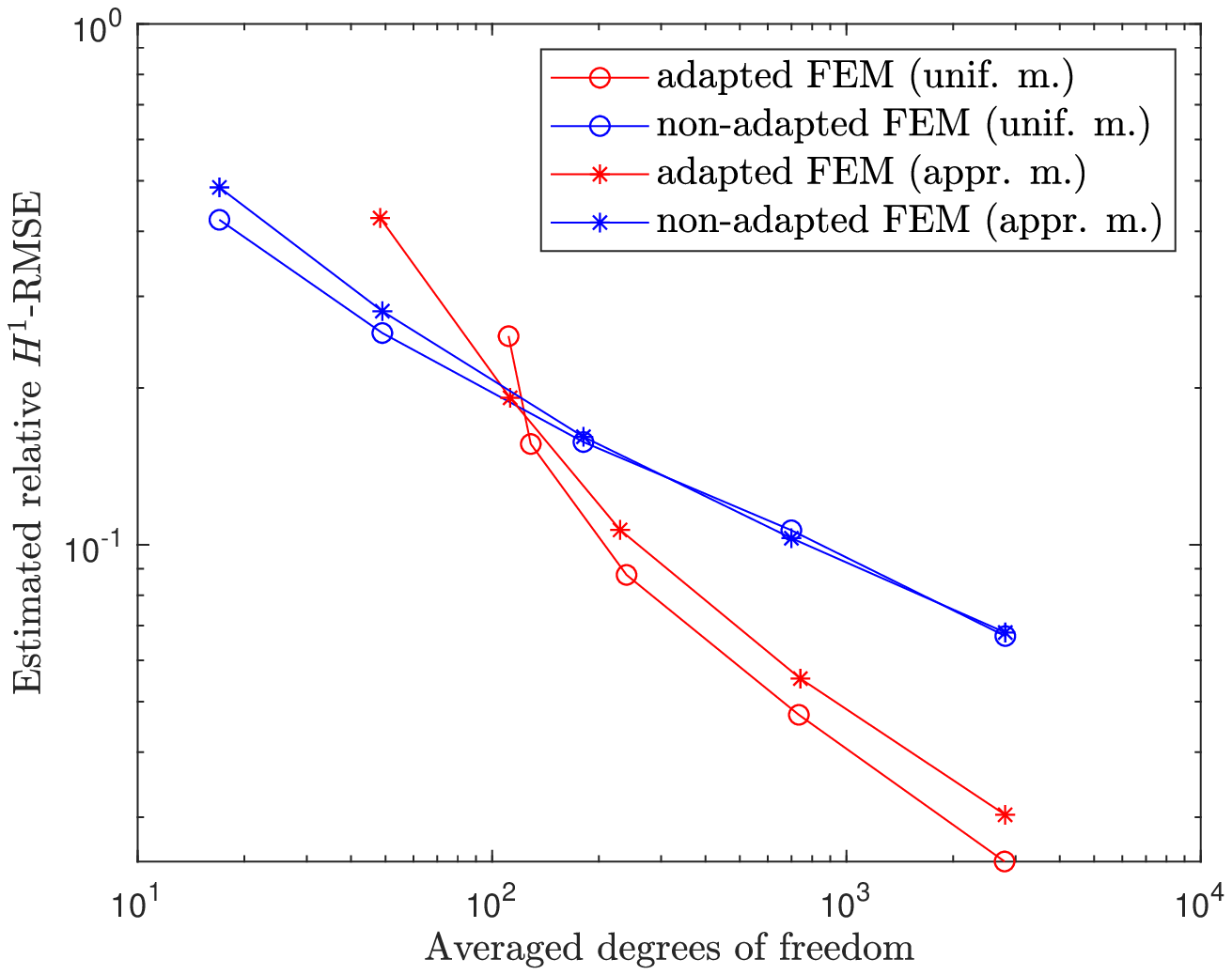}}
     \caption{Convergence results for Poisson($1$)-subordinators using the approximation approach and the Uniform Method with mixed Dirichlet-Neumann boundary conditions.}\label{Fig:NumExPoiss1MixedConv}
    \end{figure}

We see a convergence rate of approximately $0.7$ for the standard FEM discretization and full order convergence ($\hat{\kappa}_a\approx 1$) for the sample-adapted approach. On the right hand side of Figure \ref{Fig:NumExPoiss1MixedConv} on sees that the sample-adapted approach is more efficient in terms of computational effort if we consider the error-to-(averaged)DOF-plot. Only on the first level the standard FEM approach seems to be more efficient (pre-asymptotic behaviour).
If we compare the results for the approximation method with the Uniform Method (see \ref{Subsubsec:TheTwoApprMethods}), we find that, while the convergence rates are the same, the constant of the error in the sample-adapted approach is slightly smaller for the Uniform Method. This shift is exactly the additional error resulting from an approximation of the subordinators in the approximation approach. We also see that, compared to the approximation approach, on the lower levels the averaged degrees of freedom in the sample-adapted FEM approach is slightly higher if we simulate the Poisson subordinators exactly. This is caused by the fact that in this case we do not approximate the discontinuities of the field which are generated by the Poisson processes. This results in a higher average number of degrees on freedom on the lower levels because discontinuities are more likely close to each other.  
\subsubsection{Poisson subordinators: low intensity and homogeneous Dirichlet boundary conditions}
\label{subsubsec:NumExPoiss1Dir}	
Next, we consider the elliptic PDE under homogeneous Dirichlet boundary conditions. All other parameters remain as in Subsection \ref{subsubsec:NumexPoiss1Mixed}. Figure~\ref{fig:sol_Dirichlet} shows samples of the diffusion coefficient and the corresponding FE approximation of the PDE solution.

	\begin{figure}[ht]
	\centering
	\subfigure{\includegraphics[scale=0.14]{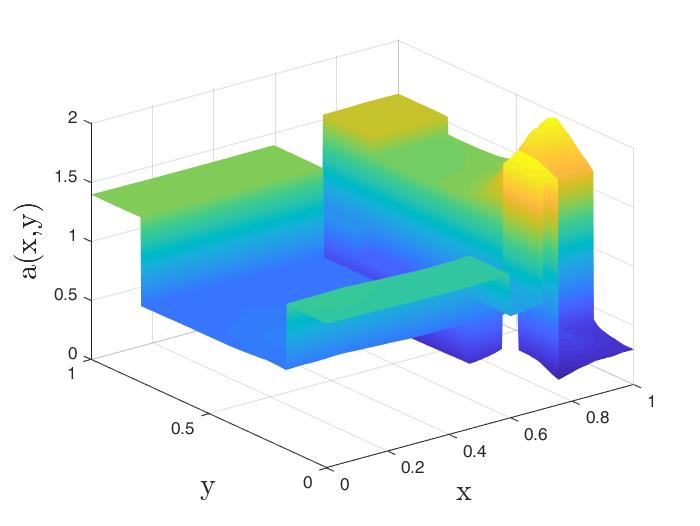}}
	\subfigure{\includegraphics[scale=0.14]{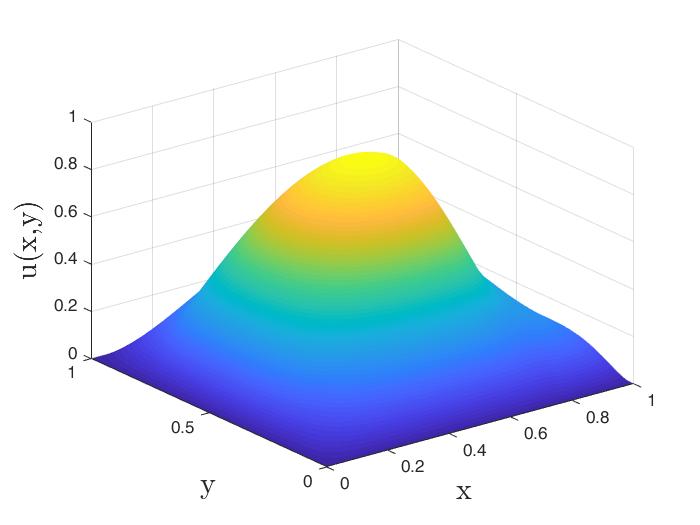}}
	\subfigure{\includegraphics[scale=0.14]{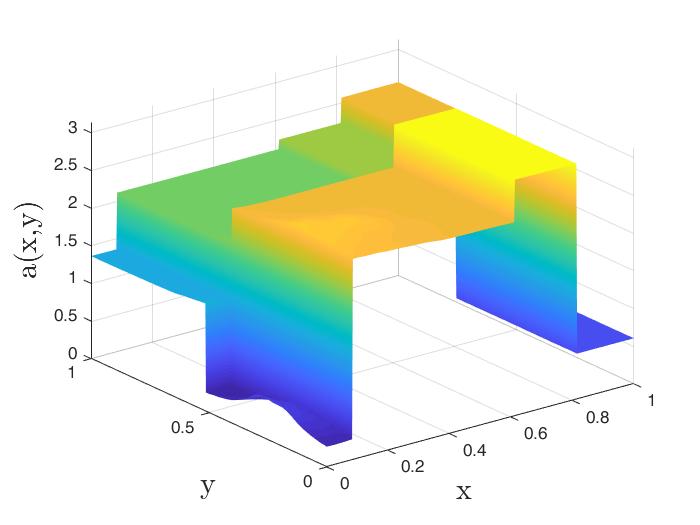}}
	\subfigure{\includegraphics[scale=0.14]{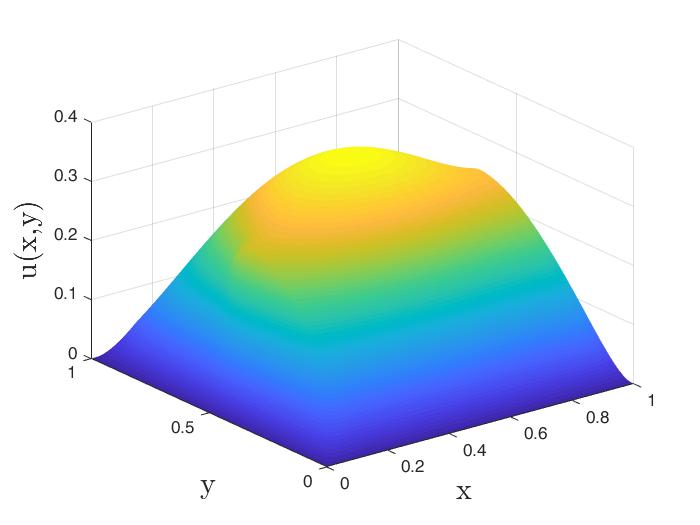}}
    \caption{Different samples of the diffusion coefficient with Poisson($1$)-subordinators and the corresponding PDE solutions with homogeneous Dirichlet boundary conditions.}\label{fig:sol_Dirichlet}
    \end{figure}
We estimate the strong error convergence rate for this problem in the same way as in the previous example using $M=250$ samples and we use the approximation approach to simulate the Poisson subordinators (see Subsection \ref{Subsubsec:TheTwoApprMethods}). Convergence results are given in Figure~\ref{FIG:NuMExPoiss1Dir}.

  	\begin{figure}[ht]
	\centering
	\subfigure{\includegraphics[scale=0.49]{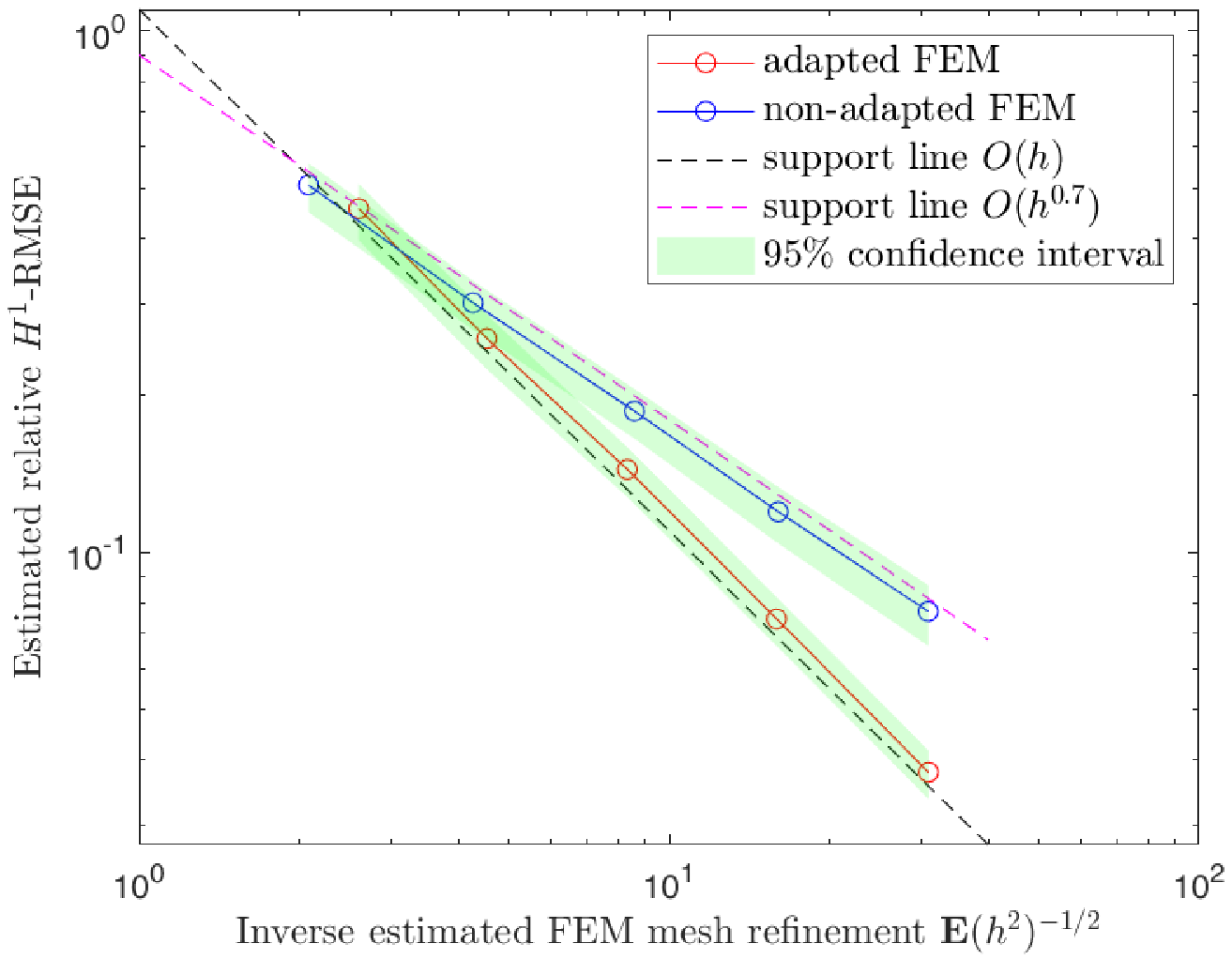}}
	\subfigure{\includegraphics[scale=0.49]{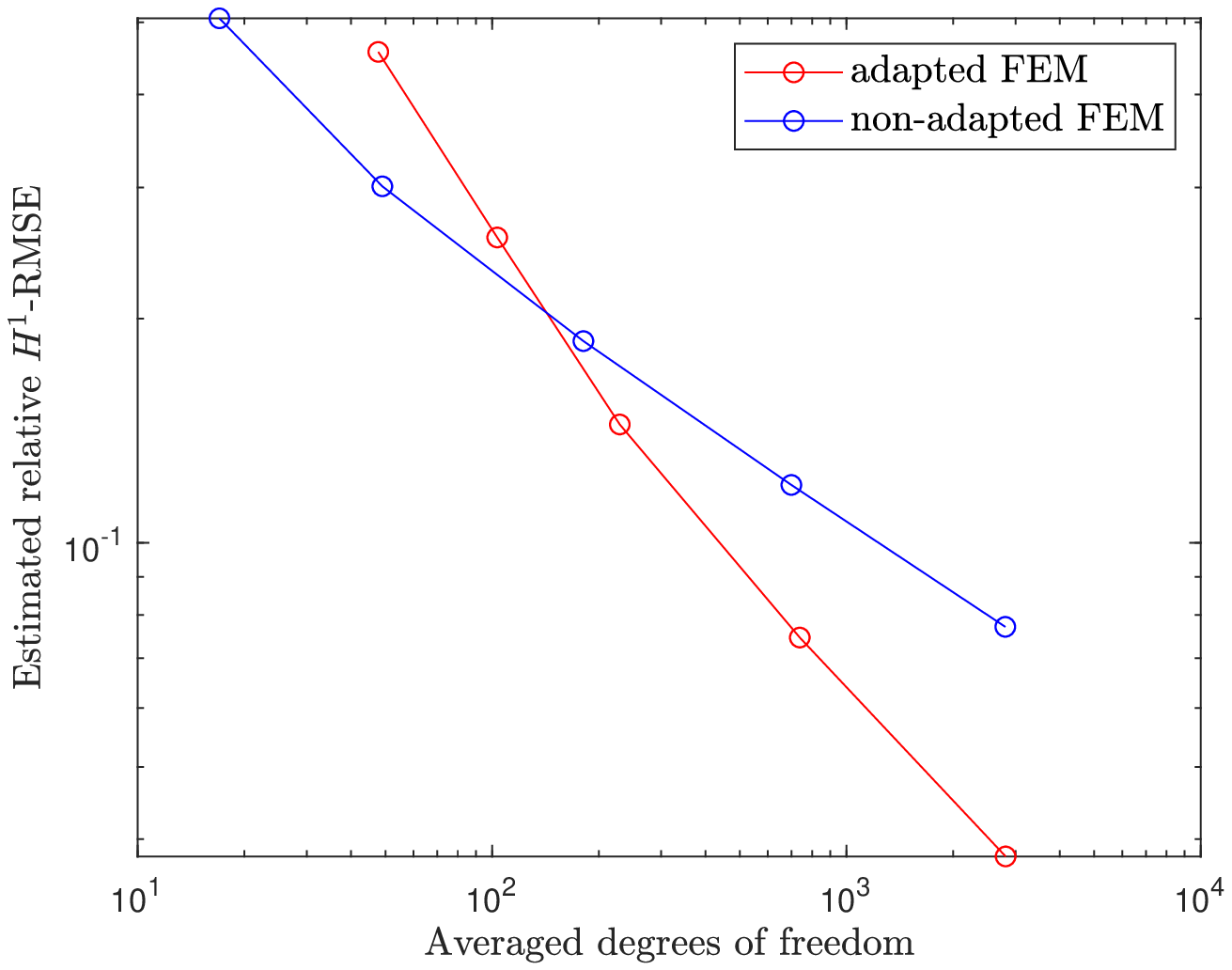}}
    \caption{Convergence results for Poisson($1$)-subordinators using the approximation approach and homogeneous Dirichlet boundary conditions.}
    \label{FIG:NuMExPoiss1Dir}
    \end{figure}

As in the experiment with mixed Dirichlet-Neumann boundary conditions we obtain convergence order of $\kappa_a\approx 0.7$ for the standard FEM approach and full order convergence for the sample-adapted approach. Also in case of homogeneous Dirichlet boundary conditions the sample-adapted FEM is more efficient in terms of the averaged number of degrees of freedom.

\subsubsection{Poisson subordinators: high intensity and mixed boundary conditions}		
\label{subsubsec:NumExPoiss5Smoothgrf}
In this section we want to consider subordinators with higher intensity, resulting in a higher number of discontinuities in the diffusion coefficient. Therefore, we consider $l_1$ and $l_2$ to be Poisson($5$) processes.

We set the cut-off level $K$ of the subordinators in Equation \eqref{EQ:DiffCoeffDefiCut} to $K=15$. For this choice it is reasonable to expect that this cut-off has no numerical influence since
\begin{align*}
 \mathbb{P}(\underset{t\in[0,1]}{\sup} l_j(t) \geq 15) = \mathbb{P}(l_j(1)\geq 15)\approx 6.9008e^{- 05},
\end{align*}
for $j=1,2$. However, setting $K=15$ means that we have to simulate the GRF $W_2$ on the domain $[0,15]^2$ which would be time consuming. Therefore, we set $K=1$ instead and consider the downscaled processes
\begin{align*}
\tilde{l}_j(t) = \frac{1}{15}l_j(t),
\end{align*}
for $t\in [0,1]$ and $j=1,2$. The variance parameter of the field $W_1$ is chosen to be $\sigma_1=1$ and the parameters of the GRF $W_2$ are set to be $\sigma_2=0.3$ and $r_2=0.5$. Figure~\ref{FIG:SamplesPoiss5SmoothGRF} shows samples of the coefficient and the corresponding pathwise FEM solution.

	\begin{figure}[ht]
	\centering
	\subfigure{\includegraphics[scale=0.14]{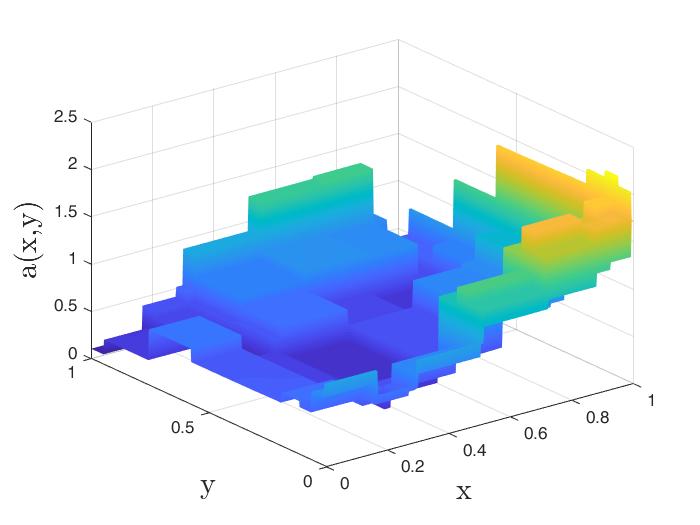}}
	\subfigure{\includegraphics[scale=0.14]{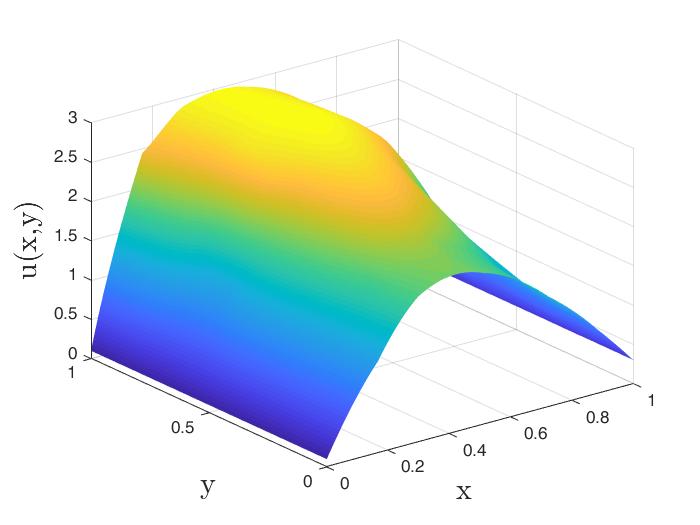}}
	\subfigure{\includegraphics[scale=0.14]{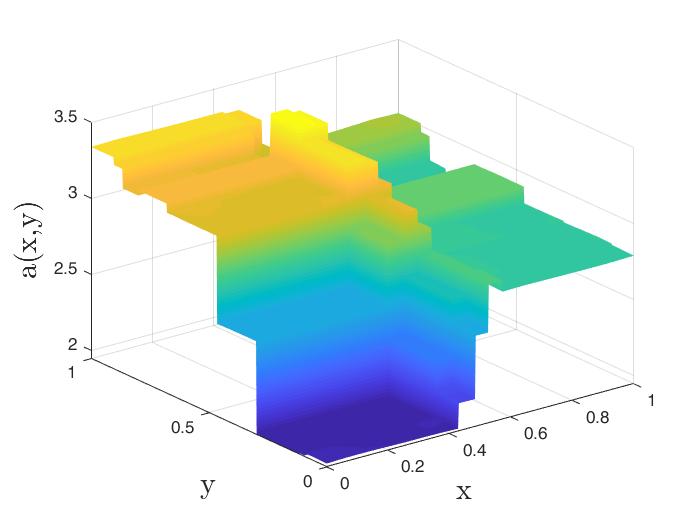}}
	\subfigure{\includegraphics[scale=0.14]{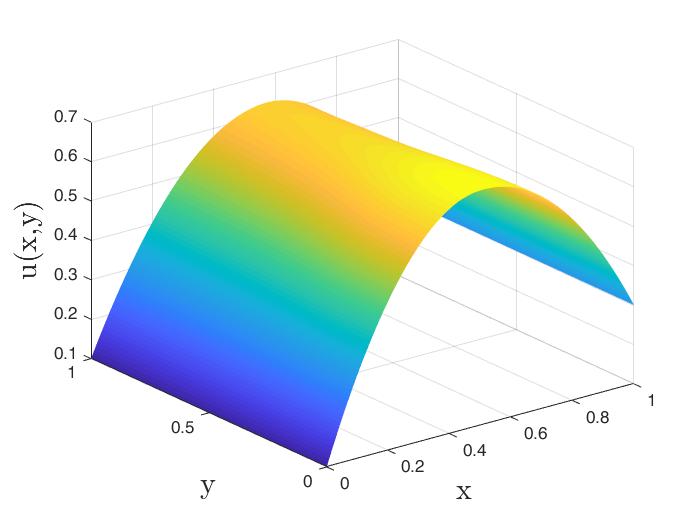}}
    \caption{Different samples of the diffusion coefficient with Poisson($5$)-subordinators (top) and the corresponding PDE solutions with mixed Dirichlet-Neumann boundary conditions (bottom).}
    \label{FIG:SamplesPoiss5SmoothGRF}
    \end{figure}

As in the first experiment, we again run this experiment using both methods described in Subsection \ref{Subsubsec:TheTwoApprMethods}: the approximation approach using Poisson-distributed increments and the Uniform Method. We use the discretization steps $h_\ell = 0.1\cdot 1.7^{-(\ell-1)}$ for $\ell = 1,\dots,7$ and $M=150$ samples.

	\begin{figure}[ht]
	\centering
	\subfigure{\includegraphics[scale=0.49]{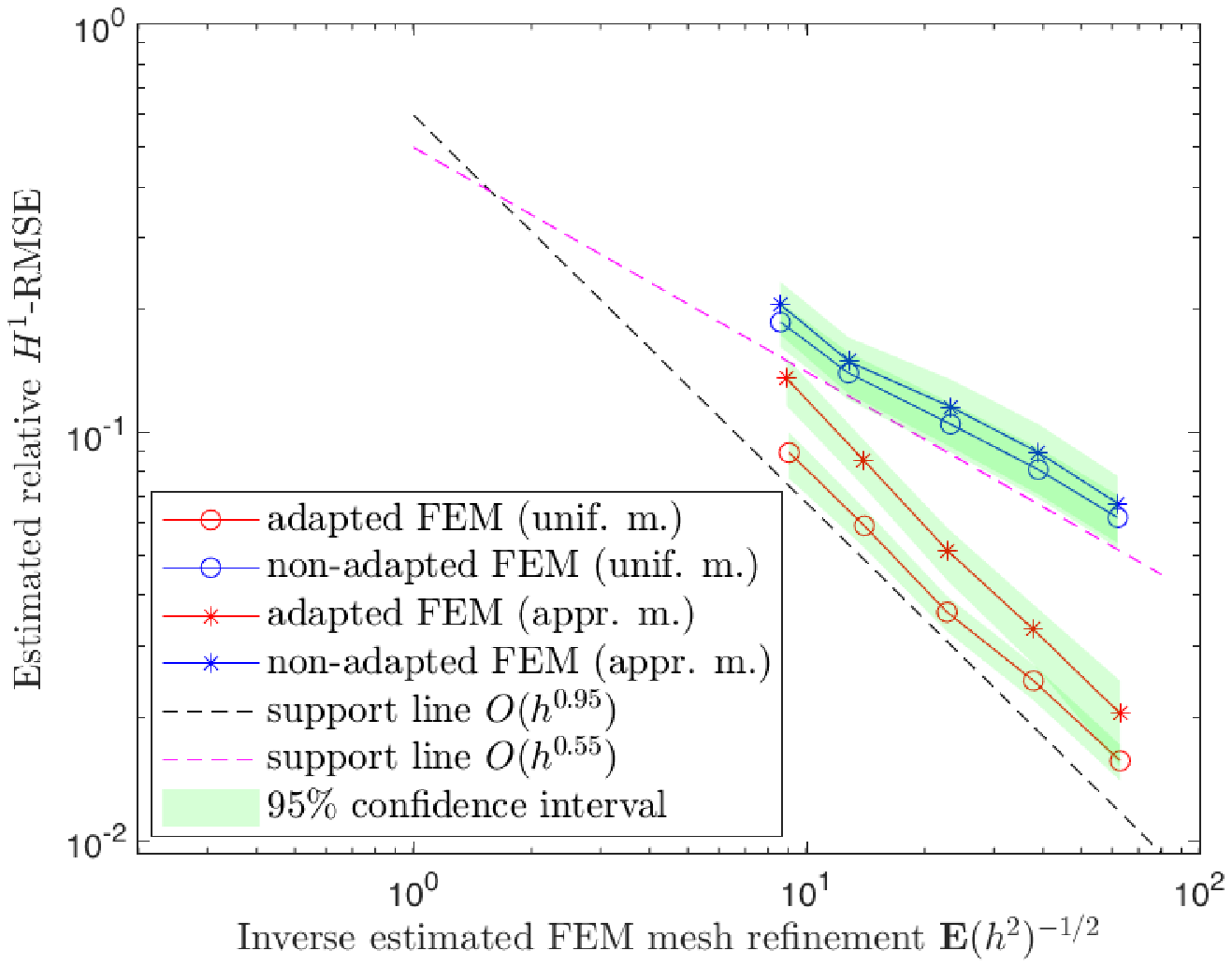}}
	\subfigure{\includegraphics[scale=0.49]{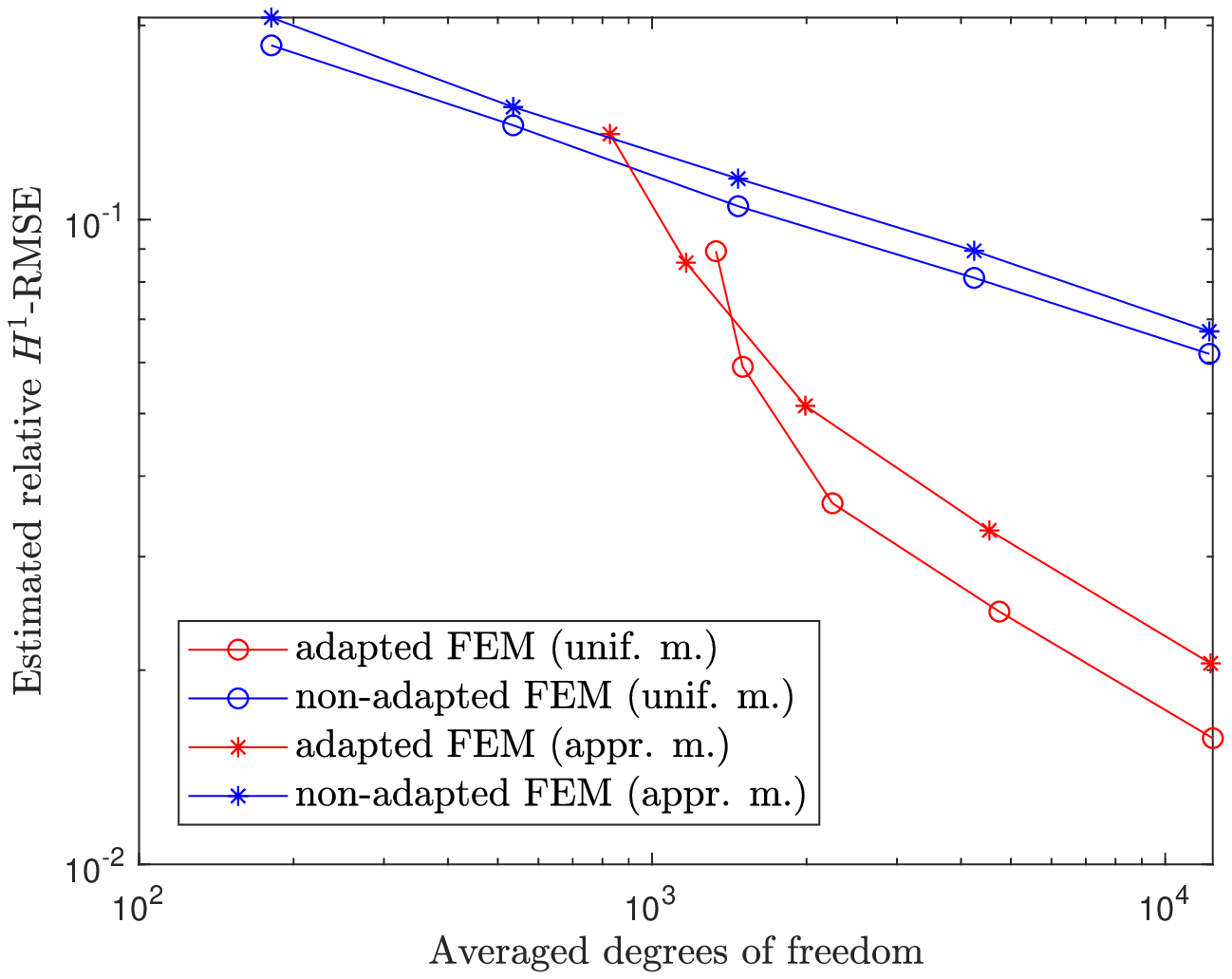}}
    \label{FIG:NumExPoiss5SmoothGRF}
    \caption{Convergence results for Poisson($5$)-subordinators using the approximation approach and the Uniform Method with mixed Dirichlet-Neumann boundary conditions.}
    \end{figure}

In Figure~\ref{FIG:NumExPoiss5SmoothGRF} we see that we get almost full order convergence for the sample-adapted FE method for both approximation approaches of the Poisson processes. Compared to the low-intensity examples with Poisson($1$)-subordinators given in Subsection \ref{subsubsec:NumexPoiss1Mixed} and \ref{subsubsec:NumExPoiss1Dir}, we get a slightly lower convergence rate of approximately $0.55$ for the standard FEM approach. This holds for both approximation methods of the Poisson subordinators. Hence, we see that the way how the Poisson-subordinators are simulated seems to have no effect on the convergence rate. 

\subsubsection{Poisson subordinators of a GRF with short correlation length: high intensity and mixed boundary conditions}	
\label{subsubsec:NumExPoiss5RoughGRF}

In our construction of the jump-diffusion coefficient, the jumps are generated by the subordinated GRF. To be precise, the number of spatial jumps is determined by the subordinators and the jump intensities (in terms of the differences in height between the jumps) are essentially determined by the GRF $W_2$. This fact allows to control the jump intensities of the diffusion coefficient by the correlation parameter of the underlying GRF $W_2$. In the following experiment we want to investigate the influence of the jump intensities of the diffusion coefficient on the convergence rates.

In Subsection \ref{subsubsec:NumExPoiss5Smoothgrf} we subordinated a Mat\'ern-1.5-GRF with pointwise standard deviation $\sigma_2= 0.3$ and a correlation length of $r_2=0.5$. In the following experiment we set the standard deviation of the GRF $W_2$ to $\sigma_2=0.5$ and the correlation length to $r_2=0.1$ and leave all the other parameters unchanged. Figure~\ref{FIG:SmoothRoughGRFComparison} compares the resulting GRF with the field $W_2$ with parameters $\sigma_2=0.3$ and $r_2=0.5$ which we used in Subsection \ref{subsubsec:NumExPoiss5Smoothgrf}. 

	\begin{figure}[ht]
	\centering
	\subfigure{\includegraphics[scale=0.14]{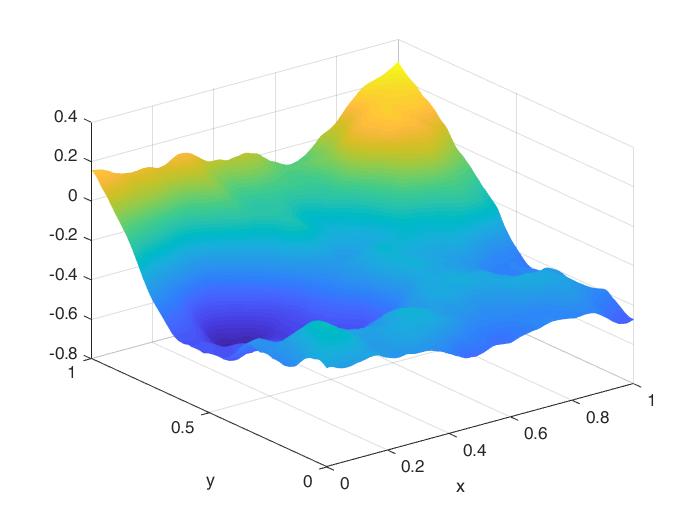}}
	\subfigure{\includegraphics[scale=0.14]{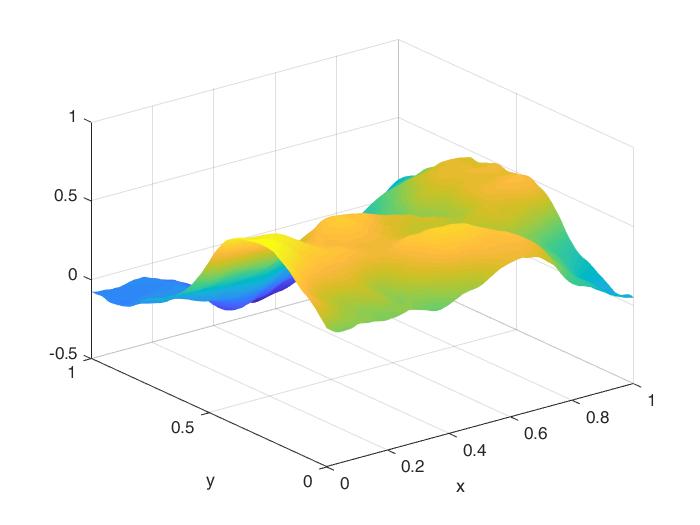}}
	\subfigure{\includegraphics[scale=0.14]{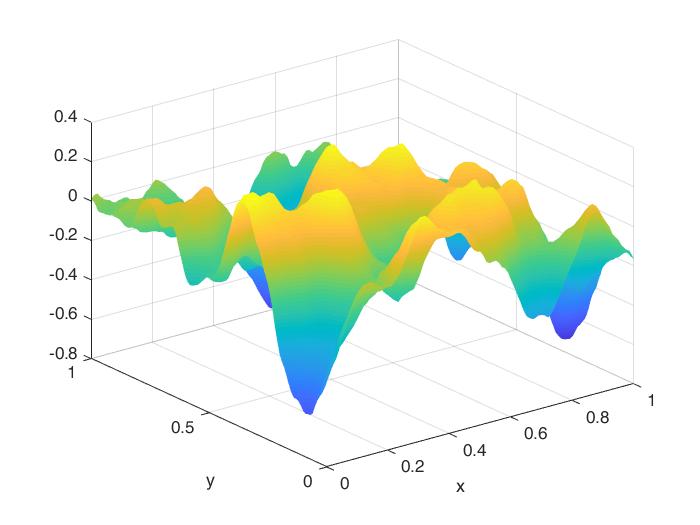}}
	\subfigure{\includegraphics[scale=0.14]{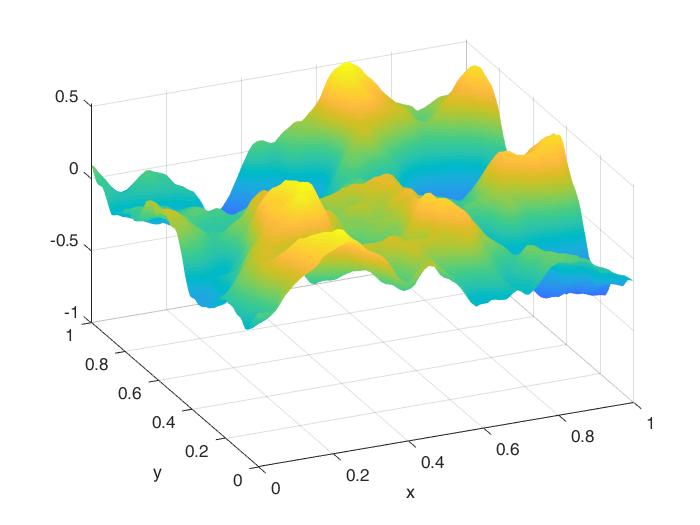}}
    \caption{Samples of a Mat\'ern-1.5-GRF with $\sigma_2=0.3$, $r_2=0.5$ (left) and with parameters $\sigma_2=0.5$, $r_2=0.1$ (right).}
    \label{FIG:SmoothRoughGRFComparison}
    \end{figure}

Subordinating the GRF with small correlation length (right plots in Figure \ref{FIG:SmoothRoughGRFComparison}) result in higher jump intensities in the diffusion coefficient as the subordination of the GRF with higher correlation length (left plots in Figure \ref{FIG:SmoothRoughGRFComparison}). Figure~\ref{fig:poisson5samples} shows samples of the diffusion coefficient and the corresponding PDE solutions where the parameters of $W_2$ are $\sigma_2=0.5$ and $r_2=0.1$. 

	\begin{figure}[ht]
	\centering
	\subfigure{\includegraphics[scale=0.14]{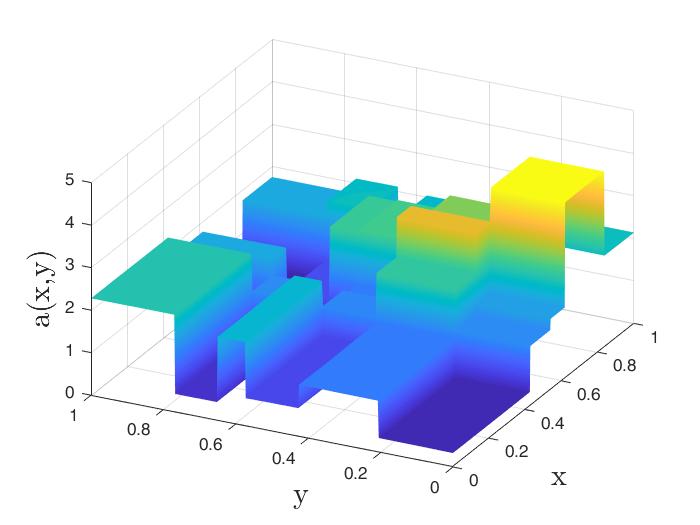}}
	\subfigure{\includegraphics[scale=0.14]{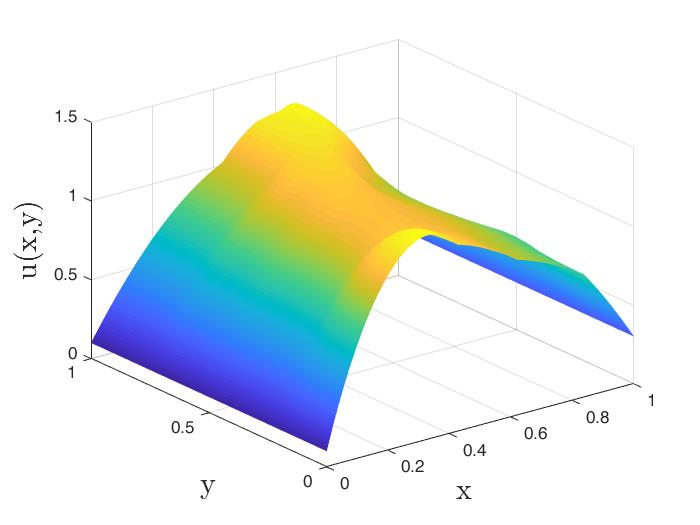}}
	\subfigure{\includegraphics[scale=0.14]{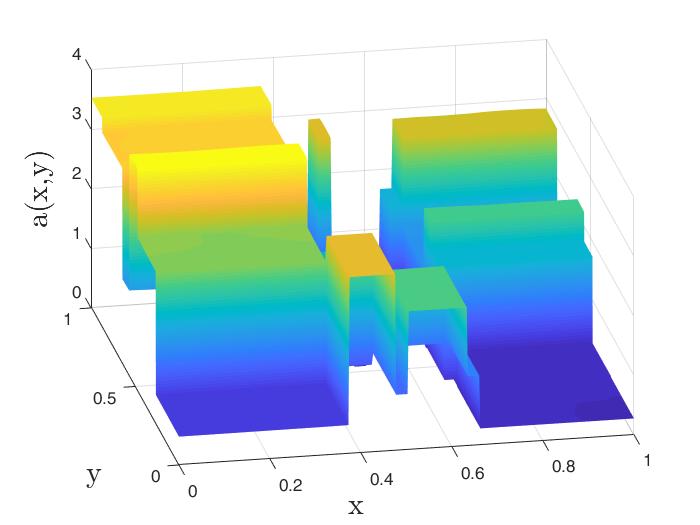}}
	\subfigure{\includegraphics[scale=0.14]{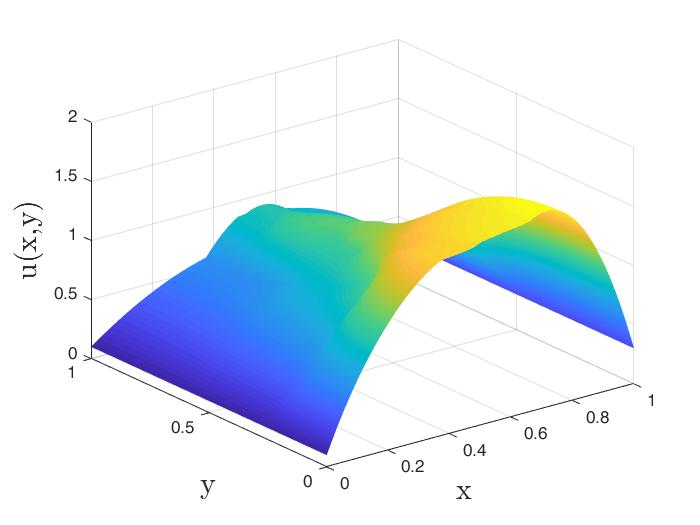}}
    \caption{Different samples of the diffusion coefficient with Poisson($5$)-subordinators and the corresponding PDE solutions with mixed Dirichlet-Neumann boundary conditions and small correlation length $r_2=0.1$ of the GRF $W_2$.}\label{fig:poisson5samples}
    \end{figure}

As expected, the resulting jump coefficient shows jumps with a higher intensity compared to the jump coefficient in the previous experiment where we used the GRF $W_2$ with parameters $\sigma_2=0.3$ and $r_2=0.5$ (see Figure \ref{FIG:SamplesPoiss5SmoothGRF}). 

We estimate the strong error convergence rate using this high-intensity jump coefficient using $M=200$ samples and approximate the Poisson subordinators by the Uniform Method.

  	\begin{figure}[ht]
	\centering
	\subfigure{\includegraphics[scale=0.49]{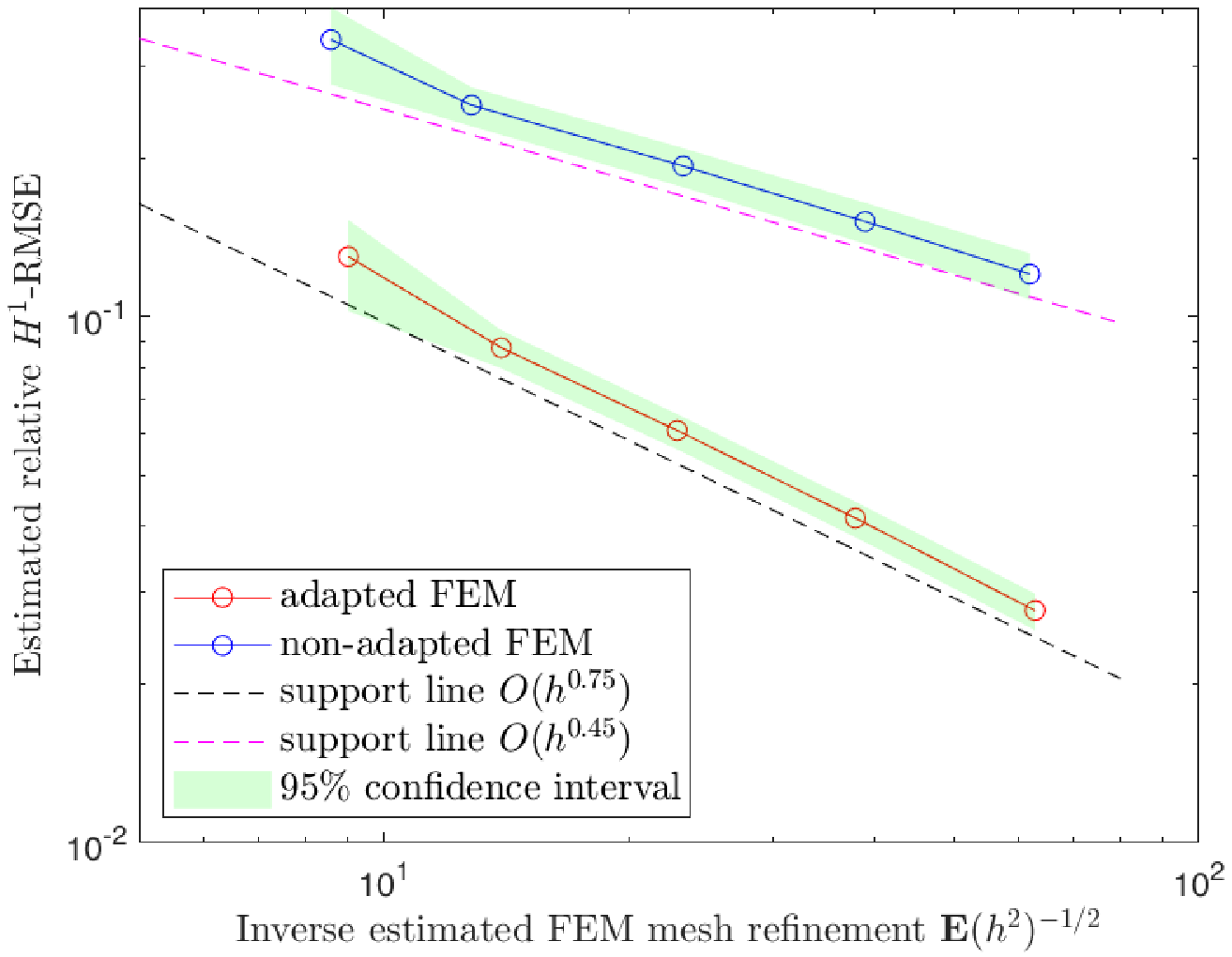}}
	\subfigure{\includegraphics[scale=0.49]{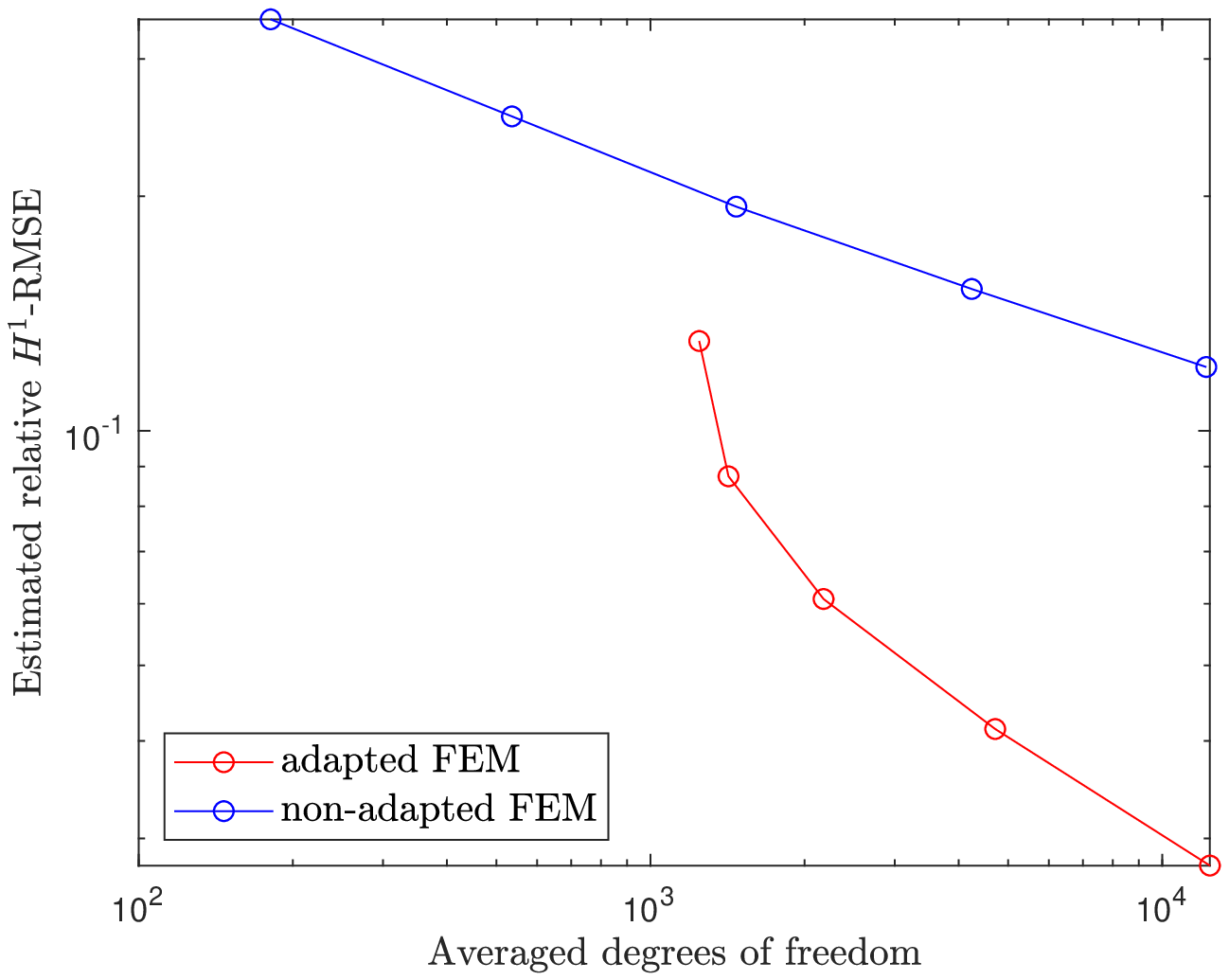}}
    \caption{Convergence results for Poisson($5$)-subordinators using the Uniform Method with mixed Dirichlet-Neumann boundary conditions and GRF parameters $\sigma_2=0.5$ and $r_2=0.1$.}
    \label{FIG:NumExPoiss5RoughGRF}
    \end{figure}

Figure~\ref{FIG:NumExPoiss5RoughGRF} shows that for the GRF $W_2$ with small correlation length the convergence rates are reduced for both approaches: the standard FEM approach and the sample-adapted version. We cannot preserve full order convergence in the sample-adapted FEM but observe a convergence rate of approximately $0.75$. In the non-adapted approach we obtain a convergence rate of approximately $0.45$.  Looking at the error-to-(averaged)DOF-plot on the right hand side of Figure \ref{FIG:NumExPoiss5RoughGRF} we see that still the sample-adapted approach is by a large margin more efficient in terms of computational effort. This experiment confirms our expectations since the FEM convergence rate has been shown to be strongly influenced by the regularity of the jump-diffusion coefficient (see e.g. \cite{AStudyOfElliptic} and  \cite{RegularityResultsForLaplaceInterfaceProblemsInTroDimensions}).

\subsection{Gamma subordinators}	
	
In order to also consider L\'evy subordinators with infinite activity we take Gamma processes to subordinate the GRF in the remaining numerical examples. We set the standard deviation of the GRF $W_1$ to be $\sigma_1=1.5$ and we choose $\sigma_2=0.3$ and $r_2=1$ for the Mat\'ern-1.5-GRF $W_2$ and leave the other parameters unchanged. 
For $a_G,b_G>0$, a $Gamma(a_G,b_G)$-distributed random variable admits the density function
\begin{align*}
 x\mapsto \frac{b_G^{a_G}}{\Gamma(a_G)}x^{a_G-1}\exp(-xb_G),~\text{ for }x>0,
\end{align*}
where $\Gamma(\cdot)$ denotes the Gamma function. A Gamma process $(X_t, t\geq 0)$ has independent Gamma distributed increments. Being precise, $X_t-X_s\sim Gamma(a_G\cdot(t-s),b_G)$ for $0<s<t$ (see~\cite[Chapter 8]{LevyProcessesInFinance}). 

The following lemma is essential to approximate the Gamma processes.
\begin{lemma}\label{LE:MomentsGamma}
Let $Z$ be a \textit{Gamma}$(a_G,b_G)$ distributed random variable for positive parameters $a_G,b_G>0$. It holds
\begin{align*}
\mathbb{E}(Z^n)=b_G^{-n} \frac{\Gamma(a_G+n)}{\Gamma(a_G)},
\end{align*}
for all $n\in \mathbb{N}_0$.
\end{lemma}
\begin{proof}
We prove the assertion by induction. The case $n=0$ is trivial. For an arbitrary natural number $n\geq 0$ we calculate
\begin{align*}
\mathbb{E}(Z^{n+1})&=\frac{b_G^{a_G}}{\Gamma(a_G)}\int_0^\infty x^{n+a_G}\exp(-xb_G)dx\\
&=-\frac{1}{b_G}\frac{b_G^{a_G}}{\Gamma(a_G)} \Big( x^{n+a_G}\exp(-xb_G)|_0^\infty - (n+a_G)\int_0^\infty x^{n+a_G-1}\exp(-xb_G)dx \Big)\\
&=\frac{n+a_G}{b_G}\mathbb{E}(Z^n) \\
&=b_G^{-(n+1)} \frac{\Gamma(a_G+n+1)}{\Gamma(a_G)},
\end{align*}
where we used the Theorem of Bohr-Mollerup (see \cite{EunfuehrungInDieTheorieDerGammafunktion}) in the last step.
\end{proof}

In our numerical experiments we choose $l_j$ to be a $Gamma(4,10)$-process for $j=1,2$. Since increments of a Gamma process are Gamma-distributed random variables it is straightforward to generate values of a Gamma process on grid points $(x_i)_{i=0}^{N_l}\subset [0,1]$ with $|x_{i+1}-x_i|\leq \varepsilon_l$ for $i=0,\dots,N_l-1$. 
We then use the piecewise constant extension of the simulated values $\{l_j(x_i),~i=0,\dots,N_l-1,~j=1,2\}$ to approximate the L\'evy subordinators: 
\begin{align*}
l_j^{(\varepsilon_l)}(x)=\begin{cases} l_j(x_i) & x\in[x_i,x_{i+1}) \text{ for } i=0,...,N_l-1,  \\ l_j(x_{N_l-1}) & x=1.  \end{cases}
\end{align*}
for $j=1,2$. Note that in this case Assumption~\ref{ASS:CutProblemEigenvalues} \textit{v} is fulfilled with for any fixed $\eta <+\infty$. To see that we consider a fixed $s\in\mathbb{N}$ with $s\leq \eta$ and calculate for an arbitrary $x\in[0,1)$ with $x\in[x_i,x_{i+1})$:
\begin{align*}
\mathbb{E}(|l_j(x)-l_j^{\varepsilon_l}(x)|^s)&\leq \mathbb{E}(|l_j(x_{i+1}) - l_j(x_i)|^s)\\
&\leq \mathbb{E}(|l_j(\varepsilon_l)|^s)\\
&=b_G^{-s}\frac{\Gamma(a_G\varepsilon_l + s)}{\Gamma(a\varepsilon_l)}\\
&=b_G^{-s}\prod_{i=1}^{s-1}(a_G\varepsilon_l + i) a\varepsilon_l\\
&\leq C_l\varepsilon_l.
\end{align*}

Figure~\ref{Fig:NumExGammaSmoothGRF} shows samples of the jump-diffusion coefficient with Gamma($4,10$)-subordinator and corresponding FE solution where we used mixed Dirichlet-Neumann boundary conditions.

	\begin{figure}[ht]
	\centering
	\subfigure{\includegraphics[scale=0.14]{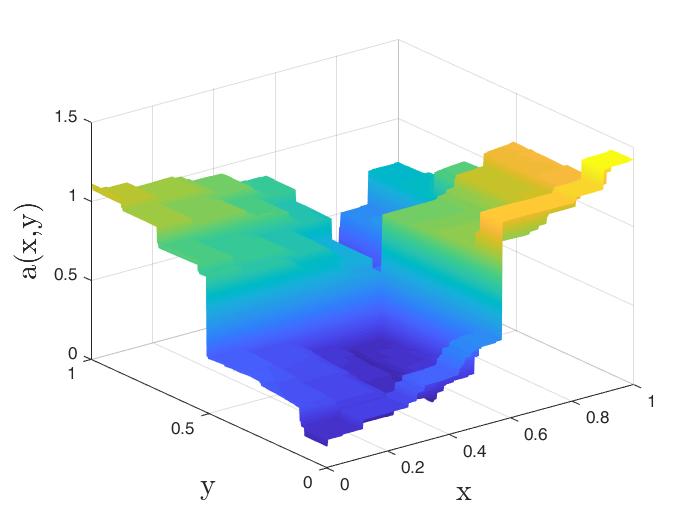}}
	\subfigure{\includegraphics[scale=0.14]{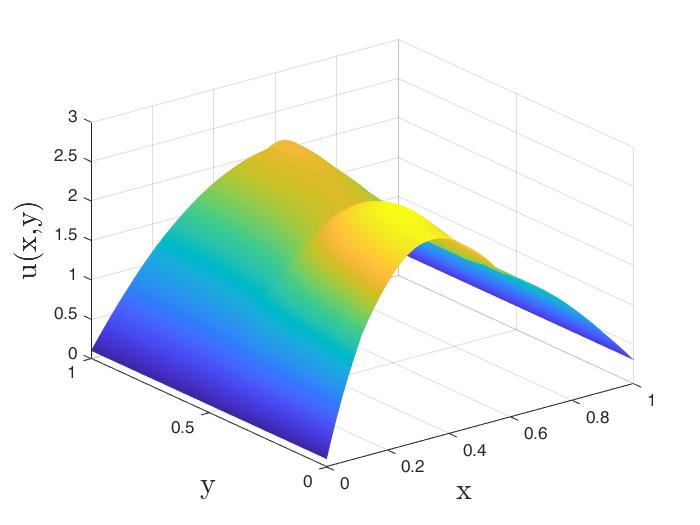}}
	\subfigure{\includegraphics[scale=0.14]{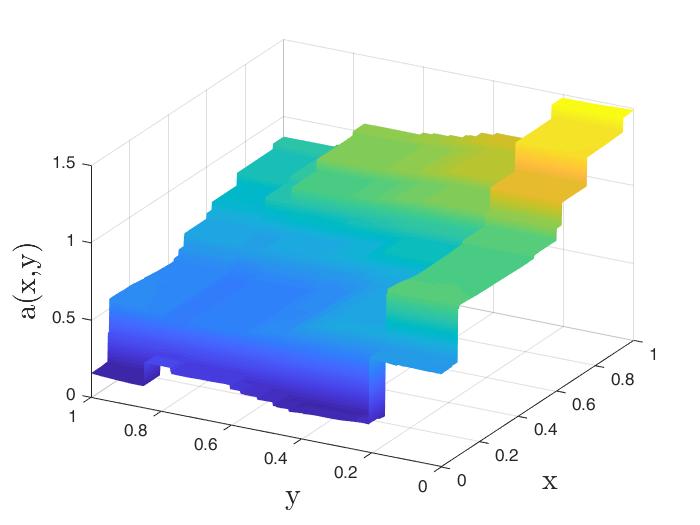}}
	\subfigure{\includegraphics[scale=0.14]{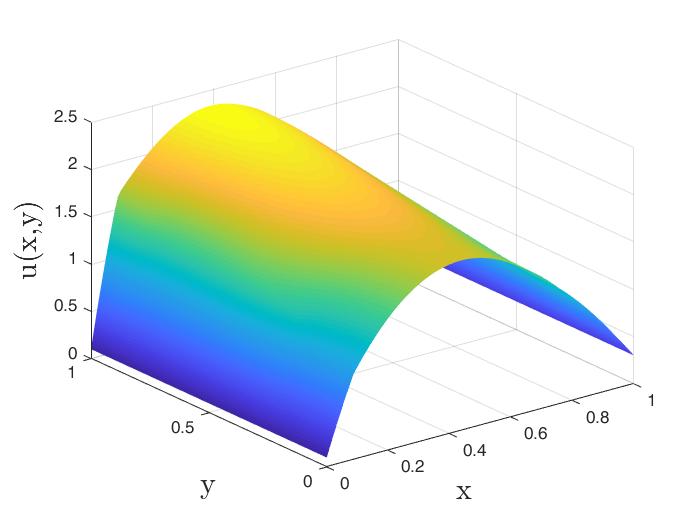}}
    \caption{Different samples of the diffusion coefficient with Gamma($4,10$)-subordinators and the corresponding PDE solutions with mixed Dirichlet-Neumann boundary conditions.}
    \label{Fig:NumExGammaSmoothGRF}
    \end{figure}

We set the diffusion cut-off to $K=2$ since in this case we obtain 
\begin{align*}
 \mathbb{P}(\underset{t\in[0,1]}{\sup} l_j(t) \geq 2) = \mathbb{P}(l_j(1)\geq 2)\approx 3.2042e^{-06},
\end{align*}
for $j=1,2$. The use of infinite-activity Gamma subordinators in the diffusion coefficient does not allow anymore for a sample-adapted approach to solve the PDE problem. Hence, we only use the standard FEM approach to solve the PDE samplewise and estimate the strong error convergence. We use $M=200$ samples to estimate the strong error on the levels $\ell = 1,\dots,5$ where we set the non-adaptive FEM solution $u_{7,\varepsilon_W,\varepsilon_l}$ on level $L=7$ to be the reference solution. We choose the FEM discretization steps to be $h_\ell = 0.4\cdot 2^{-(\ell -1)}$ for  $\ell = 1,\dots,7$.

  	\begin{figure}[ht]
	\centering
	\subfigure{\includegraphics[scale=0.5]{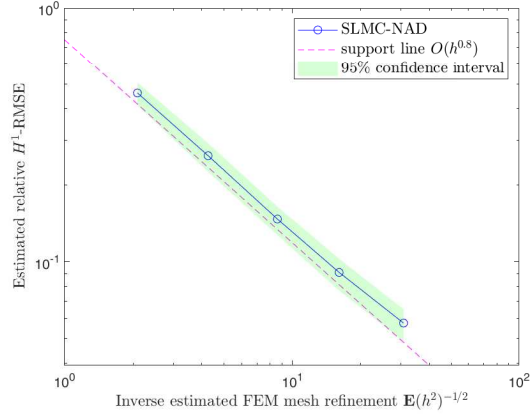}}
	\caption{Convergence results for Gamma($4,10$)-subordinators with mixed Dirichlet-Neumann boundary conditions.}\label{FIG:NumExGammaSmoothGRFConv}
    \end{figure}

Figure~\ref{FIG:NumExGammaSmoothGRFConv} shows a convergence rate of approximately $0.8$ for the standard-FEM approach. Since we do not treat the discontinuities in a special way we cannot expect full order convergence. In fact, the given convergence is comparably good since in general we cannot prove a higher convergence order than $0.5$ for the standard deterministic FEM approach without special treatment of the discontinuities (see \cite{TheFiniteElementMethodForELlipticEquationsWithDiscontinuousCoefficients} and \cite{AStudyOfElliptic}). The convergence rate of approximately $0.8$ in this example is based on the comparatively large correlation length of the underlying GRF $W_2$ (see \ref{subsubsec:NumExPoiss5RoughGRF}). 

In Subsection \ref{subsubsec:NumExPoiss5RoughGRF} we investigated the effect of a rougher diffusion coefficient on the convergence rate for Poisson($5$)-subordinators. In the following experiment we follow a similar strategy and use a shorter correlation length in the GRF $W_2$ which is  subordinated by Gamma processes. Therefore, we choose the parameters of the Mat\'ern-1.5-GRF $W_2$ to be $\sigma_2=0.3$ and $r_2=0.05$. Figure~\ref{fig:gammasamples} shows a comparison of the resulting GRFs $W_2$ with the different correlation lengths.

	\begin{figure}[ht]
	\centering
	\subfigure{\includegraphics[scale=0.13]{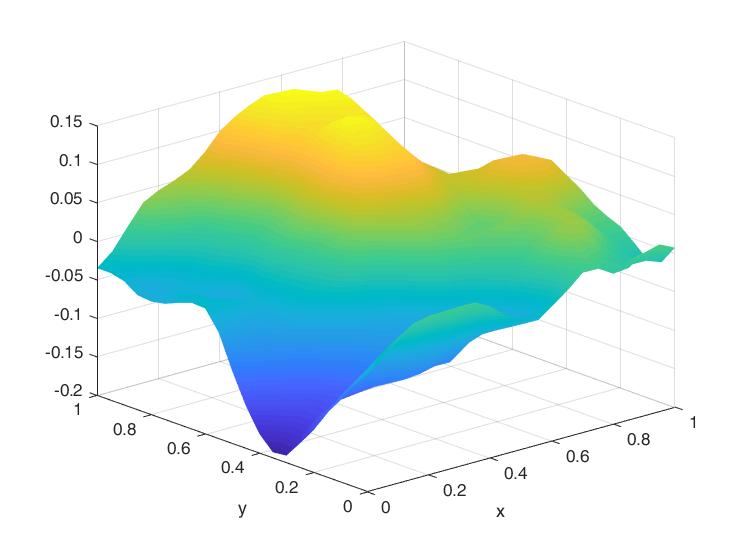}}
	\subfigure{\includegraphics[scale=0.14]{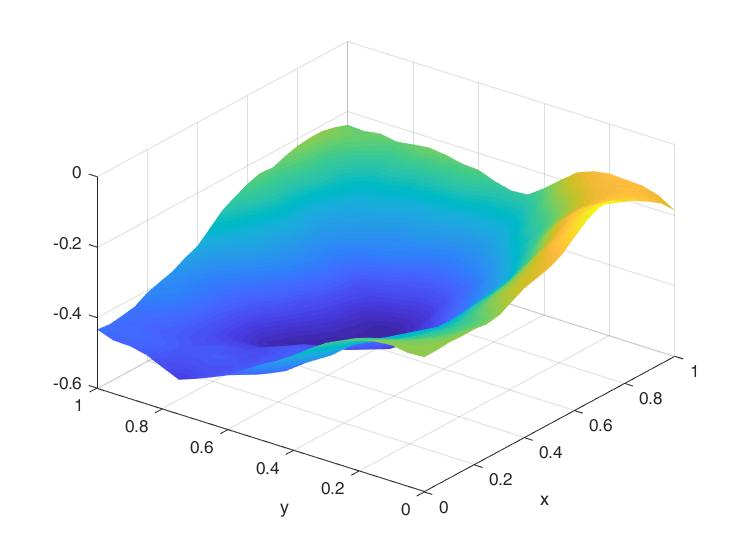}}
	\subfigure{\includegraphics[scale=0.14]{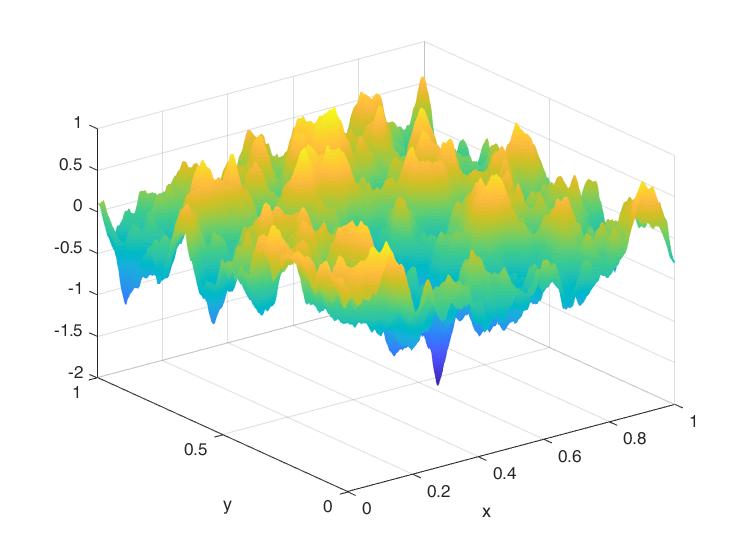}}
	\subfigure{\includegraphics[scale=0.14]{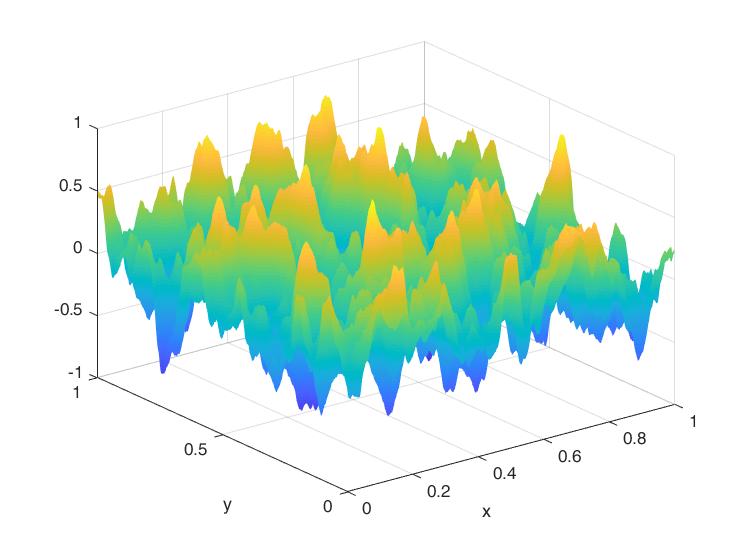}}
    \caption{Samples of a Mat\'ern-1.5-GRF with correlation length $r_2=1$ (left) and with correlation length $r_2=0.05$ (right).}\label{fig:gammasamples}
    \end{figure}

In Figure~\ref{fig:gammasolution}, the GRF with small correlation length results in higher jumps of the diffusion coefficient and stronger deformations of the corresponding PDE solution compared to the previous example (see Figure \ref{Fig:NumExGammaSmoothGRF}).

	\begin{figure}[ht]
	\centering
	\subfigure{\includegraphics[scale=0.14]{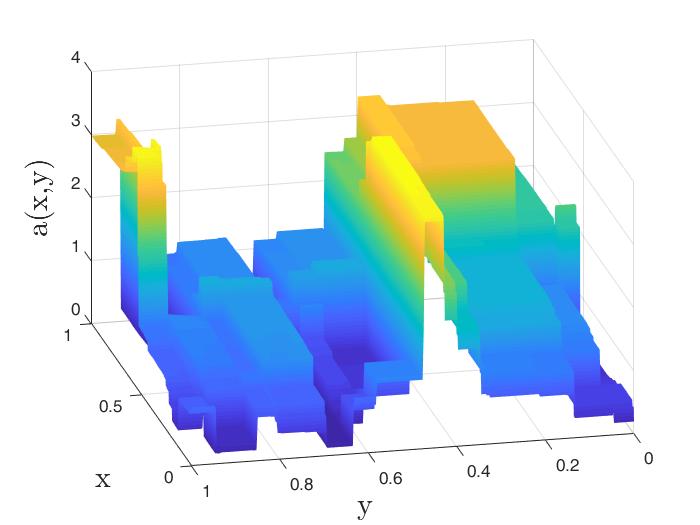}}
	\subfigure{\includegraphics[scale=0.14]{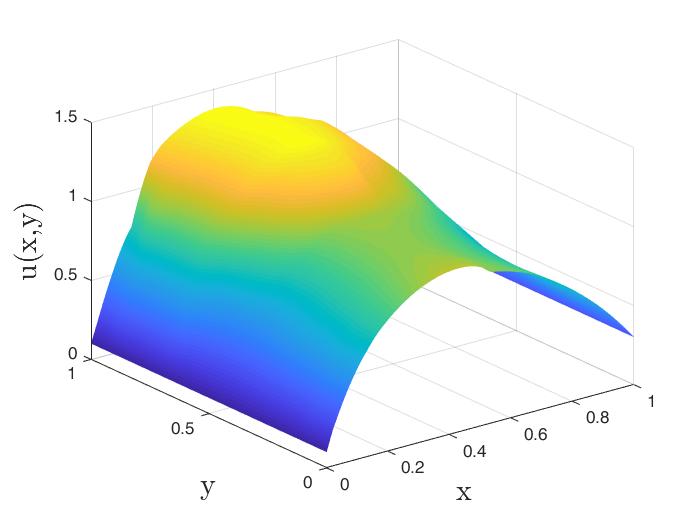}}
	\subfigure{\includegraphics[scale=0.14]{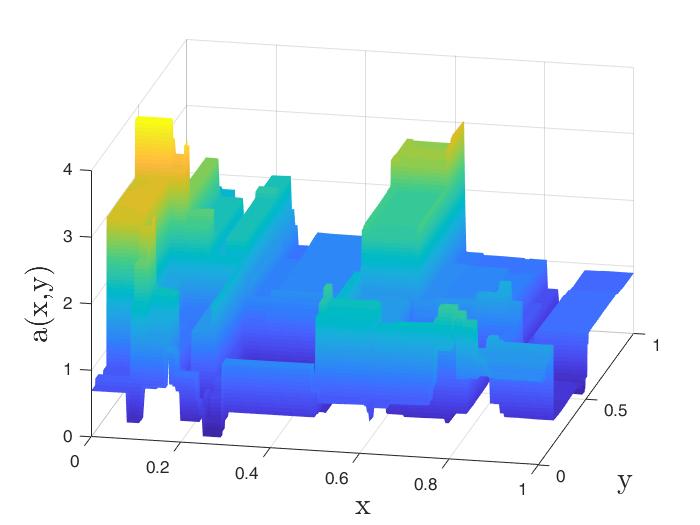}}
	\subfigure{\includegraphics[scale=0.14]{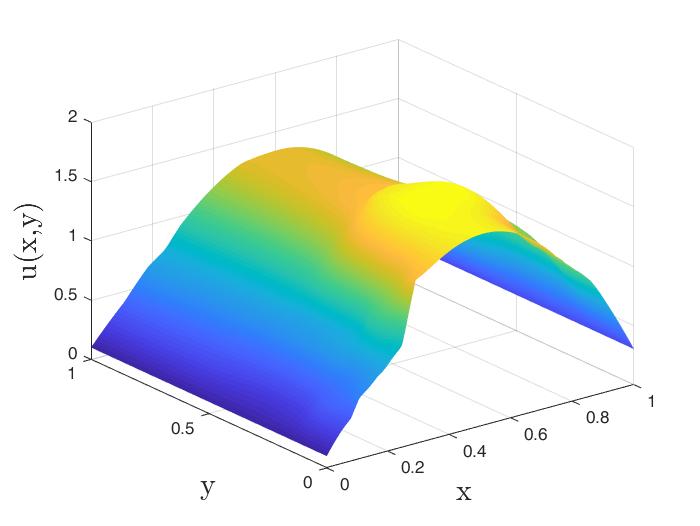}}
    \caption{Different samples of the diffusion coefficient with Gamma($4,10$)-subordinators and the corresponding PDE solutions with mixed Dirichlet-Neumann boundary conditions where the correlation length of $W_2$ is $r_2=0.05$.}\label{fig:gammasolution}
    \end{figure}
We estimate the strong error taking $M=200$ samples where we use the non-adapted FEM solution $u_{9,\varepsilon_W,\varepsilon_l}$ on level $L=9$  as reference solution and choose the FEM discretization steps to be $h_\ell = 0.1\cdot 1.5^{-(\ell -1)}$ for  $\ell = 1,\dots,9$. Figure~\ref{fig:gammaconvergence} shows the convergence on the levels $\ell=1,\dots,6$.

	\begin{figure}[ht]
	\centering
	\subfigure{\includegraphics[scale=0.5]{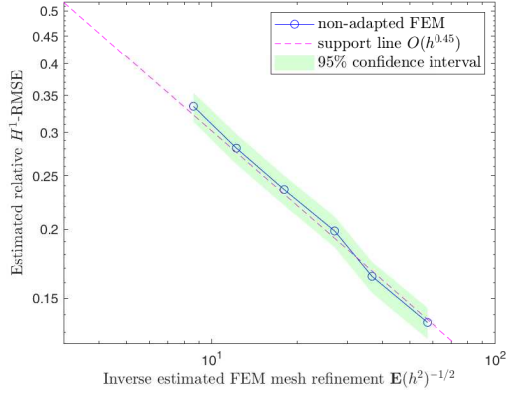}}
    \caption{Convergence results for Gamma($4,10$)-subordinators with mixed Dirichlet-Neumann boundary conditions where the correlation length of $W_2$ is $r_2=0.05$.}\label{fig:gammaconvergence}
    \end{figure}
We observe a convergence rate of approximately $0.45$ which is significantly smaller than the rate of approximately $0.8$ we obtained in the example where we used a GRF $W_2$ with correlation length $r_2=1$ (see Figure \ref{FIG:NumExGammaSmoothGRFConv}). This again confirms that, for subordinated GRFs, the convergence rate of the FE method is highly dependent on the correlation length of the underlying GRF $W_2$ and the resulting jump-intensity of the diffusion coefficient.

\appendix\section{Proof of Theorem~\ref{TH:QuantificationOfDiffApprLinftynorm}}\label{appendix:proof}

\begin{customthm}{5.3} 
\end{customthm}
	   
\begin{theorem}[Theorem 5.3]
For any $\delta>0$ and any $n\in(1,+\infty)$ there exists a constant $A=A(\delta,n)>0$ such that 
	   \begin{align*}
	   \mathbb{E}(\underset{\underline{x}\in\mathcal{D}}{\operatorname{ess}\,\sup}\,|a_K(\underline{x})-a_{K,A}(\underline{x})|^{n})^{1/n}<\delta.
	   \end{align*}
\end{theorem}

\begin{proof}\textit{Step 1: Tail estimation for the coefficient}
	   
By Assumption~\ref{ASS:CutProblemEigenvalues} \textit{ii}, the functions $\Phi_1,\Phi_2$ from Definition~\ref{DEF:DefCoeff} fulfill	   
	  
	   \begin{align*}
	   |\Phi_1'(x)|\leq \phi \exp(\psi|x|), ~|\Phi_2(x)-\Phi_2(y)|\leq C_{lip}|x-y|,
	   \end{align*}
	   for $x,y\in \mathbb{R}$. By the mean value theorem, for any $x\in\mathbb{R}$ there exists a real number $\xi\in\mathbb{R}$ with $|\xi|\leq |x|$ such that it holds
	   \begin{align}\label{EQ:EstimateForPhi1}
	   |\Phi_1(x)|&\leq C(1+|x|\Phi_1'(\xi))\leq C(1+|x|\phi \exp(\psi|\xi|)) \leq C(1+|x|\phi \exp(\psi|x|)) \leq \tilde{\phi}\exp(\tilde{\psi}|x|),
	   \end{align}
	   for positive constants $\tilde{\phi}$ and $\tilde{\psi}$ which are independent of $x\in\mathbb{R}$.
	   
	   Since the GRFs $W_1$ and $W_2$ are $\mathbb{P}-a.s.$ bounded on $\mathcal{D}$ resp. on $[0,K]^2$ it follows from \cite[Theorem 2.1.1]{RandomFieldsAndGeometry} that $\mu_1:=\mathbb{E}(\underset{(x,y)\in\mathcal{D}}{\sup}\,W_1(x,y))<+\infty$ and
	   \begin{align}\label{EQ:TailEstGRF1}
	   \mathbb{P}(\|W_1\|_{L^\infty	(\mathcal{D})}>m) \leq 2\,\mathbb{P}(\underset{(x,y)\in\mathcal{D}}{\sup}\,W_1(x,y)>m) \leq 2\exp\Big(-\frac{(m-\mu_1)^2}{2\sigma_\mathcal{D}^2}\Big),
	   \end{align}
	   for $m>\mu_1$ with a finite constant $\sigma_\mathcal{D}^2$ defined by
	   \begin{align*}
	   \sigma_\mathcal{D}^2 := \underset{(x,y)\in\mathcal{D}}{\sup}\mathbb{E}(W_1(x,y)^2)=\sum_{i=1}^\infty \lambda_i^{(1)} e_i^{(1)}(x,y)^2\leq C_e^2\sum_{i=1}^\infty \lambda_i^{(1)} <+\infty,
	   \end{align*}
	   by Assumption~\ref{ASS:CutProblemEigenvalues} \textit{i}.
	   For a given $\varepsilon \in(0,1)$ we choose the real number $A$ such that it holds
	   \begin{align}\label{EQ:ConditionOnCut1}
	   A> 3\,\tilde{\phi}\,\exp\Big(\tilde{\psi}(\sqrt{2\sigma_\mathcal{D}^2 |\ln(\varepsilon/2)|} + \mu_1)\Big).
	   \end{align}
	   With this choice we obtain the bound
	   \begin{align}\label{EQ:TailEstCoeffPart1}
	   \mathbb{P}(\underset{(x,y)\in\mathcal{D}}{\sup}\, \Phi_1(W_1(x,y))>A/3)\leq \varepsilon.
	   \end{align}
	   This can be seen by the following calculation
	   \begin{align*}
	   \mathbb{P}(\underset{(x,y)\in\mathcal{D}}{\sup}\, \Phi_1(W_1(x,y))>A/3)&\leq \mathbb{P}(\underset{(x,y)\in\mathcal{D}}{\sup}\,\tilde{\phi}\,\exp(\tilde{\psi}|W_1(x,y)|)>A/3)\\
	   &=\mathbb{P}(\|W_1\|_{L^\infty(\mathcal{D})} > 1/\tilde{\psi}\, \ln(A/(3\tilde{\phi})))\\
	   &\leq 2\,\exp\Big(-\frac{(1/\tilde{\psi}\,\ln(A/(3\tilde{\phi}))-\mu_1)^2}{2\sigma_\mathcal{D}^2}\Big)\\
	   &\leq \varepsilon,
	   \end{align*}
	   where we used \eqref{EQ:EstimateForPhi1} in the first step, the estimate~\eqref{EQ:TailEstGRF1} in the third step and condition~\eqref{EQ:ConditionOnCut1} in the last step.
	   
	   Obviously, an estimation as in Equation~\eqref{EQ:TailEstGRF1} holds for the GRF $W_2$:
	   \begin{align}\label{EQ:TailEstGRF2}
	   	   \mathbb{P}(\|W_2\|_{L^\infty	([0,K]^2)}>m) \leq 2\,\mathbb{P}(\underset{(x,y)\in[0,K]^2}{\sup}\,W_2(x,y)>m) \leq 2\exp\Big(-\frac{(m-\mu_2)^2}{2\sigma_K^2}\Big)
	   \end{align}
	   for $m>\mu_2$ with $\mu_2:=\mathbb{E}(\underset{(x,y)\in[0,K]^2}{\sup}\,W_2(x,y))<+\infty$ and 
	   	   \begin{align*}
	   \sigma_K^2 := \underset{(x,y)\in[0,K]^2}{\sup}\mathbb{E}(W_2(x,y)^2)=\sum_{i=1}^\infty \lambda_i^{(2)} e_i^{(2)}(x,y)^2\leq C_e^2 \sum_{i=1}^\infty \lambda_i^{(2)} <+\infty.
	   \end{align*}
	   By the Lipschitz continuity of $\Phi_2$ we conclude the existence of a constant $\phi_2>0$ such that 
	   \begin{align}\label{EQ:EstimatePhi2}
	   |\Phi_2(x)|\leq \phi_2(1+|x|),
	   \end{align}
	   for $x\in\mathbb{R}$. If we again fix a positive $\varepsilon\in(0,1)$ and choose the real number $A$ such that
	   \begin{align*}
	   A>3\phi_2\Big(\sqrt{2\sigma_K^2|\ln(\varepsilon/2)|} + \mu_2 + 1),
	   \end{align*}
	   we obtain the following bound
	   \begin{align}\label{EQ:TailEstCoeffPart2}
	   \mathbb{P}(\underset{(x,y)\in[0,K]^2}{\sup}\, \Phi_2(W_2(x,y))>A/3)\leq \varepsilon.
	   \end{align}
	   This can be seen by the following calculation:
	   \begin{align*}
	   \mathbb{P}(\underset{(x,y)\in[0,K]^2}{\sup}\, \Phi_2(W_2(x,y))>A/3)&\leq\mathbb{P}(\phi_2(1+\|W_2\|_{L^\infty([0,K]^2)})>A/3)\\
	   &\leq   \mathbb{P}(\|W_2\|_{L^\infty([0,K]^2)}>A/(3\phi_2) - 1)\\
	   &\leq 2\,\exp\Big( -\frac{(A/(3\phi_2) - 1 - \mu_2)^2}{2\sigma_K^2} \Big)\\
	   &\leq \varepsilon.
	   \end{align*}
	   
\textit{Step 2: Finite moments of the coefficient} 

In this step we want to show that for any $n\in[1,+\infty)$ it holds
\begin{align}\label{EQ:FiniteMomentsOfCoeffInLinftyNorm}
\mathbb{E}(\underset{\underline{x}\in\mathcal{D}}{\operatorname{ess}\,\sup}\,|a_K(\underline{x})|^{n})=:C_{a_K}(n,K,D)<+\infty.
\end{align}
We use the definition of the coefficient $a_K$ in \eqref{EQ:DiffCoeffDefiCut} and H\"older's inequality to calculate
\begin{align*}
\mathbb{E}(\underset{\underline{x}\in\mathcal{D}}{\operatorname{ess}\,\sup}\,|a_K(\underline{x})|^{n})&\leq \mathbb{E}(| \overline{a}_+ + \underset{(x,y)\in\mathcal{D}}{\sup}\,\Phi_1(W_1(x,y)) + \underset{(x,y)\in[0,K]^2}{\sup}\,\Phi_2(W_2(x,y))|^n)\\
&\leq 3^{\frac{n-1}{n}} \Big( \overline{a}_+^n + \mathbb{E}(\underset{(x,y)\in\mathcal{D}}{\sup}\,\Phi_1(W_1(x,y))^n) + \mathbb{E} (\underset{(x,y)\in[0,K]^2}{\sup}\,\Phi_2(W_2(x,y))^n) \Big)\\
&=3^{\frac{n-1}{n}}(\overline{a}_+^n + I_1 + I_2).
\end{align*}
Therefore, it remains to show that it holds $I_1,I_2<+\infty$.

By Fubini's theorem, for every nonnegative random variable $X$ it holds
\begin{align*}
\mathbb{E}(X)=\int_\Omega X \,d\mathbb{P}=\int_\Omega \int_0^\infty \mathds{1}_{\{X\geq c\}}dc\,d\mathbb{P}=\int_0^\infty \mathbb{P}(X\geq c)dc
\end{align*}
if the right hand side exists. We use this fact and Equation~\eqref{EQ:EstimateForPhi1} to estimate for $I_1$:
\begin{align*}
I_1&\leq \tilde{\phi}^n\mathbb{E}(\exp(\tilde{\psi}n\|W_1\|_{L^\infty(\mathcal{D})})\\
&=\tilde{\phi}^n\int_0^\infty \mathbb{P}(\exp(\tilde{\psi}n\|W_1\|_{L^\infty(\mathcal{D})})>c)dc\\
&=\tilde{\phi}^n\int_0^\infty \mathbb{P}(\|W_1\|_{L^\infty(\mathcal{D})}>\ln(c)/(\tilde{\psi}n))dc\\
&=\tilde{\phi}^n \tilde{\psi}n\int_{-\infty}^{+\infty} \exp(\tilde{\psi}nc) \mathbb{P}(\|W_1\|_{L^\infty(\mathcal{D})}>c)dc\\
&\leq \tilde{\phi}^n \tilde{\psi}n\Big(\frac{1}{\tilde{\psi}n} + \mu_1\exp(\tilde{\psi}n\mu_1) + \int_{\mu_1}^\infty 2 \exp(\tilde{\psi}nc - \frac{(c-\mu_1)^2}{2\sigma_\mathcal{D}^2})dc\Big)<+\infty,
\end{align*}
where we split the integral and used Equation~\eqref{EQ:TailEstGRF1} in the last step. In a similar way, we use Equation~\eqref{EQ:EstimatePhi2} to calculate for the second summand $I_2$:
\begin{align*}
I_2&\leq \phi_2^n\mathbb{E}((1+\|W_2\|_{L^\infty([0,K]^2)})^n)\\
&=\phi_2^n\int_0^\infty\mathbb{P}((1+\|W_2\|_{L^\infty([0,K]^2)})^n>c)dc\\
&=\phi_2^n\int_0^\infty \mathbb{P}(\|W_2\|_{L^\infty([0,K]^2)}>c^\frac{1}{n}-1)dc\\
&=\phi_2^n n \int_{-1}^\infty (c+1)^{n-1} \mathbb{P}(\|W_2\|_{L^\infty([0,K]^2)}>c)dc\\
&\leq \phi_2^nn\Big( (\mu_2+1)^n + 2\int_{\mu_2}^\infty(c+1)^{n-1} \exp(-\frac{(c-\mu_2)^2}{2\sigma_K^2}dc \Big)<+\infty,
\end{align*}
where we used Equation~\eqref{EQ:TailEstGRF2} in the last step. This proves Equation~\eqref{EQ:FiniteMomentsOfCoeffInLinftyNorm}.

\textit{Step 3: Estimate for the approximation of the diffusion coefficient.}
	   
Now, let $\delta\in(0,1)$ be arbitrary. 	   
Choose $A=A(\delta)>0$ such that 
	   \begin{align*}
	   A>\max\Big\{ 3\overline{a}_+, 3\,\tilde{\phi}\,\exp\Big(\tilde{\psi}(\sqrt{2\sigma_\mathcal{D}^2 |\ln(\varepsilon/2)|} + \mu_1)\Big), 3\phi_2\Big(\sqrt{2\sigma_K^2|\ln(\varepsilon/2)|} + \mu_2 + 1\Big) \Big\}
	   \end{align*}
	   for $\varepsilon:=1-\sqrt{1-(\delta^{2s} /C_{a_K}(2s,K,D))}$.
	   
	   We estimate using H\"older's inequality:
	   \begin{align*}
	   \mathbb{E}(\underset{\underline{x}\in\mathcal{D}}{\operatorname{ess}\,\sup}\,|a_K(\underline{x})-a_{K,A}(\underline{x})|^{s})&\leq \mathbb{E}(\|a_K\|_{L^\infty(\mathcal{D})}^s\mathds{1}_{\{\|a_K\|_{L^\infty(\mathcal{D}))}\geq A\}})\\
	   &\leq \mathbb{E}(\|a_K\|_{L^\infty(\mathcal{D})}^{2s})^\frac{1}{2}\mathbb{P}(\underset{\underline{x}\in\mathcal{D}}{\operatorname{ess}\,\sup}\,|a_K(\underline{x})|\geq A)^\frac{1}{2}\\
	   &=C_{a_K}(2s,K,D)^\frac{1}{2}\mathbb{P}(\underset{\underline{x}\in\mathcal{D}}{\operatorname{ess}\,\sup}\,|a_K(\underline{x})|\geq A)^\frac{1}{2}
	   \end{align*}
	   For the second factor we estimate using the independence of $W_1$ and $W_2$
	   \begin{align*}
	   \mathbb{P}(\underset{\underline{x}\in\mathcal{D}}{\operatorname{ess}\,\sup}\,|a_K(\underline{x})|\geq A)&=1-\mathbb{P}(\underset{\underline{x}\in\mathcal{D}}{\operatorname{ess}\,\sup}\,|a_K(\underline{x})|\leq A)\\
	   &\leq 1-\mathbb{P}(\|\Phi_1(W_1)\|_{L^\infty(\mathcal{D})}\leq A/3)\cdot\mathbb{P}(\|\Phi_2(W_2)\|_{L^\infty([0,K]^2)}\leq A/3)\\
	   &\leq 1-(1-\mathbb{P}(\|\Phi_1(W_1)\|_{L^\infty(\mathcal{D})}\geq A/3))\cdot(1-\mathbb{P}(\|\Phi_2(W_2)\|_{L^\infty([0,K]^2)}\geq A/3))\\
	   &\leq 1-(1-\varepsilon)^2\\
	   &\leq  \delta^{2s} /C_{a_K}(2s,K,D),
	   \end{align*}
	   where we used Equations~\eqref{EQ:TailEstCoeffPart1} and~\eqref{EQ:TailEstCoeffPart2} in the fourth step and therefore we obtain
	   \begin{align*}
	   \mathbb{E}(\underset{\underline{x}\in\mathcal{D}}{\operatorname{ess}\,\sup}\,|a_K(\underline{x})-a_{K,A}(\underline{x})|^{s})^{1/s}\leq \delta.
	   \end{align*}
	   \end{proof}

	\section*{Acknowledgments}
	Funded by Deutsche Forschungsgemeinschaft (DFG, German Research Foundation) under Germany's Excellence Strategy - EXC 2075 -     390740016.

	\bibliographystyle{siam}
	\bibliography{references}
	
\end{document}